\documentclass[11pt,oneside,reqno]{amsart}
\usepackage{amssymb,amsmath,mathrsfs,amsthm}
\usepackage{caption}
\usepackage{enumitem}
\usepackage{mathtools}
\usepackage{geometry}
\usepackage{amsfonts}

\usepackage{esint,comment}

\usepackage[dvips]{graphicx}
\usepackage{color}
\allowdisplaybreaks
\newtheorem{thm}{Theorem}[section]
\newtheorem{cor}[thm]{Corollary}
\newtheorem{corollary}[thm]{Corollary} 

 \newtheorem{lemma}[thm]{Lemma}
\newtheorem{prop}[thm]{Proposition}

\newtheorem{innercustomthm}{Assumption}
\newenvironment{customthm}[1]
  {\renewcommand\theinnercustomthm{#1}\innercustomthm}
  {\endinnercustomthm}

\theoremstyle{definition}

\newtheorem{rem}[thm]{Remark}

\def\colr{\color{red}}
\def\colu{\color{blue}}
\def\colb{\color{black}}

\def \no#1#2#3 {{\bf #1} (#3), #2.}
\def \eds#1#2#3 {#1, #2, #3.}

\def\R{{\mathbb R}}

\def\d{{\rm d}}

\def\T{{\mathbb T}}
\def\N{{\mathbb N}}
\def\LL{{\mathcal{L}}}

\def\:{{\colon}}

\def\be#1{\begin{equation}\label{#1}}
\def\ee{\end{equation}}

\def\<{\langle}
\def\>{\rangle}
\def\coloneqq{:=}

\renewcommand{\hat}[1]{\widehat{#1}}
\renewcommand{\bar}[1]{\overline{#1}}
\renewcommand{\tilde}[1]{\widetilde{#1}}

\newcommand{\na}{\nabla}
\newcommand{\bb}{b}
\newcommand{\ou}{{\overline{u}}}
\newcommand{\ow}{{\overline{\omega}}}

\newcommand{\lec}{\lesssim}
\newcommand{\bs}{\begin{split}}
\newcommand{\essss}{\end{split}}

\renewcommand{\div}{\operatorname{div}}
 
\newcommand{\eqnb}{\begin{equation}}
\newcommand{\eqne}{\end{equation}}

\renewcommand{\ee}{\mathrm{e}}

\newcommand{\p}{\partial}

\newcommand{\re}{\mathrm{Re}}
\newcommand{\im}{\mathrm{Im}}

\newcommand{\out}{\mathrm{out}}

\let\Re\relax
\DeclareMathOperator{\Re}{Re}
\let\Im\relax
\DeclareMathOperator{\Im}{Im}

\newcommand{\wo}{{\widetilde{\omega}}}
\newcommand{\ho}{{\widehat{\omega}}}

\newcommand{\oo}{{\overline{\omega}}}

\renewcommand{\T}{\mathbb{T}}
\renewcommand{\R}{\mathbb{R}}

\newcommand{\Aa}{\widetilde{\mathsf{A}}}
\newcommand{\Bb}{\widetilde{\mathsf{B}}}
\newcommand{\Dd}{\widetilde{\mathsf{C}}}

\newcommand{\C}{\mathbb{C}}

\newcommand{\Z}{\mathbb{Z}}

\newcommand{\app}{\mathrm{app}}

\newcommand{\gec}{\gtrsim}

\newcommand{\wu}{{\widetilde{u }}}
\newcommand{\wL}{{\widetilde{L }}}
\renewcommand{\d}{\mathrm{d}}

\newcommand{\supp}{\operatorname{supp}}

\numberwithin{equation}{section}

\begin{document}

\title[Instability of vortex columns]{Linear and nonlinear instability of vortex columns}

\author[D. Albritton]{Dallas Albritton} 
\address[D. Albritton]{Department of Mathematics, University of Wisconsin-Madison,   Madison, WI 53706, USA}
\email[]{dalbritton@wisc.edu}

\author[W. S. O\.za\'nski]{Wojciech S. O\.za\'nski} 
\address[W. S. O\.za\'nski]{ Department of Mathematics, Florida State University, Tallahassee, FL 32306, USA  \newline and
Department of Mathematics, Princeton University,  Princeton, NJ 08544, USA }
\email[]{wozanski@fsu.edu}






\begin{abstract}
We consider vortex column solutions $v = V(r) e_\theta + W(r) e_z$ to the $3$D Euler equations. 
We give a mathematically rigorous construction of the countable family of unstable modes discovered by Liebovich and Stewartson (\emph{J. Fluid Mech.} 126, 1983) via formal asymptotic analysis. The unstable modes exhibit $O(1)$ growth rates and concentrate on a ring $r= r_0$ asymptotically as the azimuthal and axial wavenumbers $n, \alpha \to \infty$ with a fixed ratio. We construct these so-called ring modes with an inner-outer gluing procedure. 
Finally, we prove that each linear instability implies nonlinear instability for vortex columns. In particular, our analysis yields nonlinear instability for the Batchelor trailing line vortex $V(r) \coloneqq \frac{q}{r} (1-\ee^{-r^2})$ and $W(r) \coloneqq \ee^{-r^2}$ when $0<q <\log 2 / \sqrt{1-\log 2} \approx 1.251$.
\end{abstract}

\maketitle

\date\today


\setcounter{tocdepth}{2}
\tableofcontents


%

\section{Introduction}\label{sec_intro}


Vortex columns are steady solutions to the $3$D Euler equations
\eqnb\label{3d_euler}
\begin{split}
\p_t v + (v\cdot \nabla ) v + \nabla p &=0, \\
\div \, v &=0
\end{split}
\eqne
having the form
\eqnb\label{vortex_col}
v = V(r) e_\theta  + W(r) e_z \, ,
\eqne
where $e_r$, $e_\theta$, $e_z$ denote the standard cylindrical basis vectors, $V(r)$ is the angular profile, and $W(r)$ is the axial profile. 
Notably, the vortex columns we consider may have non-zero axial velocity. As a prototypical example, we keep in mind the \emph{Batchelor vortex}
\eqnb\label{trailing_vortex}
V(r) \coloneqq \frac{q}{r} (1-\ee^{-r^2}),\qquad W(r) \coloneqq \ee^{-r^2} \, ,
\eqne
where $q \geq 0$ is a swirl parameter. This vortex (more specifically, its viscous analogue in which the core diffuses downstream) was proposed by Batchelor~\cite{Batchelor64} as a downstream model for the trailing vortex behind a wing. It is sometimes known as the \emph{trailing line vortex}, \emph{trailing vortex} or \emph{$q$-vortex}.

Vortices are among the simplest possible flows, and it is a classical goal of hydrodynamic stability theory to understand under what conditions these flows are stable. For vortices without axial flow, the most well known conditions are due to Rayleigh~\cite{rayleigh1879stability,rayleigh_1916} (see also Synge~\cite{synge1933stability}); with axial flow, analogous conditions were established by Howard and Gupta~\cite{hg}. The stability problem for vortex columns is further motivated by \emph{vortex breakdown}, a nonlinear phenomenon historically characterized by the development of a stagnation point on the axis and a region of reversed flow as the vortex transitions to a more complex unsteady flow downstream, see~\cite{Leibovich1978,Leibovich1984,LuccaNegro2001} and~\cite[p. 74]{van1982album}. It is natural to expect that instabilities of the underlying vortex play a role in triggering its breakdown.


Following numerical work~\cite{Lessen_Singh_Paillet_1974,Duck1980}, Liebovich and Stewartson~\cite{ls} discovered a countable family of unstable modes via asymptotic analysis as the axial and (negative) azimuthal wavenumbers $\alpha, n \to \infty$ in tandem with fixed ratio $\beta := \alpha/n$. The eigenfunctions are sometimes called ``ring modes'' because they asymptotically concentrate in a neighbourhood of a ring $r=r_0$. The countable family is indexed by a parameter $m \in \N$, with $m=1$ corresponding to a most unstable ``ground state" and $m \geq 2$ to less unstable ``excited states", all of whose growth rates are asymptotically $O(1)$:
\begin{equation}
	\label{eq:twistedeval}
    \omega = i\lambda = -n C_1+i C_2 + ( 1-i )n^{-1/2} (2m-1) C_3 + o(n^{-1/2}) \, ,
\end{equation}
where $C_1,C_2,C_3>0$ (see \eqref{choice_of_tilde_om_m_copy} for precise values). The ring modes are numerically reported to be the most unstable modes~\cite{ls}, at least in certain regimes. They break axisymmetry but retain helical symmetry. The asymptotic analysis of~\cite{ls} and its viscous counterpart~\cite{Stewartson1982} were followed by many works on the stability of vortices, especially the Batchelor vortex, and its relationship to vortex breakdown, see, for example,~\cite{StewartsonBrownNearNeutral1985,Duck1986,ls_87,cs,Khorrami_1991,Mayer_Powell_1992,DuckKhorrami,Abid1998} and, more recently,~\cite{OlendraruSellier,GALLAIRE2003,FABRE2004,HEATON2007,HEATON20072,AbidkNonlinearModeSelection,MAO2011,Billant_Gallaire_2013}. \\



In this paper, we revisit the stability of vortex columns from the perspective of rigorous partial differential equations (PDEs). Our first contribution is to give a rigorous mathematical construction of the ring modes discovered by Liebovich and Stewartson. This turns out to be a challenging perturbation problem, which we tackle using an inner-outer gluing method; see Theorem~\ref{thm_main}. Our second contribution is to construct Euler solutions which exhibit this instability at the nonlinear level; see Theorem~\ref{thm:nonlinearinstab}. Previous PDE works focused on the (in)stability of two-dimensional vortices and columnar vortices without axial flow, see, for example,~\cite{gs_spectral,gs_linear,bedrossian2019vortex,vishik_2,lin_siam,bardosguostrauss,friedlandervishik}. In contrast, this paper gives the first PDE result focusing on the important setting of vortices with axial flow, which contains genuinely new phenomena. We focus on the ring modes for their historical importance and because they are, reportedly, the most unstable. We anticipate that variants of our methods will work in many settings containing asymptotically localized instabilities. \\


Finally, while the ring modes are reportedly the most unstable, they are well localized away from the symmetry axis. It was proposed, for example, in~\cite{HEATONPeake2007,AbidkNonlinearModeSelection} that algebraic growth on the symmetry axis associated to the inviscid continuous spectrum plays a more important role in triggering vortex breakdown. It will be interesting to investigate these claims from the PDE perspective.

\subsection{The Rayleigh equation for vortex columns}

We begin by deriving the analogue of the Rayleigh equation for vortex columns, as in, e.g.,~\cite{hg}.\footnote{For two-dimensional shear flows $u = (U(y),0)$ with $x$-wavenumber $m \in \R$, eigenvalue $\lambda = -imc$, and $\hat{\psi}_m(y) = \phi(y)$, the Rayleigh equation (see, for example, Drazin and Reid~\cite{DrazinReid}) is
\begin{equation}
    \label{eq:shearRayleigh}
    (U(y) - c) (\p_y^2 - m^2) \phi - U''(y) \phi = 0 \, .
\end{equation} }
 It will be convenient to define the \emph{angular velocity} $\Omega$ and \emph{circulation} $\Gamma$:
\begin{equation}
    \Omega(r) \coloneqq V(r)/r \, , \quad \Gamma(r) \coloneqq r V(r) \, .
\end{equation}
We consider the eigenvalue problem for the linearized Euler equations around the vortex column $V(r) e_r + W(r) e_z$ in cylindrical variables, 
\begin{equation}
    \label{eq:linearizedevaleqn}
   \begin{split}
\lambda u_r   +\left( \frac{V}r \p_\theta u_r + W \p_z u_r \right) -2 \frac{Vu_\theta }r + \p_r p&=0,\\
\lambda  u_\theta +\left(  u_r V'+\frac{V}r \p_\theta u_\theta + W \p_z u_\theta \right) + \frac{Vu_r}r + \frac{1}r \p_\theta p &=0,\\
\lambda u_z +\left(  u_r W' + \frac{V}r \p_\theta u_z + W \p_z u_z \right) + \p_z p&=0.
   \end{split}
\end{equation}
These can be block diagonalized by applying the Fourier transform in the $z$ variable and Fourier series in the $\theta$ variable. It will therefore be enough to consider solutions 
\begin{equation}\label{perturbation_form}
    u = (u_r (r) e_r+ u_\theta (r) e_\theta+ u_z (r) e_z) \ee^{i(\alpha z - n \theta)}
\end{equation}
with eigenvalue $\lambda = -i\omega$, where we abuse notation by reusing $u_r$, etc. Namely, we seek solutions of the form $u\ee^{-i\omega t}$ to the linearized Euler equations, which are equivalent to the system \eqref{eq:linearizedevaleqn} for $u$ and $\lambda =-i\omega \in \C$. \emph{Notice the notational convention of $-n$ in \eqref{perturbation_form} which is meant to conform with~\cite{ls}.}  We will refer to \emph{unstable modes} as nontrivial solutions $u$ of the form \eqref{perturbation_form} with $\re \, \lambda >0$, i.e. 
\[\im \,\omega >0 \, .\]
A calculation shows that~\eqref{eq:linearizedevaleqn} can be reduced to an equation involving only the radial velocity~$u_r(r)$:
\eqnb\label{eq_for_u}
\frac{\d }{\d r } \left( \frac{r}{1+\beta^2 r^2} \frac{\d }{\d r} (r u_r(r)) \right) - n^2 \left( 1 + \frac{a(r)}{n \,\gamma (r) } + \frac{b(r)}{\gamma(r)^2} \right) u_r(r) =0,
\eqne
where\footnote{In the definition of $a$ in~\cite[(4.3b)]{ls}, there is a typo ($qn$ instead of $q$), which is corrected in~\cite[(2.3)]{ls_87}.}
\eqnb\label{def_abd}
\begin{split}
\beta &\coloneqq \frac{\alpha }{n},\\
a(r) &\coloneqq   r \frac{\d }{\d r} \left( \frac{(\beta r^2 +q )W'(r)}{r(1+\beta^2r^2 ) } \right) ,\\
b(r) &\coloneqq \frac{\beta r^2 (1-\beta q) \Phi (r) }{q(1+\beta^2 r^2)},\\
\Phi (r) &\coloneqq \frac{1}{r^3} \frac{\d }{\d r} (rV(r))^2,\\
q&\coloneqq - \frac{\frac{\d}{\d r} (r V(r) )}{W'(r)},
\end{split}
\eqne
and 
\begin{align}
\gamma (r) &\coloneqq  n \Lambda (r) - \omega \, , \label{def_of_gamma}\\ 
\Lambda (r) &\coloneqq  \beta W(r) - \Omega (r) \, , \label{def_of_Lambda}
\end{align}
see \cite{ls}. The angular and axial velocities $u_\theta$ and $u_z$ can be recovered from $u_r$ by taking $r \p_z$ of the second equation of \eqref{eq:linearizedevaleqn}  and subtracting $\p_\theta$ of the third equation in \eqref{eq:linearizedevaleqn}. This and the divergence-free condition give 
\begin{equation}\label{utheta_uz_from_ur}
\begin{pmatrix}
    in\omega + in^2 \frac{V}r-i\alpha n W & i\alpha \omega r +in \alpha V + \alpha^3 W r \\i\alpha  &-\frac{in}r
\end{pmatrix}\begin{pmatrix}
    u_z  \\u_\theta  
\end{pmatrix}=\begin{pmatrix}
    u_r \left( \alpha (rV)' +n W' \right)  \\-u_r' - \frac{u_r}r 
\end{pmatrix} .
\end{equation}
We note that the determinant of the matrix on the left-hand side of \eqref{utheta_uz_from_ur},
\[
n^2 \left( \omega \left( \frac1r + \beta^2 r \right) +  \left( \frac{n}{r^2} + n\beta^2 \right)V - \left( \frac{\alpha}r + \alpha \beta^2 r \right)W  \right)
\]
does not vanish for any $r>0$ (as $\omega$ is the only term with nonzero imaginary part).

The above Rayleigh equation \eqref{eq_for_u} is more complicated than its analogue for shear flows and vortices (see \cite{ABCD,ko_jmfm,lin_siam,linnonlinear,vishik_1,vishik_2}), but it is still possible to extract some stability conditions from it. Notably, there is a necessary and sufficient condition for stability with respect to a restricted class of perturbations. Namely, Howard and Gupta~\cite{hg} demonstrated that a columnar vortex is spectrally stable to \emph{axisymmetric perturbations}, that is, $n = 0$, precisely when its Rayleigh function $\Phi$ is non-negative. 
A second interesting special case is when $W = 0$, that is, when we consider 3D perturbations of a 2D vortex. When the 2D vortex satisfies the Rayleigh condition~\cite{rayleigh_1916}, namely, that its vorticity profile is decreasing, Gallay and Smets demonstrated spectral stability in the enstrophy class~\cite{gs_spectral} and, in the energy class, that the growth rate of the linear semigroup is, at most, subexponential~\cite{gs_linear}. We refer to the review article~\cite{gallay_stability} for further discussion; in Section~\ref{sec:instabcrit} we also discuss additional stability criteria for vortex columns~\eqref{vortex_col}, including an analogue~\cite{barston} of Howard's semicircle theorem~\cite{howard}.




  

\subsection{Linear instability}



We consider any vortex column~\eqref{vortex_col} satisfying the following.
\begin{customthm}{A}\label{ass}
The velocity components $V,W$ are such that $\Omega, W \in C^\infty([0,+\infty))$ and satisfy $\Omega^{(2k+1)}(0)=0$, $W^{(2k+1)}(0)=0$ and $\Omega^{(k)}(r), W^{(k)}(r) \to 0$ as $r\to \infty$, for all $k\geq 0$. Moreover, $0 < \inf q < \sup q < +\infty$ and there exist $r_0,\beta>0$, such that 
\begin{equation}
\label{choice_r0_beta}
    b'(r_0) = \Lambda'(r_0) =0,\qquad b_0 \coloneqq b(r_0) >0 , \qquad \Lambda'' (r_0) >0 \, .
\end{equation}
Finally, we assume that there is no $r_1  \in [0,+\infty)$, $r_1\ne r_0$, such that $\Lambda(r_1) = \Lambda(r_0)$. 
\end{customthm}

We note that one particular example satisfying \textbf{Assumption~\ref{ass}} is the Batchelor vortex \eqref{trailing_vortex}
\[
V(r) \coloneqq \frac{q}{r} (1-\ee^{-r^2}),\qquad W(r) \coloneqq \ee^{-r^2},
\]
in which case $q$, defined in~\eqref{def_abd}, is the constant $q = \max_r \Omega(r)$ representing the maximal pitch angle of the helical particle trajectories of the basic flow.  We show in Appendix~\ref{sec_choice_r0_beta} that \eqnb\label{beta_range}
\frac{q}2 < \beta < \frac1q 
\eqne
is a necessary condition for Assumption~\ref{ass}, and we show that there exist a unique choice of $r_0,\beta>0$ satisfying~\eqref{choice_r0_beta} if
\eqnb\label{q_restr}
q<\frac{\log 2 }{\sqrt{1-\log 2}}\approx 1.251.
\eqne
Note that \eqref{beta_range} implies that $q<2^{1/2}$, but, as shown in Appendix~\ref{sec_choice_r0_beta}, for $q$ close to $2^{1/2}$ there are no  $r_0,\beta$ satisfying \eqref{choice_r0_beta}. We also note that the two limiting cases $\beta \to (1/q)^-$ and $\beta \to (q/2)^+$, were studied by Capell and Stewartson \cite{cs} and  Leibovich and Stewartson \cite{ls_87}, respectively.\\


In order to describe the observation of Leibovich and Stewartson~\cite{ls}, we rewrite the equation in terms of the variable
\begin{equation}\label{phi_vs_ur}
\phi (r) \coloneqq \left( \frac{r^3}{1+\beta^2r^2 }\right)^{\frac12} u_r(r) \, .
\end{equation}
Then~\eqref{eq_for_u} becomes
\begin{equation}\label{eq_phi}
 \phi'' - k \phi = 0  ,
\end{equation}
where the potential $k(r)$ is defined by
\eqnb\label{def_of_k}
k\coloneqq  p n^2 \left( 1 + \frac{a}{n \gamma} + \frac{b}{\gamma^2} + \frac{d}{n^2} \right) ,
\eqne
where 
\eqnb\label{def_abd1}
\begin{split}
d(r) &\coloneqq -\frac{1+10\beta^2 r^2 - 3\beta^4 r^4 }{4(1+\beta^2 r^2 )^3},\\
p(r) &\coloneqq \frac{1+\beta^2 r^2}{r^2}.
\end{split}
\eqne
We sometimes omit the dependence on $r$ from the notation, although we shall keep in mind that all terms in \eqref{eq_phi}--\eqref{def_of_k} depend on $r>0$, except for $\beta >0$, $n\in \N$, and $\omega\in \C$, which are constants.

We now summarize the key points of the asymptotics analysis of Leibovich and Stewartson~\cite{ls}. They sought a sequence of unstable modes with growth rates $\Im \omega = O(1)$ as $n \to +\infty$ and the wavenumber ratio $\beta = \alpha/n$ is constant. For unstable modes to exist, it is necessary that the real part of the potential $k$, defined in~\eqref{def_of_k}, be somewhere negative. It is not obvious that this is possible, due to the presence of the ``$1$'' in the bracket in~\eqref{def_of_k}. It is plausible that the terms $a/(n\gamma)$ and $d/n^2$, containing $1/n$, will be negligible compared to $1 + b/\gamma^2$, and we temporarily ignore them. (Recall~\eqref{def_of_gamma} that $\gamma = n\Lambda - \omega$.) For $b/\gamma^2$  not to be negligible, it is necessary that $\omega = n\Lambda(r_0) + o(n)$ as $n\to \infty$, for a value $r_0$. Subsequently, we suppose that
\begin{equation}
    \label{eq:firstomegaexp}
    \omega = \underbrace{n\Lambda(r_0) + ib(r_0)^{1/2}}_{=: \omega_{\rm app}} + o(1)\qquad  \text{ as }n\to \infty,
\end{equation} so that $\omega_{\rm app}$ precisely cancels the $1$. In order to maximize the growth rate, it is natural to choose $r_0$ satisfying $b'(r_0) = 0$. The non-trivial behavior is concentrated near $r=r_0$, so we Taylor expand the potential around $r_0$:
\eqnb\label{intro_k_expansion}
k (r) = k_0 + k_2 (r-r_0)^2 + \text{\emph{remainder}} \, ,
\eqne
where the the \emph{remainder} may depend on $r$ and $n$, and the requirement $b'(r_0)=\Lambda'(r_0) = \gamma'(r_0) = 0$ eliminates the $k_1$ term in the sense that the linear part of the expansion arises only from $p,a,d$ (recall~\eqref{def_of_k}), which can be treated as part of the remainder. A direct computation shows that
\eqnb\label{k0_and_k2}
k_0 = -i \mu n^2 \frac{2p_0 }{b_0^{1/2}}\, ,\qquad k_2 = in^3 \frac{p_0 \Lambda''(r_0) }{b_0^{1/2}} \, ,
\eqne
where $p_0\coloneqq p(r_0)$, and $\mu$ contains the leading order terms of the $o(1)$ remainder in~\eqref{eq:firstomegaexp}, see Appendix~\ref{app_exp_of_k} for details. The main observation of~\cite{ls} is that replacing $k$ in the Rayleigh equation \eqref{eq_phi} by its expansion \eqref{intro_k_expansion} gives the \emph{Weber equation}
\begin{equation}
\label{intro_eq_phi_inner}
 \phi''(r) - (k_0 + k_2 (r-r_0)^2 ) \phi(r) =0 \, ,
\end{equation}
which has solutions that decay at both $r\to \pm \infty$ precisely when
\begin{equation}
\label{ks_relation}
k_0 = -  (2m-1)  k_2^{1/2}    
\end{equation}
for some  $m \in \N$, in which case the exact solutions are constant multiples of the  $m$-th \emph{Weber function} 
\eqnb\label{weber_wm}
w_m (r) \coloneqq  c_m  \exp \left( -\frac12 k_2^{1/2} (r-r_0)^2 \right) H_m \left( k_2^{1/4} (r-r_0)\right) ,
\eqne
where  $H_m (x)\coloneqq (-1)^{m-1} \ee^{x^2} (\frac{\d}{\d x})^{m-1}  \ee^{-x^2} $  denotes the $m$-th Hermite polynomial and $c_m > 0$ is a normalization factor ensuring $n^{3/8} \| w_m \|_{L^2} = 1$. Here the branch of $k_2^{1/2}$ is chosen so that $\re \, k_2^{1/2} >0$, i.e.,
\eqnb\label{choice_of_k_2}
k_2^{1/2} \coloneqq   (1+i){n^{3/2}} \left(\frac{p_0 \Lambda'' (r_0) }{2b_0^{1/2}} \right)^{1/2} ,
\eqne
and we define $k_0$ by \eqref{ks_relation}.\footnote{As for $k_2^{1/4}$ appearing in the definition of $w_m$ above, we pick any of the roots of $k_2^{1/2}$. (Recall that each Hermite polynomial is either odd or even, so the two choices of $k_2^{1/4}$ give the same solution \eqref{weber_wm}, up to the sign.)} Since $\mu$ appears in \eqref{k0_and_k2}, this defines infinitely many choices of $\mu$, which we denote by $\mu_m$, and which determine the next order of approximation in~\eqref{eq:firstomegaexp},
\begin{equation}
\label{choice_of_tilde_om_m_copy}
  \omega_m = \underbrace{n\Lambda_0 + i b_0^{1/2}}_{= \omega_{\rm app}} + \underbrace{(1-i ) n^{-1/2} (2m-1)  \left( \frac{b_0^{1/2} \Lambda''(r_0) }{8 p_0} \right)^{1/2}}_{=: \mu_m } + \hat{\omega}_m \, ,
\end{equation}
where $\hat{\omega}_m  = o(n^{-1/2})$ as $n\to \infty$ remains to be determined.

The above choice sets the ``inner length scale'' as $O(n^{-3/4})$, as can be seen from the scaling factor $n^{3/4}$ hiding inside $k_2^{1/4}$ in~\eqref{weber_wm}.  
In conclusion, we expect the $m$-th family of unstable modes to look like the $m$-th Weber function on length scale $|r-r_0| \leq O(n^{-3/4})$ and be centered at $r = r_0$. The eigenvalue $\omega$ (which is a part of the definition \eqref{def_of_k} of potential $k$) should satisfy the asymptotics~\eqref{choice_of_tilde_om_m_copy}. In~\cite{ls}, the authors present an impressive (formal) calculation which yields the eigenvalue up to and including the $O(n^{-3/2})$ term of $\omega$ for the ``fundamental" mode $m=1$.

To accompany the asymptotics in~\cite{ls}, Liebovich and Stewartson propose the condition that
\begin{equation}
    \label{eq:lssufficient}
V \frac{\d \Omega}{\d r} \left[ \frac{\d\Omega}{\d r} \frac{\d\Gamma}{\d r} + \left( \frac{\d W}{\d r} \right)^2 \right] < 0 \text{ somewhere in the flow field }
\end{equation}
 as a sufficient condition for instability. 
 Seemingly~\eqref{eq:lssufficient} should be supplemented by mild background conditions, like in~\textbf{Assumption~\ref{ass}}. \\


In Theorem~\ref{thm_main}, we give a rigorous mathematical construction of the full family ($m \geq 1$) of unstable ring modes via an inner-outer gluing procedure. 

\begin{thm}[Unstable modes] \label{thm_main}
Suppose that $v$ is a vortex column \eqref{vortex_col} satisfying \textbf{Assumption~\ref{ass}}. Let $m \in \N$. Then, for all $n$ sufficiently large depending on $m$, there exists a solution $(\phi_m,\omega_m)$ to the Rayleigh equation~\eqref{eq_phi} satisfying, for every $\delta> 0$, the estimates
\begin{equation}\label{apriori_l2}
  n^{3/8}  \| \phi_m - w_m \|_{L^2} \lesssim n^{-1/2+\delta},
\end{equation}
\begin{equation}
    \omega_m = \omega_{\rm app} + \mu_m + O(n^{-1+\delta}) \, ,
\end{equation}
where the implied constants may depend on $m,\delta$, and $v$. Moreover, the solutions are unique in the following sense. Let $M \in \N$ and $C_M \coloneqq M (b_0^{1/2} \Lambda'' (r_0) / p_0 )^{1/2}$.
Then, for all sufficiently large $n$ (depending on $M$), any solution $\psi \in H^1_0(\R^+)$ to the Rayleigh equation~\eqref{eq_phi} with $\omega \in B(\omega_{\rm app}, C_M n^{-1/2})$ satisfies that, for some $m=1, \ldots , M$, $\omega = {\omega}_m$ and $\psi$ is a multiple of $\phi_m$. 
\end{thm}

We moreover provide precise estimates on $\phi_m$ in Gaussian weighted spaces, see Theorem~\ref{thm_main2} for a more precise statement. We note that the eigenfunctions $\phi_m$ are smooth and belong to the Sobolev space $H^k$ for all $k \geq 0$, as we demonstrate in Section~\ref{sec_prop_LL}.

The uniqueness assertion in Theorem~\ref{thm_main} provides an asymptotic classification of \emph{sequences} of unstable eigenfunctions $\phi_k$ and eigenvalues $- i\omega_k$ satisfying $|\omega_k - (n_k \Lambda_0 + ib_0^{1/2})| \lesssim n_k^{-1/2}$ and $n_k \to +\infty$ as $k \to +\infty$. Note that $C_M$ is such that $B(\omega_{\rm app}, C_M n^{-1/2})$ contains the first $M-1$ modes, see~Figure~\ref{fig_semicircle} below.

\subsubsection{Instability criteria} \label{sec:instabcrit}
We now discuss how Theorem~\ref{thm_main} relates to some known instability criteria. 

First, according to the semicircle theorem of Barston \cite{barston}, every unstable mode $\omega$ belongs to $B(\omega_0,\omega_{max})$, where
\[
\omega_0 \coloneqq \frac{n}{2} \left( \min_{r} \Lambda(r) + \max_r \Lambda(r) \right),\qquad \omega_{\rm max}^2 \coloneqq \frac{n^2}4 \left( \max_{r} \Lambda(r) - \min_r \Lambda(r) \right)^2 + \max_r \widetilde{R}(r)^2,
\]
and
\[
\widetilde{R}(r)^2 \coloneqq \left(      \left( r\Omega (r) \Omega'(r) \right)^2 + \Omega (r)^2 \left( \max_{r} \Lambda(r) - \min_r \Lambda(r) \right)^2 \right)^{\frac12} - r \Omega (r) \Omega'(r).
\]
In the case of the Batchelor vortex \eqref{trailing_vortex}, we obtain
\[
\omega_0 = n \Lambda (r_0)/2 \, , \qquad \omega_{\rm max} = \sqrt{ n^2 \Lambda (r_0)^2/4 + \max_r \widetilde{R} (r)^2 } \, ,
\]
see Figure~\ref{fig_semicircle}.
\begin{center}
\includegraphics[width=0.8\textwidth]{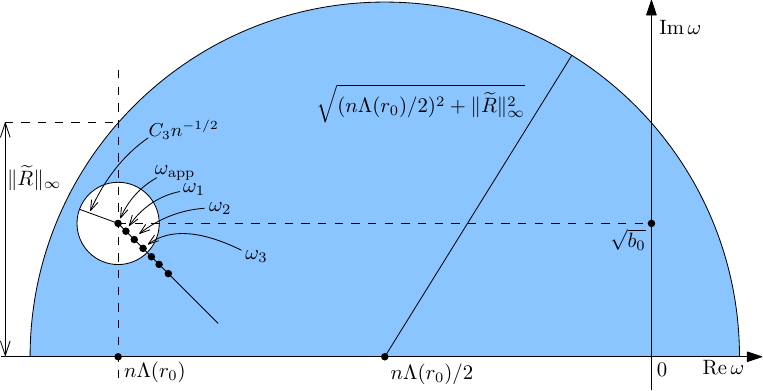}
  \captionof{figure}{Unstable modes  \eqref{omega_modes} and the Barston \cite{barston} semicircle theorem. As required by the semicircle theorem, $\sqrt{b_0} < \max_r \widetilde{R}$, i.e., all the modes obtained in Theorem~\ref{thm_main} belong to the semicircle $B(\omega_0,\omega_{\rm max})\cap \{ \im \, \omega >0 \}$. For example, in the case $q=1/4$, one has $\beta \approx 0.1589$, $r_0\approx 0.8252$ (see Appendix~\ref{sec_choice_r0_beta}),
  which then shows that $\sqrt{b_0} \approx 0.194 < 0.218 \approx \max_r \widetilde{R}$.  }\label{fig_semicircle} 
\end{center}

We also point out the criterion of Howard and Gupta \cite[(21)]{hg} that the flow \eqref{vortex_col} is stable if
\eqnb\label{hg_crit}
\alpha^2 \Phi - \frac{2\alpha n}{r^2} V W' - \frac14 \left( \alpha V' + n \Omega' \right)^2 \geq 0
\eqne
for all $r$. In the case of the Batchelor vortex \eqref{trailing_vortex}, this condition becomes
\eqnb\label{hg_trailing_v}
-\left( 1-\ee^{-r^2} \right) (1+\beta q) \ee^{-r^2} \frac{4q \beta }{r^2} - \frac{q^2}{r^6} \left( 1-\ee^{-r^2} \left( 1+r^2+\frac{\beta}q r^4 \right) \right)^2 \geq 0 \, ,
\eqne
which is actually violated for all $r>0$; the second term on the left-hand side (including the minus sign) equals $-\frac14 \Lambda'(r)^2$, which vanishes at $r_0$, the center of concentration of the ring modes. 

We also note that Howard and Gupta \cite{hg} provide an upper bound on $\im\,\omega$ for any unstable solution of \eqref{eq_for_u},
\[
(\im\,\omega)^2 \leq \max_r \frac{\beta^2 r^2}{1+\beta^2 r^2} \left( \frac14 (W')^2 - \Phi + \frac{1}{2\beta r^4} W' (r^3 V)' + \frac1{\beta^2} (\Omega')^2 \right), 
\]
see also the discussion in \cite[Section~5]{ls}. Furthermore, \cite[(23)]{hg} provides an analogue of the Rayleigh criterion for instability, 
\eqnb\label{rayleigh_by_hg}
\begin{split}
&\int_0^\infty \left( \frac{\beta^2 r }{1+\beta^2r^2} | (r u)' |^2 + \left( 1+ \frac{\beta^2  r}{\gamma }  \left( \frac{2n V+r^2 \gamma' }{r+\beta^2 r^3}  \right)' - \frac{2 \beta^3 V  n^2}{\gamma^2 (1+\beta^2 r^2)} \left( \beta rV +  W \right)'  \right) |u |^2 r \right) \d r =0
\end{split}
\eqne
Considering the scenario of neutral limiting modes (i.e., when $\im\,\omega >0$ is obtained as a perturbation of $\im \, \omega = 0$), we have that, for the limiting mode, $\gamma$ is a real function (recall~\eqref{def_of_gamma}), and one can deduce from \eqref{rayleigh_by_hg} that, if $\left( \beta rV +  W \right)' $ vanishes, then 
\[ 
\left( \frac{2n V+r^2 \gamma' }{r+\beta^2 r^3}  \right)' \text{ must change sign in }(0,\infty) \, ,
\]
which is an analogue of the Rayleigh's inflection point criterion for parallel shear flows~\cite{rayleigh1879stability}. 
However, in the case of the Batchelor vortex \eqref{trailing_vortex}, the vanishing is equivalent to $\beta q =1$, which does not occur with the restriction \eqref{beta_range}. This is not surprising, as the instabilities obtained by Theorem~\ref{thm_main} are not of neutral limiting type.

However, this raises an interesting open problem of existence of the upper limiting modes as $\beta \to (1/q)^-$, that is, when the relation between $\alpha$ 
and $n$ 
is not constant. In that case, the analysis of this paper breaks down (as the assumption $b(r_0)>0$ in \eqref{choice_r0_beta} fails). Stewartson and Capell~\cite{cs} provide some evidence of existence of such modes. However, the question of a rigorous verification and establishing a connection with the well-understood theory of instabilities bifurcating from neutral limiting modes (see~\cite{ko_jmfm,lin_siam} in the shear flow case) remains an interesting open problem. 
See~\cite{ls} for additional information.

\subsection{Nonlinear instability}
In our second result (Theorem~\ref{thm:nonlinearinstab} below) we deduce nonlinear instability from the modal instability. Namely, we assume existence of an unstable solution to the linearization of the $3$D Euler equations \eqref{3d_euler} in the form \eqref{perturbation_form}, and prove that such linear instability gives rise to \emph{nonlinear instability}. 

We denote the vortex column \eqref{vortex_col} by
\begin{equation}\label{vortex_col_repeat}
   \ou = V(r) e_\theta + W(r) e_z.
\end{equation}
We focus on the domain $\R^2\times \T$ for convenience, due to the periodicity in $z$ of the modes of the form~\eqref{perturbation_form}. After rescaling, we may take $\T$ to be the flat torus of unit length.

We suppose that $\ou$ satisfies  $\nabla^k \bar{u} \in L^\infty$ for all $k \geq 0$ and
\begin{equation}
    \label{eq:decayassumption}
    |\oo | + r|\nabla \oo | \lesssim r^{-\bb} \, ,
\end{equation}
for some $\bb > 0$, where $\oo \coloneqq \mathrm{curl}\, \ou$. Note that this is valid for the Batchelor vortex \eqref{trailing_vortex}. 

We fix $Q\in (1,2)$ and consider the linearization of the $3$D Euler equations~\eqref{3d_euler},
\eqnb\label{linearization}
  \omega_t - \LL \omega = 0 \, ,
 \eqne
in $L^Q(\R^2 \times \T)$, where 
\eqnb\label{def_LL}
- \LL \omega  \coloneqq [\ou , \omega ]-[u,\oo ] = \ou \cdot \nabla \omega - \omega \cdot \nabla \ou - \oo \cdot \nabla u + u \cdot \nabla \oo .
\eqne 
 Here we use the Lie bracket notation $[v,w]\coloneqq v\cdot \nabla w - w\cdot \nabla v$ for vector fields, and we also denote by $u$ the velocity field generated by $\omega$ via the Biot-Savart law. The definition of the Biot-Savart law on $\R^2\times \T$ is a subtle matter, since the domain $\R^2 \times \T$ allows, in particular, two dimensional flows (i.e. constant in $z$). In particular, according to the two-dimensional Biot-Savart law ${\rm BS}_2 : L^2 \to \dot H^1$, the velocity fields are only defined \emph{up to constants}, unless further conditions are imposed, e.g., more decay at spatial infinity; this is why we work in the space $L^Q$.  For any $p\geq 1$, we set
 \begin{equation}
     \label{def_Lpsigma} 
     L^p_\sigma \coloneqq \{ f \in L^p (\R^2 \times \T ; \R^3 ) \colon f\text{ is weakly divergence-free}\} \, .
 \end{equation}
 We will consider a decomposition of $L^Q_\sigma$ into invariant subspaces 
 \eqnb\label{def_Xna}
 X_{n,\alpha } \coloneqq \{ u\in L^Q_\sigma \colon u = v(r) \ee^{in\theta + i \alpha z } \text{ for some }n,\alpha \in \Z \text{ and } v\in L^Q((\R^+ ,r\d r);\R^3 ) \} \, .
 \eqne
Then $\LL \colon  D(\LL ) \subset L^Q_\sigma   \to L^Q_\sigma$ is a closed unbounded operator with dense domain
 \begin{equation}
 D(\mathcal{L}) := \{ \omega \in L^Q_\sigma : \bar{u} \cdot \nabla \omega \in L^Q \} \, .
 \end{equation}
We assume modal linear instability of the Euler equations \eqref{3d_euler} in vorticity form around $\ou$: 
\eqnb\label{linear_ass}
\text{there exists }\eta  \in D(\mathcal{L}) \text{ and } \lambda \in \C\text{ such that }  \re \, \lambda >0\text{ and } \eta \ee^{\lambda t} \text{ satisfies \eqref{linearization}},
\eqne
namely, that
\eqnb\label{unst_eq}
 \LL \eta = \lambda \eta 
 \eqne
 holds in the sense of distributions. If an eigenfunction exists, then there must be an eigenfunction in one of the invariant subspaces $X_{n,\alpha}$ with the same eigenvalue.\footnote{To prove this, we apply the projection $P_{n,\alpha}$ onto the invariant subspace $X_{n,\alpha}$ to the eigenvalue equation~\eqref{unst_eq}. This projection commutes with the operator $\mathcal{L}$; this is essentially due to the fact that $\bar{u}$ is axisymmetric and $z$-independent. In order for $\eta$ to be non-zero, we must have that $P_{n,\alpha} \eta$ is non-zero for some particular $n,\alpha$; then $P_{n,\alpha} \eta$ is an eigenfunction with the same eigenvalue.}


 The existence of infinitely many such $\eta$s and $\lambda$s follows from Theorem~\ref{thm_main} for any vortex column \eqref{vortex_col} for which the sufficient condition \eqref{choice_r0_beta} holds, see Corollary~\ref{cor_lin_unst}.

To avoid confusion, we note that, in the context of Theorem~\ref{thm:nonlinearinstab} ``$\omega$'' will refer to the vorticity, rather than to the complex wave speed~\eqref{eq:twistedeval} in the linear instability part (Theorem~\ref{thm_main}). We will only use $\lambda $ to denote the assumed unstable eigenvalue. 
 
We can now state our second main result.
\begin{thm}[Linear implies nonlinear instability]
    \label{thm:nonlinearinstab}
Suppose that a vortex column \eqref{vortex_col_repeat} satisfying the decay assumption \eqref{eq:decayassumption} is linearly unstable in the sense of \eqref{linear_ass}. Then the vortex column is \emph{nonlinearly unstable} in the $L^\infty(\R^2 \times \T)$ topology of velocity in the following sense: For every $N\geq 2$ there exists $\delta > 0$ such that, for any $\varepsilon > 0$, there exist $T_\varepsilon > 0$ and a strong solution $u$ of the $3$D Euler equations \eqref{3d_euler} on $\R^2 \times \T \times [0,T_\varepsilon]$ such that
    \begin{equation}\label{ee3}
        \| \omega|_{t=0} - \oo \|_{L^{Q} \cap H^N} \leq \varepsilon
    \end{equation}
    and
    \begin{equation}
        \| u|_{t=T_\varepsilon} - \ou  \|_{L^\infty} \geq \delta  \, .
    \end{equation}
\end{thm}


We expect that a similar theorem is possible in the energy space with a different proof. 
The term `strong solution' is in reference to a particular well-posedness class which we explain in Section~\ref{sec_pf_thm2}. 
 

While the slab domain $\R^2 \times \T$ is convenient for discretizing the spectrum, it is inconvenient for the Biot-Savart law and Sobolev embedding. However, a similar nonlinear instability theorem will hold on $\R^3$ (cf. Corollary~1.2 in Grenier \cite{grenier}).

We emphasize that the instability provided by Theorem~\ref{thm:nonlinearinstab} ensures that the growth of the velocity persists no matter the order of the $H^N$ Sobolev norm of the initial perturbation in~\eqref{ee3}.

There are many mathematical works investigating under which conditions does linear instability imply nonlinear instability for the Euler equations~\cite{friedlanderstraussvishikearly,friedlanderstraussvishik,bardosguostrauss,friedlandervishik,linnonlinear,LinZengCPAM2013,LinZengCorrigendum2014}. Curiously, to the best of our knowledge, unstable vortex columns do not seem to be covered by any of them. The main difficulty is discussed in Section~\ref{sec_pf_thm2}. \\

In Section~\ref{sec_proof_thm1}, we construct the unstable eigenfunctions advertised in Theorem~\ref{thm_main}. In Section~\ref{sec_pf_thm2}, we prove the nonlinear instability of Theorem~\ref{thm:nonlinearinstab}.

\section{Proof of Theorem~\ref{thm_main}}\label{sec_proof_thm1}


Recall from Theorem~\ref{thm_main} that $m\geq 1$ represents the mode of the instability. From this point onward we will often write
\[
w(r) = w_m(r),
\] 
keeping in mind that the existence proof of Theorem~\ref{thm_main} is carried out for a fixed $m$.\\

The proof of Theorem~\ref{thm_main} is based on a gluing procedure partially inspired by~\cite{WeiHarmonicMap,ABCgluing}. Let $\eta \in C^\infty_0(\R ; [0,1])$ be a smooth cutoff such that $\eta (x) =1$ for $|x|\leq 1$ and $\eta (x) =0$ for $|x|\geq 2$. We set $\eta_{\ell}(\cdot) = \eta(\frac{\cdot-r_0}{\ell})$ and, given $l_{\rm out}, \ell_{\rm in}>0$, we define
\eqnb\label{cutoffs_def}
\begin{split}
\chi_{\rm in}(\cdot) &= \eta_{\ell_{\rm in}}\\
\chi_{\rm out}(\cdot) &= 1 - \eta_{\ell_{\rm out}} .
\end{split}
\eqne
We require that the two cutoffs $\chi_{\rm in}$ and $\chi_{\rm out}$ overlap in the sense that $\supp  \chi'_{\rm in} \subset \{ \chi_{\rm out} = 1 \}$ and vice versa, which holds given 
\begin{equation}\label{overlap_cond}
2\ell^{\rm out} \leq \ell^{\rm in} \, .    
\end{equation}
We make the ansatz
\begin{equation}
    \label{inner+outer}
\phi \coloneqq \phi_{\rm in} \chi_{\rm in} + \phi_{\rm out} \chi_{\rm out} \, .
\end{equation}
The ansatz involves some redundancy of the `inner solution' $\phi_{\rm in}$ and `outer solution' $\phi_{\rm out}$ on the overlap region; this is typical of gluing procedures. The Rayleigh equation $\phi''-k\phi =0$ (recall~\eqref{eq_phi}) holds  if
\eqnb\label{in_out_problem}
\begin{cases}
&\phi_{\rm in}'' - k \phi_{\rm in} + (2\phi_{\rm out}' \chi_{\rm out}' + \phi_{\rm out} \chi_{\rm out}'') =0 \hspace{0.3cm}\text{ on } {\rm supp} \, \chi_{\rm in} \\
&\phi_{\rm out}'' - k \phi_{\rm out} + (2\phi_{\rm in}'\chi_{\rm in}' + \phi_{\rm in}\chi_{\rm in}'' )=0 \hspace{0.7cm} \text{ on } {\rm supp} \, \chi_{\rm out}  \, .
\end{cases}
\eqne
We regard $\phi_{\rm in}$ and $\phi_{\rm out}$ as functions defined on $\R$ and $\R^+$, respectively. 

The `inner equation' resembles the Weber equation $\phi'' - (k_0 + k_2(r-r_0)^2) \phi = 0$, where $k_0 = O(n^{3/2})$ and $k_2 = O(n^3)$ were identified in~\eqref{k0_and_k2}, and, as noted below~\eqref{choice_of_tilde_om_m_copy}, it is best observed at the length scale $n^{-3/4}$. Therefore, we introduce the inner variable
\eqnb\label{def_xi}
\xi \coloneqq (r-r_0)n^{3/4} \, .
\eqne
We denote by 
\eqnb\label{kerr_def}
k_{\rm err} (r) \coloneqq k(r) - (k_0 + k_2 (r-r_0)^2)
\eqne
the remainder in the approximation of $k$, and we note that
\eqnb\label{Kerr_est}
k_{\rm err } (r) = in^2 \ho \frac{2p_0}{b_0^{1/2}}+ n (1+ n^3 (r-r_0)^4 ) O(1 )  
\eqne
provided that $|\mu_m| + |r-r_0| = O(n^{-1/2})$, see~\eqref{remainder_11_app} in the appendix for the precise statement and proof. 
 We rescale the potential $k$ according to
\eqnb\label{k_rescaling}
\begin{split}
K (\xi ) &\coloneqq n^{-3/2} k(r)  , \quad K_0 \coloneqq n^{-3/2} k_0 , \quad K_2\coloneqq n^{-3} k_2 \\
K_{\rm err} (\xi ) &\coloneqq n^{-3/2} k_{\rm err} (r) = K (\xi ) - (K_0 + K_2 \xi^2 ) \, .
\end{split}
\eqne
We rescale the solution according to
\eqnb\label{rescalings}
\Phi_{\rm in} (\xi ) \coloneqq \phi_{\rm in} (r),\qquad X_{\rm out} (\xi )\coloneqq \chi_{\rm out} (r) \, ,
\eqne
and similarly for $\Phi_{\rm out}$ and $X_{\rm in}$. We denote by $R\colon L^2 \to L^2$  the rescaling from the outer variable to the inner variable,
\eqnb\label{def_rescaling}
(Rf)(\xi ) \coloneqq f(r) = f (r_0 + n^{-3/4}\xi ) \, .
\eqne
For example, $\Phi_{\rm in} = R\phi_{\rm in}$. 
We loosely follow the convention that functions denoted by a capital letters are functions of $\xi$ and are rescalings, via $R$, of the corresponding function of $r$ denoted by the respective lower case letters. (One exception is the potential $k$, whose rescaling \eqref{k_rescaling} also involves multiplication by $n^{-3/2}$.) Notably,
\eqnb\label{norm_of_R}
\| R \|_{L^2 \to L^2} = n^{3/8} ,\qquad \| R^{-1} \|_{L^2 \to L^2} = n^{-3/8} \, .
\eqne

With the above notation, we rewrite the inner equation of system~\eqref{in_out_problem} as
\eqnb\label{inner_eq_rescaled}
{\Phi}_{\rm in}'' - (K_0 +K_2 \xi^2 )  {\Phi}_{\rm in} = K_{\rm err} {\Phi}_{\rm in} - R\phi_{\rm out} X_{\rm out}''  - 2 n^{-3/4}R\phi_{\rm out}' X_{\rm out}' \text{ on } {\rm supp} \, X_{\rm in} \, .
\eqne
The main point of the rescaled inner equation \eqref{inner_eq_rescaled} is that the coefficients $K_0$, $K_2$ \emph{do not depend on $n$}. The solutions of the homogeneous equation
\begin{equation}
    W_m'' - (K_0 +K_2 \xi^2 )  W_m = 0 \text{ on } \R
\end{equation}
are precisely constant multiples of the the Weber functions $W_m =  R w_m$, recall \eqref{weber_wm}, which implies in particular that $\| W_m \|_{L^2} = 1$. As mentioned in the beginning of the section, we will drop the index $m$. We make the ansatz
\begin{equation}
\Phi_{\rm in} = W + \Psi \, ,
\end{equation}
where $\Psi$ is a remainder whose PDE is
\eqnb\label{eq_Psi}
\Psi''- (K_0 + K_2 \xi^2 )\Psi =  K_{\rm err} W  + K_{\rm err} \Psi  + \Phi_{\rm out} X_{\rm out}''-2 (\Phi_{\rm out} X_{\rm out}')' \, ,
\eqne
satisfied on ${\rm supp} \, X_{\rm in}$. It will be more convenient to pose this PDE on $\R$, and we have the freedom to modify the potential off of ${\rm supp} \, X_{\rm in}$. We do so by cutting off by $\widetilde{\ell_{\rm in}}\coloneqq 2 \ell_{\rm in}$, $\widetilde{\chi_{\rm in}}\coloneqq \eta_{\widetilde{\ell_{\rm in}}}$, $\widetilde{X}_{\rm in} \coloneqq R\widetilde{\chi_{\rm in}}$, and we replace~\eqref{eq_Psi} by the same equation, but with the right-hand side multiplied by $\widetilde{X}_{\rm in}$:
\eqnb\label{eq_Psi_better}
\Psi''- (K_0 + K_2 \xi^2 )\Psi =  \widetilde{X}_{\rm in} (K_{\rm err} W  + K_{\rm err} \Psi ) + \Phi_{\rm out} X_{\rm out}''-2 (\Phi_{\rm out} X_{\rm out}')' =:G .
\eqne

The operator $\frac{\d^2}{\d \xi^2} - (K_0 + K_2 \xi^2)$ has a one-dimensional kernel spanned by $W$, and its range is $({\rm span} \, W^*)^\perp$ (see Lemma~\ref{lem:solvabilityinnerprojected}).\footnote{
Multiplying~\eqref{eq_Psi_better} equation by $W$ yields that $\int W G =0$, which suggest that this is a solvability condition for $\Psi$. 
} The equation~\eqref{eq_Psi_better} is therefore a suitable candidate for the \emph{Lyapunov-Schmidt reduction}. We thus define the projection $Q \colon L^2 \to L^2$ onto the range of $\frac{\d^2}{\d r^2}-(K_0 +K_2 \xi^2 )$:
\eqnb\label{def_of_Q}
Q F \coloneqq F - \left(\int F W \right) W^* .
\eqne
Indeed, if $\Upsilon'' - (K_0 +K_2\xi^2 ) \Upsilon =F$ then multiplication by $W$ and integration by parts shows that $\int F W =0$. We can thus rewrite \eqref{eq_Psi} as the \emph{projected equation}
\eqnb\label{inner_Psi_projected}
\Psi'' - (K_0 + K_2 \xi^2 ) \Psi = QG,
\eqne
coupled with the one-dimensional \emph{reduced equation} $(I-Q) G=0$, that is,
\eqnb\label{reduced_eq}
\int GW =0 \quad \Leftrightarrow \quad  \int  \left( \widetilde{X}_{\rm in}K_{\rm err} (W + \Psi ) - R\phi_{\rm out} X_{\rm out}''  -  2  n^{-3/4}R\phi_{\rm out}' X_{\rm out}' \right) W \, \d\xi =0 \, .
\eqne
The projected equation~\eqref{inner_Psi_projected} is uniquely solvable when we we supplement it with an additional condition
\begin{equation}
    \label{eq:Psibarcond}
    \int \Psi W^* = 0 \, ,
\end{equation}
for example, and when $K_{\rm err}$ is small. The latter condition dictates some constraints on $\ell_{\rm in}, \ell_{\rm out}$ (in addition to~\eqref{overlap_cond}), see the discussion below Lemma~\ref{lem:solvabilityprojected} for details.

Thus far, we have neglected the outer equation  (recall~\eqref{in_out_problem}). It will be extended to $\R^+$:
\begin{equation}
    \label{eq:extendedouterequation}
    \phi_{\rm out}'' - \widetilde{k} \phi_{\rm out} + (2\phi_{\rm in}'\chi_{\rm in}' + \phi_{\rm in}\chi_{\rm in}'' ) = 0 \text{ on } \R^+ \, ,
\end{equation}
and, after we analyze $k$ in the region $|r-r_0| \gtrsim n^{-3/4}$, where $\widetilde{k}$ is defined in \eqref{ktilde_def} and denotes an appropriate extension of $k$ guaranteeing coercivity.

To complete the construction, the projected equation~\eqref{inner_Psi_projected} with supplementary condition~\eqref{eq:Psibarcond} and the outer equation~\eqref{eq:extendedouterequation} will be solved simultaneously, as a system, and the solution $(\Psi,\phi_{\rm out})$ will be holomorphic with respect to $\omega$. The reduced equation~\eqref{reduced_eq} will determine $\omega$ by specifying the remaining $o(n^{-1/2})$ term $\widehat{\omega}=\widehat{\omega}(n)$ in the expansion~\eqref{choice_of_tilde_om_m_copy} of $\omega$, see Lemma~\ref{L_reduced}. To solve this equation, \emph{it will be necessary to examine precisely how $\omega$ enters the error term $K_{\rm err}$}, see~\eqref{fgh}. This is analogous to how the Plemelj formula enters the calculation,  for shear flows and vortices, of the one-sided derivative $c'(m_0)$ of the wavespeed $c$ with respect to the wavenumber $m$ as $\Im c \to 0^+$ at a limiting neutral mode $(\phi_0,m_0,c_0)$.\footnote{A \emph{limiting neutral mode} is a non-trivial solution $(\phi_0,m_0,c_0)$ of the Rayleigh equation~\eqref{eq:shearRayleigh} with $\Im c_0=0$ and which is, furthermore, approximable by solutions $(\phi,m,c)$ with $\Im c > 0$. In~\cite{ko_jmfm,lin_siam}, the notation $\alpha$ is used instead of $m$. In~\cite{vishik_1,vishik_2} and~\cite[Chapter 4]{ABCD}, $z$ is used instead of $c$. Of course, our calculation is different, and our unstable modes are not perturbations of a neutral mode.}  

We define the \emph{inner function space}
\eqnb\label{def_of_Y}
Y \coloneqq \{ F\in H^2 (\R ) \colon \| F \|_Y \coloneqq \|  F \|_{H^2 (\R) } + \|  \xi^2  F(\xi ) \|_{L^2 (\R )} <\infty \}
\eqne
and its Gaussian weighted analogue
\begin{equation}\label{def_of_Yw}
Y_w \coloneqq \{ F\in H^2 (\R ) \colon \| F \|_{Y_w} \coloneqq \|  |F''| + |\xi| |F'| + \xi^2 |F|  \|_{L^2_w} <\infty \} \, ,
\end{equation}
where
\begin{equation}
    \| F \|_{L^2_w} \coloneqq \|  \ee^{q_0 |K_2|^{1/2} \xi^2/\sqrt{8}} F \|_{L^2} \, ,
\end{equation}
with parameter $q_0$ taking any value in $(0,1)$. For convenience we fix  $q_0 \coloneqq 1/2$.
We define the \emph{outer function~space}
\eqnb\label{def_of_Z}
Z \coloneqq H^1_0 (\R^+) \, ,\qquad \| f \|_Z \coloneqq n^{-3/8} \| f' \|_{L^2} + n^{3/8} \| f \|_{L^2} \, .
\eqne
We now give a more precise version of the existence part of Theorem~\ref{thm_main}.

\begin{thm}[Existence for glued system]
\label{thm_main2}
Suppose that $v$ is a vortex column~\eqref{vortex_col} satisfying \textbf{Assumption~\ref{ass}}. Let $m \in \N$ and $\delta \in (0,1/8]$. There exist constants $D_{\rm in}, D_{\rm out} \geq 1$ such that, for all $n$ sufficiently large, the following property holds: Let
\eqnb\label{claim_choices_of_scales}
\ell_{\rm out} = D_{\rm out} n^{-3/4} ,\qquad \ell_{\rm in} = D_{\rm in}^{-1} n^{-3/4+\delta} \, .
\eqne
Then there exists a solution $(\Psi,\phi_{\rm out},\omega) \in Y \times Z \times \C $ to the glued system \eqref{eq_Psi}--\eqref{eq:extendedouterequation} satisfying
\begin{equation}
    \label{omega_modes}
\omega = \omega_{\rm app} + \mu_m + \hat{\omega} \, , \quad
  |\widehat{\omega}| \lesssim n^{-1+4\delta} \, ,
\end{equation}
\eqnb\label{apriori}
\| \Psi \|_{Y_w} \lesssim n^{-1/2+4\delta} \, , \quad
\| \phi_{\rm out} \|_{Z} \lesssim \exp\left( - C^{-1} n^{2\delta} \right) \, .
\eqne
\end{thm}
(Recall \eqref{eq:firstomegaexp} and \eqref{choice_of_tilde_om_m_copy} that $\omega_{\mathrm{app}} = n\Lambda(r_0) + i b_0^{1/2}$ and $\mu_m = (1-i ) n^{-1/2} (2m-1)  \left( \frac{b_0^{1/2} \Lambda''(r_0) }{8 p_0} \right)^{1/2}$, respectively.)

In particular, the glued solution $\phi$ in~\eqref{inner+outer} satisfies Gaussian weighted-in-$\xi$ estimates when restricted to $|r-r_0| \leq 2\ell_{\rm in}$ and super polynomial-in-$n$ estimates when restricted to $|r-r_0| \geq 2\ell_{\rm in}$. By the uniqueness part of  Theorem~\ref{thm_main}, \emph{the unique solution $(\phi,\omega)$ satisfies all these estimates at~once}, that is, for every $\delta \in (0,1/8]$, \eqref{omega_modes}--\eqref{apriori} hold if $n$ is sufficiently large.

The precise reasoning for the choices of cut-off scales $\ell_{\rm in}$ and $\ell_{\rm out}$ is given in Section~\ref{sec_solv_proj}. The scheme for uniqueness is sometimes known as ``reverse gluing", and we discuss it in Section~\ref{sec_uniqueness}.

We proceed in the following way:
\begin{itemize}
    \item Section~\ref{sec_solv_inner}. We analyze the inner operator $\frac{\d^2}{\d \xi^2} - (K_0 + K_2\xi^2)$ on $\R$.
    \item Section~\ref{sec_solv_outer}. We analyze the outer operator $\frac{\d^2}{\d r^2} - \widetilde{k}$ on $\R^+$.
    \item Section~\ref{sec_solv_proj}. We solve the projected system ~\eqref{inner_Psi_projected},~\eqref{eq:Psibarcond}, and~\eqref{eq:extendedouterequation}.
    \item Section~\ref{sec_solv_reduced}. We solve the one-dimensional reduced equation~\eqref{reduced_eq} to complete Theorem~\ref{thm_main2}.
    \item Section~\ref{sec_uniqueness}. We prove the uniqueness assertion in Theorem~\ref{thm_main}.
\end{itemize}

\subsection{The inner equation}\label{sec_solv_inner}

Recall from~\eqref{ks_relation} that the constants $K_0,K_2$ satisfy $K_0 = \ee^{-3i\pi/4}|K_0|$, $K_2 = i|K_2|$, and $|K_0| = (2m-1) |K_2|^{1/2}$ for some $m\geq 1$, though we suppress the dependence on $m$. The main result of this section is the following.

\begin{lemma}[Solvability of projected inner equation]
    \label{lem:solvabilityinnerprojected}
For each $m\geq 1$ and $F\in L^2$, there exists a unique solution $\Upsilon \in H^2 \cap L^2 (\xi^2 \d \xi )$ of~\eqref{inner_Psi_projected} and \eqref{eq:Psibarcond}, that is,
\eqnb\label{model_inner_prob_withP}
\Upsilon'' - (K_0 + K_2 \xi^2 ) \Upsilon = Q F \, , \qquad \int \Upsilon W^* \, \d\xi = 0 \, ,
\eqne
and moreover,
\eqnb\label{non-gaussian_est}
\| \Upsilon \|_Y := \| \Upsilon \|_{H^2} + \| \xi^2 \Upsilon (\xi ) \|_{L^2} \lec \|  F \|_{L^2} \, .
\eqne
If additionally $q_0 \in [0,1)$ and $\ee^{q_0 |K_2|^{1/2} \xi^2/\sqrt{8}} F \in L^2$, then we have the Gaussian weighted estimate
\begin{equation}
    \label{gaussian_decay_lemma}
\| \ee^{q_0 |K_2|^{1/2} \xi^2/\sqrt{8}} (|\Upsilon''| + |\xi| |\Upsilon'| + \xi^2 |\Upsilon|) \|_{L^2} \lec_{q_0} \| \ee^{q_0 |K_2|^{1/2} \xi^2/\sqrt{8}} F \|_{L^2} \, .
\end{equation}
\end{lemma}

We emphasize that the right-hand side of the inner equation \eqref{eq_Psi_better} involves $K_{\rm err} \Psi$, and so, in order to use Lemma~\ref{lem:solvabilityinnerprojected} to solve \eqref{eq_Psi_better}, one also needs to control $K_{\rm err}$. This is taken into account in the solvability lemma of the entire projected system \eqref{inner_Psi_projected}--\eqref{eq:extendedouterequation} in Lemma~\ref{lem:solvabilityprojected}, where $K_{\rm err}$ is controlled using \eqref{eq:estsonthepotential}.





Notably, the potential $K_0 + K_2\xi^2$ is complex, so the problem is not self-adjoint, and one is not guaranteed existence of an orthonormal basis of eigenfunctions. 
Nevertheless, we will classify the spectrum and eigenfunctions and apply the Fredholm theory to conclude. This classification is presumably well-known; however, we are unaware of a simple reference, and so we provide the details.

It will be convenient to normalize $K_2$ by changing variables: $x = \nu \xi$ with $\nu = |K_2|^{1/4}$. Then
\begin{equation}
	\label{eq:operatortocite}
    \frac{\d^2}{\d \xi^2} - K_0 - K_2 \xi^2 = |K_2|^{1/2} \left( \frac{\d^2}{\d x^2} + (2m-1) \ee^{i\pi/4 } - \ee^{i \pi/2} x^2 \right) \, .
\end{equation}

In order to analyze it, we will consider the spectral problem for the operators 
\begin{equation}
    \label{eq:theta0def}
    L_{\zeta} \coloneqq \frac{\d^2}{\d x^2} - \ee^{2i \zeta} x^2 \, ,
\end{equation}
where $\zeta \in (-\pi/2,\pi/2)$ and then~\eqref{eq:operatortocite} corresponds to the case $\zeta = \pi/4$. Formally, using the so-called Wick rotation, i.e. the change of variables $y=\ee^{i\zeta /2}x$, we obtain
\begin{equation}
     \frac{\d^2}{\d x^2} - \ee^{2i \zeta} x^2 = \ee^{i\zeta} \left( \frac{\d^2}{\d y^2} - y^2 \right) \, .
\end{equation}
It is well known (see \cite[Section~12]{nist}) that the $L^2(\R )$ eigenfunctions of the operator $\frac{\d^2}{\d y^2} - y^2$ are, up to constant multiple, precisely the Hermite functions $G_m(y) = \ee^{-y^2/2} H_m(y)$, with corresponding eigenvalues $-(2m-1)$, $m \geq 1$. When $\zeta \in (-\pi,\pi)$, $L_{\zeta}$ therefore has $L^2$ eigenfunctions $G_m(\ee^{i\zeta/2}x)$ with corresponding eigenvalue $-(2m-1) \ee^{i\zeta}$, $m \geq 1$. In Lemma~\ref{lem:theta0spectrum}, we verify that these eigenvalues, produced by Wick rotation, are algebraically simple and exhaust the spectrum.

To begin, we prove \emph{a priori} decay estimates. It will be convenient to devise a weight function
\begin{equation}
    w(q_0,\zeta)(x) = \exp \left( \frac{q_0}{2} (\Re  \ee^{i\zeta}) x^2 \right) \, .
\end{equation}
The eigenfunctions $G_m(\ee^{i \zeta/2} x)$ decay like $w^{-1}(1,\zeta)(x)$ as $|x| \to \infty$, up to polynomial corrections.


\begin{lemma}[Decay estimates]\label{lem_decay}
Let $\zeta \in (-\pi/2,\pi/2)$, $\lambda \in \C$, and $f \in L^2$. Suppose that $\phi \in L^2$ satisfies
\begin{equation}
    \lambda \phi - \phi'' + \ee^{2i\zeta} x^2 \phi = f
\end{equation}
in the sense of distributions on $\R$. Then
\eqnb\label{decay_est_h2_and_x2_only}
\|  \phi \|_{H^2} + \| \xi^2 \phi  \|_{L^2 }  \lec (\Re \ee^{i\zeta})^{-1} \| f \|_{L^2}  + (\Re \ee^{i\zeta})^{-1} |\lambda| \, \| \phi \|_{L^2} \, .
\eqne
If additionally $q_0 \in [0,1)$ and $w f \in L^2$, then we have the Gaussian weighted estimate
\eqnb\label{gaussian_decay}
\| w(q_0,\zeta) (|\phi''| + |\xi| |\phi'| + \xi^2 |\phi|) \|_{L^2} \leq C(q_0,\zeta) \| w(q_0,\zeta)  f \|_{L^2} + C(|\lambda|,q_0,\zeta) \| \phi \|_{L^2} \, .
\eqne
\end{lemma}

\begin{proof} We first note that the equation implies that $\phi \in H^2_{\rm loc}(\R)$. We multiply the equation by $\ee^{-i\zeta} \mu  \phi^*$, where $\mu : \R \to [0,\infty)$ is compactly supported and Lipschitz, and integrate by parts:
\begin{equation}
    \label{eq:summe}
    \begin{aligned}
        (\Re \ee^{-i\zeta}) \int \mu  |\phi' |^2 + \Re \left( \ee^{-i\zeta} \int \mu ' \phi' \phi^* \right) + (\Re \ee^{i\zeta}) \int  \mu   x^2 |\phi|^2 &= \Re \left( \ee^{-i\zeta}  \int \mu  f \phi^* - \lambda \mu |\phi|^2 \right) \, .
    \end{aligned}
\end{equation}
Since $\Re \ee^{i\zeta} = \Re \ee^{-i\zeta}$, we may divide through and use basic manipulations to obtain
\begin{equation}
    \label{eq:myfirstbabyestimate}
    \begin{aligned} 
        \int \left( \mu |\phi' |^2 + \mu x^2 |\phi|^2 \right) &\leq (\Re \ee^{i\zeta})^{-1} \int \left( \mu |f| |\phi| + |\lambda| \, \mu |\phi|^2 + |\mu'| |\phi'| |\phi| \right) .
    \end{aligned}
\end{equation}
We first choose $\mu = \chi^2(x/R)$, where
\[
\chi(x) \coloneqq \begin{cases} 1 \qquad &|x| \leq 1 ,\\
2-|x|& x \in (1,2) ,\\
0& |x| \geq 2 \, ,
\end{cases}
\]
so that $|\mu'|^2 \leq 4\mu/R^2$. We split the boundary term, involving $\mu' \phi' \phi^* = \mu' \mu^{-1/2} \phi' \times \mu^{1/2} \phi^*$, using Young's inequality. We take $R\to \infty$ to obtain
\eqnb\label{base_case}
\int \left( |\phi'|^2+ x^2 |\phi |^2 \right) \lec (\Re \ee^{i\zeta})^{-1} \int \left( |f|^2 + |\lambda| |\phi |^2 \right) .
\eqne
(When $|\lambda| \ll \Re \ee^{i\zeta}$, we can additionally absorb the term $(\Re \ee^{i\zeta})^{-1} |\lambda| \int |\phi|^2$.) 

Next, we multiply the equation by $\ee^{-i\zeta} x^2 \mu \phi^*$ and integrate by parts to obtain
\begin{equation}
    \label{eq:mysecondbabyestimate}
    \begin{aligned} 
        \int \left( \mu x^2 |\phi' |^2 + \mu x^4 |\phi|^2 \right) &\leq (\Re \ee^{i\zeta})^{-1} \int \left( \mu x^2 |f| |\phi| + |\lambda| \, \mu x^2 |\phi|^2 + |\mu'| \, x^2 |\phi'| |\phi| + 2|x| |\mu| |\phi'| |\phi| \right) .
    \end{aligned}
\end{equation}
The terms containing $|\lambda| \mu x^2 |\phi|^2$ and $2|x| |\mu| |\phi'| |\phi|$ are controlled using the previous estimate~\eqref{base_case}. We take $R \to \infty$ to obtain
\begin{equation}
\int  x^2 |\phi' |^2  + \int x^4 |\phi |^2  \lec (\Re \ee^{i\zeta})^{-2}  \int \left(|f|^2 + |\lambda|^2 |\phi |^2 \right)  ,
\end{equation}
which together with the equation yields~\eqref{decay_est_h2_and_x2_only}.

With the above information, it will not be necessary to take $\mu$ compactly supported to show the Gaussian decay~\eqref{gaussian_decay}. In~\eqref{eq:myfirstbabyestimate} and~\eqref{eq:mysecondbabyestimate}, we choose $\mu = \min(M,w^2(q_0,\theta))$ with $M \geq 2$ and $q_0 \in [0,1)$. The crucial observation is that $|\mu'|^2 \leq 4 q_0^2 (\Re \ee^{i\zeta})^2 x^2 \mu^2$, which is the sharp inequality necessary to control the boundary term. 
The proof is completed by taking $M \to \infty$. \end{proof}

To formalize the spectral problem for~\eqref{eq:theta0def}, it will be convenient to recall a few notions from functional analysis.
Let $L_{\zeta} \colon D(L_{\zeta}) \subset L^2(\R;\C) \to L^2 (\R ; \C )$ be the unbounded operator defined by \eqref{eq:theta0def}, i.e.,
\begin{equation}
L_{\zeta} \Psi (x) \coloneqq \Psi''(x)  - \ee^{2i\zeta} x^2\Psi (x) 
\end{equation}
with dense domain 
\begin{equation}
    D(L_{\zeta}) \coloneqq \{ f \in L^2 : L_0f \in L^2 \} \, .
\end{equation}
We equip $D(L_{\zeta})$ with the graph norm 
\begin{equation}\label{norm_DL}
    \| \Psi \|_{D(L_{\zeta})} \coloneqq \| \Psi \|_{L^2} + \| L_{\zeta} \Psi \|_{L^2} \, ,
\end{equation}
which is equivalent to
\eqnb\label{graph_norm_equivalence}
\| \Psi \|_{D(L_{\zeta})} \cong \| \Psi \|_{H^2} + \| \xi^2 \Psi \|_{L^2} .
\eqne
The ``$\lec $'' part of the equivalence follows trivially, while the ``$\gtrsim$'' part follows from \eqref{decay_est_h2_and_x2_only}. In particular, the operators are defined on a common domain. With~\eqref{graph_norm_equivalence}, it is not difficult to verify that $L_{\zeta}$ is closed, i.e., that $D(L_{\zeta})$ is a Banach space with the graph norm: If $\{ \Psi_k \}_{k\geq 0} \subset D(L_{\zeta})$ with $\| \Psi_k - \Psi \|_{L^2} \to 0$ and $\| L\Psi_k - f \|_{L^2} \to 0$ as $k\to \infty$, then $\Psi_k$ is a Cauchy sequence in the equivalent norm, and we deduce that $L_{\zeta}\Psi_k \to L_{\zeta}\Psi$.

One can further prove the following.
\begin{lemma}\label{lem_embedding} 
The embedding $D(L_{\zeta}) \subset L^2$ is compact.
\end{lemma}
\begin{proof} The proof can be obtained by a diagonal argument using a sequence of increasing radii $R_n\to \infty$, as well as the compact embedding $H^2 (-R_n ,R_n)\subset L^2 (-R_n,R_n)$ for each $n$, and the decay at $\pm \infty$ provided by $\| \cdot \|_{D(L_{\zeta})}$. 
\end{proof}

As a consequence of Lemma~\ref{lem_embedding}, the operator $L_{\zeta}$ has compact resolvent, i.e., there exists $\lambda$ in the resolvent set $\varrho(L_{\zeta})$ (which is non-empty, as it includes $\{ \Re \lambda > 0 \}$) such that $R(\lambda,L_{\zeta}) : L^2 \to L^2$ is compact. Therefore, $\sigma(L_{\zeta})$ consists of isolated eigenvalues with finite multiplicity, see~\cite[p. 187, Theorem 6.29]{Katobook}. Thus, to classify the spectrum, it will suffice to classify eigenvalues.

\begin{lemma}[Spectrum of $L_{\zeta}$]
    \label{lem:theta0spectrum}
Let $\zeta \in (-\pi/2,\pi/2)$. Then $\sigma(L_{\zeta})$ consists of algebraically simply eigenvalues $- \ee^{i\zeta} (2m-1)$, $m \geq 1$, with corresponding eigenspaces spanned by $G_m(\ee^{i\zeta} \cdot)$.
\end{lemma}

\begin{proof}
    First, we prove algebraic simplicity by continuity. To be precise, we note that the map $\zeta \mapsto L_{\zeta} : (-\pi/2,\pi/2) \to \mathcal{B}(D(L_0);L^2)$ is continuous, and we fix any eigenvalue of $L_0$, namely, $\lambda_m = -(2m-1)$, $m\geq 1$, which is known to be algebraically simple because its geometric multiplicity is one and $L_0$ is self-adjoint. The curve $\zeta \mapsto \ee^{i\zeta} \lambda_m$ is a curve of eigenvalues of $L_{\zeta}$, with each of the eigenvalues isolated in $\sigma(L_{\zeta})$. Then~\cite[p. 212-213, Theorem 3.16]{Katobook} yields that the algebraic multiplicity of $\ee^{i\zeta}\lambda_m$ is constant along the curve.

    Second, we demonstrate that there is no other spectrum, that is, the only solutions $\phi \in L^2$ to the eigenvalue problem $\phi'' - \ee^{2i\zeta} x^2 \phi = \lambda \phi$
are trivial except when $\lambda = \ee^{i\zeta} \lambda_m$, $m \geq 1$. Due to the conjugation symmetry, it is sufficient to consider $\zeta \in [0,\pi/2)$. We appeal to ODE techniques. It will be equivalent to classify solutions to
\begin{equation}
    w''(y) - \ee^{2i\zeta} \frac{1}{4} y^2 w = \ee^{i\zeta} aw
\end{equation}
for $a\in \C$, where $\phi(x) = w(y)$, $y = \sqrt{2} x$, and $\lambda = 2 \ee^{i\zeta} a$. The solution space is the linear span of the parabolic cylinder functions $U(a,z)$, $V(a,z)$ evaluated on the line $z=\ee^{i\zeta/2} y$. We follow the conventions of~\cite[Section~12.2]{nist}, in which $U$ and $V$ are linearly independent entire solutions to $U''(z) - z^2 U(z)/4 = aU(z)$. We use the asymptotic expansions for $|z| \to \infty$ in certain sectors:
\begin{equation}
    \begin{aligned}
U(a,z) &=  \left( \frac{\ee^{-z^2/4}}{ z^{a+\frac12}} + i \frac{\sqrt{2\pi} }{\Gamma \left(a+\frac12  \right) } \ee^{-i\pi a } \ee^{z^2/4 }z^{a-\frac12 } \right) \left( 1+ O_{\delta , a}(z^{-2}) \right) \text{ for } |\mathrm{Arg}\,z| \in [\pi/4+\delta , 5\pi /4 -\delta ]\\
V(a,z) &= \sqrt{ \frac2{\pi}}     \ee^{z^2/4} z^{a-\frac12} \left( 1+ O_{\delta , a}(z^{-2}) \right) \text{ for } |\mathrm{Arg}\,z| \in [0,\pi/4-\delta ] \, ,
\end{aligned}
\end{equation}
where $\Gamma$ denotes the Gamma function, and $\delta > 0$ is arbitrarily small (see \cite[Section~12.9]{nist} and the comprehensive discussion in~\cite{temme} for details). Hence, $V$  blows up as $|z| \to \infty$ along the ray ${\rm Arg} \, z = \zeta/2$ (that is, as $x \to \infty$). Meanwhile, $U$ blows up as $|z| \to \infty$ along the ray ${\rm Arg} \, z = \zeta/2 + \pi$ (that is, as $x \to -\infty$) only when $a\ne -(m+1/2)$, $m \geq 0$, so that $1/\Gamma (a+1/2) \ne 0$. When $a = -(m+1/2)$, $m\geq 0$, $1/\Gamma (a+1/2)$ has a simple zero which cancels the blowing-up part in the asymptotic expansion for $U(a,z)$. In that setting, $U(a,z)$ is a constant multiple of a Hermite function $G_{m+1}(z/\sqrt{2})$ (see \cite[Eq. 12.7.2]{nist}). 
\end{proof}

To complete the proof of Lemma~\ref{lem:solvabilityinnerprojected}, we will appeal to the Fredholm theory. For $\lambda \in \sigma(L_{\zeta})$,~a theorem  concerning isolated eigenvalues of finite multiplicity (see~\cite[p.~239, Theorem 5.28]{Katobook}) guarantees that $\lambda - L_{\zeta}$ is Fredholm. In particular, its range is closed. We recall the following.



\begin{lemma}\label{lem_closed_range_thm} {\rm (Banach closed range theorem \cite[Section~VII.5]{yosida})}
Let $L\colon D(L) \subset \mathcal{H} \to \mathcal{H}$ denote a closed linear operator on some Hilbert space $\mathcal{H}$ such that ${\rm range}\; L$ is closed, and let $L^*$ denote its adjoint operator, defined on $D(L^*)\subset \mathcal{H}$. Then  
\[
{\rm range} \; L = (\ker L^*)^\perp\qquad \text{ and }\qquad 
{\rm range} \; L^* = (\ker L)^\perp .
\]
\end{lemma}
We can now conclude the proof of Lemma~\ref{lem:solvabilityinnerprojected}.
\begin{proof}[Proof of Lemma~\ref{lem:solvabilityinnerprojected}]
Let $\zeta \in (-\pi/2,\pi/2)$ and $\lambda = - \ee^{i\zeta} (2m-1)$, $m \geq 1$, be an eigenvalue of $L_{\zeta}$. The adjoint (see~\cite[p. 167-168]{Katobook}) of $\lambda - L_{\zeta}$ is
\begin{equation}
    (\lambda - L_{\zeta})^* = \lambda^* - L_{-\zeta} \, .
\end{equation}
We apply the Banach closed range theorem with $\mathcal{H}=L^2$ to obtain
\begin{equation}
    \label{eq:cosequenceofclosedrange}
\begin{aligned}
    \mathrm{range}\, (\lambda - L_{\zeta})  &= (\mathrm{ker}\, (\lambda^* - L_{-\zeta}) )^\perp \\
    &= ({\rm span} \, G_m(\ee^{-i\zeta/2} \cdot) )^\perp \\
    &= \left\lbrace F\in L^2 \colon \int F(y) G_m(\ee^{i\zeta/2} y) \, \d y = 0 \right\rbrace =: H \, .
    \end{aligned}
\end{equation}
Let
\begin{equation}
    V \coloneqq \left\lbrace \Psi \in D(L_{\zeta}) \colon \int_{\R } \Psi(y) (G_m(\ee^{i \zeta} y))^*\, \d y =0 \right\rbrace \, ,
\end{equation}
which is a Banach space when equipped with the norm $\| \cdot \|_{D(L_{\zeta})}$. 
Then $(\lambda - L_{\zeta})|_{V} \colon V \to H$ is a bounded bijection (by the definitions~\eqref{norm_DL} and \eqref{eq:cosequenceofclosedrange}, respectively). The bounded inverse theorem guarantees that $(\lambda - L_{\zeta}|_{V})^{-1} \colon H \to V$ is bounded, as desired. The proof is completed by choosing $\zeta = \pi/4$, undoing the change of variables $\xi = x/|K_2|^{1/4}$, and applying Lemma~\ref{lem_decay} to obtain the claim about Gaussian decay. 
\end{proof}

Finally, we have an important corollary, which will be used to analyze the reduced equation in Section~\ref{sec_solv_reduced}.

\begin{corollary}
    \label{cor:Inonzero}
For all $\zeta \in (-\pi/2,\pi/2)$ and $m \geq 1$, we have
\begin{equation}
    \label{eq:integraltocalculate}
    \int G_m(\ee^{i\zeta/2} y)^2 \, \d y \neq 0 \, .
\end{equation}
\end{corollary}
\begin{proof}
    We can calculate~\eqref{eq:integraltocalculate} by shifting the integration contour to the real axis:
    \begin{equation}
        \ee^{i\zeta/2} \int_{\R} G_m(\ee^{i\zeta/2} y)^2 \, \d y = \int_{\gamma} G_m(z)^2 \, \d z = \int_{-\infty}^{+\infty} G_m(z)^2 \, \, \d z > 0 \, ,
    \end{equation}
    where $\gamma$ is the contour parameterized by $\gamma(y) = \ee^{i\zeta/2} y$, $y \in \R$.
\end{proof}

\begin{rem}
In view of Corollary~\ref{cor:Inonzero}, we could have chosen $V = ({\rm span} \, G_m(\ee^{-i\zeta/2} \cdot))^\perp$ as the domain of $\lambda - L_{\zeta}$ and maintained invertibility in Lemma~\ref{lem:solvabilityinnerprojected}. This is more natural if one wishes to expand solutions in a non-orthogonal basis $\{ G_m(\ee^{i\zeta} \cdot) \}_{m \in \N}$. We do not address completeness of such a ``basis'' here.
\end{rem}

\subsection{The outer equation}\label{sec_solv_outer}

In this section, we develop the solvability theory for the outer equation in~\eqref{in_out_problem}. We only need to find solutions to the outer equation on $\supp\,\chi_{\out}$. This fact gives us some freedom in defining the solution $\phi_{\out}$ for $|r-r_0| \lec \ell_{\out}$. Namely, it lets us replace $k$ by 
\begin{equation}\label{ktilde_def}
    \widetilde{k} \coloneqq \widetilde{\chi}_{\rm out} k + (1-\widetilde{\chi}_{\rm out} ) n^{3/2} ,
\end{equation}
where $\widetilde{\chi}_{\rm out}\coloneqq 1 - \eta_{\widetilde{\ell}_{\rm out}}$, and $\widetilde{\ell}_{\rm out}\coloneqq\ell_{\rm out}/2$, so that $\widetilde{\chi_{\rm out}}=1$ on $\supp\,\chi_{\rm out}$. Thus we can consider, instead of the second equation in~\eqref{in_out_problem}, the problem
 \eqnb\label{outer_eq}
\upsilon'' - \widetilde{k} \upsilon =f \text{ on } \R^+ \, .
\eqne
However, this fact also creates difficulties, particularly for the regime $|r-r_0| \in (\ell_{\out} , C^{-1}n^{-1/2})$, since $\re \,k<0$ for such $r$ (see \eqref{bounds_im_re_k} below); therefore, coercivity cannot come from $\re \, k$ alone, which makes the well-posedness of the outer problem non-trivial. The key is to track $\im \, k$, which changes size from $\gec n^{3/2}$ for $|r-r_0|\geq Cn^{-3/4}$, to $\gec n^2$ for $|r-r_0|\sim C^{-1}n^{-1/2}$, and finally shrinking to $O(n^{-1})$ for $|r-r_0| =O(1)$, see \eqref{bounds_im_re_k} below. 
We need to take into account quantitative pointwise estimates on $\re\,k$, $\im\,k$, which we achieve in Lemma~\ref{Lk} below, see also Fig.~\ref{fig_reK_vs_-imK} for a sketch.

Thanks to Lemma~\ref{Lk} below, we will be able to find a linear combination~\eqref{def_h} of the real and imaginary parts of \eqnb\label{outer_basic_est}
\int_{\R^+} \left( | \upsilon' |^2 + \widetilde{k} |\upsilon |^2 \right) = \int_{\R^+} f \overline{\upsilon}
\eqne
(which is obtained by multiplying \eqref{outer_eq} by $\overline{\upsilon}$ and integrating by parts), to obtain coercivity of the problem and an \emph{a~priori} estimate, see Lemma~\ref{lem:solvabilityouter}.

Before proceeding with the rigorous version of the above discussion, we also point out that we will need holomorphic dependence of the solution $\phi_{\out}$ on $\widehat{\omega} \in B(\widehat{D}^{-1} n^{-1/2})$, where $\widehat{D}>0$ is a large constant, determined by Lemma~\ref{lem:solvabilityprojected}. 
This is needed for our application of Rouch\'e's theorem in solving the reduced equation~\eqref{reduced_eq}, which we discuss in Section~\ref{sec_solv_reduced}.

\begin{lemma}[Potential estimates]
    \label{Lk}
    Let $B \geq 1$ and suppose that $|\omega - \omega_{\rm app}| \leq B n^{-1/2}$. There exists $C \geq 1$ such that, for sufficiently large $n$ (depending on $B$), the potential $k$ satisfies the estimates
\eqnb
\label{bounds_im_re_k}
\begin{split}  \im \,k(r) & \begin{cases}
\gec n^{3/2} \hspace{3cm} & |r-r_0 | \in  [C n^{-3/4} , C^{-1} n^{-1/2}],\\
\sim   \frac{n^3(r-r_0)^2}{(1+n^2(r-r_0)^4)^2}  &|r-r_0 | \in  [ C^{-1} n^{-1/2}, C^{-1} ] ,\\
= p O(1)  &|r-r_0 | \geq  C^{-1},
\end{cases}\\
\re \,k(r) & \begin{cases}
= O( n^{3/2}) \hspace{3cm} & |r-r_0 | \in  [C n^{-3/4} , C^{-1} n^{-1/2}],\\
\gec   n^2 & |r-r_0 | \in  [ C^{-1} n^{-1/2}, C^{-1} ] ,\\
\gec pn^2  & |r-r_0 | \geq  C^{-1},
\end{cases}
\end{split}
\eqne
see Fig.~\ref{fig_reK_vs_-imK} for a sketch. In particular, there exists $A>0$ such that
\eqnb\label{def_h}
h(r) \coloneqq \re\,k(r) + A\,\im\,k(r) \gec p n^{3/2} \qquad \text{ for all }|r-r_0| \geq Cn^{-3/4}.
\eqne
\end{lemma}
\begin{center}
\includegraphics[width=0.8\textwidth]{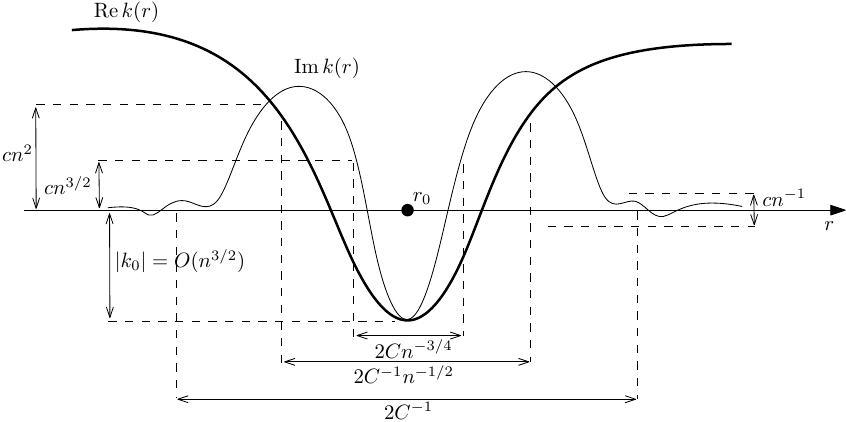}
  \captionof{figure}{The sketch of pointwise estimates on $\re\,k(r)$, $\im\,k(r)$ for $r \sim r_0$.}\label{fig_reK_vs_-imK} 
\end{center}
\begin{proof}
First, since there is no $r_1 \neq r_0$ satisfying $\Lambda(r_1) = \Lambda(r_0)$ (recall \textbf{Assumption~\ref{ass}}), we have that whenever $|r-r_0| \geq C^{-1}$, $|\gamma| \gtrsim n$, so
\[
|k-pn^2| \lec n^{-2} \, ,
\]
with implicit constant depending on $C$. For the remaining range $|r-r_0|\in [Cn^{-3/4}, C^{-1} ]$, we will use the Taylor expansion of $\gamma$ at $r_0$ to obtain
\eqnb\label{temp01}
\gamma (r) = \frac12 n \Lambda''(r_0) (r-r_0)^2 + n O(|r-r_0|^3)  - i b_0^{1/2} - O(n^{-1/2}).
\eqne

If $|r-r_0|\in [Cn^{-3/4},C^{-1}n^{-1/2}]$, then $-ib_0^{1/2}$ is the dominant term in the above expansion, and so, using the expansion $(1+x)^{-2}=1-2x + O(x^2)$ for $|x|<1$, we obtain
\[
\frac{1}{\gamma^2} = \frac{1}{b_0} \left( -1 + i \frac{\Lambda''(r_0)}{b_0^{1/2}} n (r-r_0)^2  + n O (|r-r_0|^3 + n^{-3/2}+n|r-r_0|^4)\right) ,
\]
provided $C$ is chosen sufficiently large. Thus, using the expansion $b=b_0 + O(|r-r_0|^2)$,
\[
\frac{b}{\gamma^2} = -1 + i \frac{\Lambda''(r_0)}{b_0^{1/2}} n (r-r_0)^2  + n O (|r-r_0|^3 + n^{-3/2}+n|r-r_0|^4+ n^{-1}|r-r_0|^2) \, .
\]
Hence
\[
\re\,k = pn^2 \left( 1+ \frac{b}{\gamma^2} + O(n^{-1}) \right) = O( n^{3/2}) \, ,
\]
while
\[
\im\, k = pn^2 \left( \frac{\Lambda''(r_0)}{b_0^{1/2}}n(r-r_0)^2 +   O(n^{-1/2}) \right) \gec n^{3/2} , 
\]
so that the first term is dominant, given $C$ is sufficiently large. 

If $|r-r_0|\in [C^{-1}n^{-1/2} , C^{-1}]$, then we use \eqref{temp01} and $b=b_0 + O(|r-r_0|^2)$ to obtain
\[
\left| \frac{b}{\gamma^2} \right| \leq \frac{|b|}{b_0+ \left( \frac{\Lambda''(r_0)}4 n(r-r_0)^2 \right)^2}  \leq \frac{b_0}{b_0+ \left( \frac{\Lambda''(r_0)}4 n(r-r_0)^2 \right)^2} + \underbrace{\frac{O(|r-r_0|^2)}{n^2 |r-r_0|^4}}_{\lesssim 1/n} \leq 1-\delta  
\]
for some $\delta >0$, and so 
\[
\re\,k \geq pn^2 \left(1- \left| \frac{b}{\gamma^2} \right| - O(n^{-1}) \right) \gec n^2.
\]
On the other hand, we use the fact that $\im [(a+ib)^{-2}] = -2ab/(a^2+b^2)^2$ to write \eqnb\label{temp011}
\begin{split}
\im (\gamma^{-2}) &= \frac{-2\left[\frac{\Lambda''(r_0)}{2} n (r-r_0)^2 + n O(|r-r_0|^3) + O(n^{-1/2})\right]\left[- b_0^{1/2} + O(n^{-1/2}) \right] }{\left[\left(\frac{\Lambda''(r_0)}{2} n (r-r_0)^2 + n O(|r-r_0|^3) + O(n^{-1/2}) \right)^2 + \left( b_0^{1/2} + O(n^{-1/2}) \right)^2\right]^2 } \\
&\sim \frac{n(r-r_0)^2}{(1+n^2 (r-r_0)^4 )^2} >0,
\end{split}
\eqne
as required. 
\end{proof}

Since $\ell_{\rm out}\geq C n^{-3/4}$, as in~\eqref{claim_choices_of_scales}, the estimates in Lemma~\ref{Lk} are valid on $\supp\,\chi_{\rm out}$, that is, where the outer equation is required to be satisfied. As mentioned above, we modify $k$ outside of $\supp\, \chi_{\rm out}$ according to \eqref{ktilde_def}, and, analogously to \eqref{def_h}, we obtain
\eqnb\label{outer_h_tilde}
\widetilde{h}(r) \coloneqq \re\,\widetilde{k}(r) + A\,\im\,\widetilde{k}(r) \gec p n^{3/2} \qquad \text{ for all }r>0.
\eqne
We can thus take a similar linear combination of the real and imaginary parts of the \emph{a~priori} estimate \eqref{outer_basic_est} to obtain unique solvability of the outer equation \eqref{outer_eq} in $Z$, defined in~\eqref{def_of_Z}. Namely, we proved the following. 

\begin{lemma}[Solvability of outer equation]
    \label{lem:solvabilityouter}
    Let $B \geq 1$. Suppose that $|\omega - \omega_{\rm app}| \leq B n^{-1/2}$ and $n$ is sufficiently large depending on $B$.
    For every $f\in L^2 (\R^+)$, there exists a unique solution $\upsilon \in H^1_0 (\R^+)$ of
\eqref{outer_eq},
where $\widetilde{k}$ is defined in \eqref{ktilde_def}. Moreover, 
\eqnb\label{outer_sol_norm}
n^{3/8} \| \upsilon \|_Z = \| \upsilon' \|_{L^2} + n^{3/4} \| \upsilon \|_{L^2} \lec n^{-3/4} \| f \|_{L^2} \, .
\eqne
\end{lemma}

We conclude this section with a verification that the solution $\upsilon$ to \eqref{outer_eq} is holomorphic with respect to  $\widehat{\omega}$. To this end, we extract from $\widetilde{k}$ the main part $\widetilde{k}_{\app}$, which is independent of $\widehat{\omega}$; namely, we define $\widetilde{k}_{\app}$ by \eqref{ktilde_def}, except that $\omega$ (in the definition \eqref{def_of_k} of $k$) is replaced by $\omega-\widehat{\omega}=n\Lambda_0 + i b_0^{1/2} + \mu_m$ (recall \eqref{choice_of_tilde_om_m_copy}). We can thus  rewrite \eqref{outer_eq} as
\eqnb\label{outer_eq_rewrite}
-\upsilon''+ \widetilde{k}_{\app} \upsilon = f + \left( \widetilde{k}_{\app} - \widetilde{k} \right) \upsilon.
\eqne 
The point of this is that the only dependence on $\widehat{\omega }$ occurs in the last term, which can be made small (and, hence, be absorbed) and is also holomorphic with respect to $\widehat{\omega}$, as it is a multiplication operator of $\upsilon$ by a function $\widetilde{k}_{\app} - \widetilde{k}$ holomorphic in the norm $\| \cdot \|_{L^\infty}$.\footnote{To see this, one utilizes that both $a,b$ (in the potential $k$, recall \eqref{def_of_k}) are $O(r^2)$ as $r\to 0$, which combats the apparent $O(r^{-2})$ singularity in the prefactor $p(r)$.}  The key estimate is
\begin{equation}
\|\widetilde{k}_{\rm app} - \widetilde{k} \|_{L^\infty} \leq C^{-1} n^{3/2} \, ,
\end{equation}
where $C^{-1}$ can be made arbitrarily small provided that $\widehat{\omega} \ll n^{-1/2}$ and $n$ is sufficiently large. Define the multiplication operator $K\colon Z \to L^2$, $K\upsilon \coloneqq \left( \widetilde{k}_{\app} - \widetilde{k} \right) \upsilon$, and $S\colon L^2 \to Z$ the solution operator to $-\upsilon''+\widetilde{k}_{\app} \upsilon =g$ (i.e., $Sg=\upsilon$). Then we see that any solution $\upsilon$ of \eqref{outer_eq} satisfies 
\[
\upsilon = S (f + K\upsilon ) \, ,
\] 
i.e., $\upsilon = (I-SK)^{-1} Sf$, where the invertibility of $I-SK$ follows from the facts that $\|S \|_{L^2\to N}\lesssim n^{-3/4}$ by \eqref{outer_sol_norm} and $n^{-3/4}\| K\|_{Z \to L^2} \leq C^{-1}$. This verifies that the solution in Lemma~\ref{lem:solvabilityouter} is  holomorphic with respect to  $\hat{\omega} \in D(0,\varepsilon n^{-1/2})$ for sufficiently small $\varepsilon$ and sufficiently large $n$.

\subsection{The projected system}\label{sec_solv_proj}
The projected system consists of the projected inner equation~\eqref{inner_Psi_projected}, the outer equation~\eqref{outer_eq}, and the supplementary condition~\eqref{eq:Psibarcond}:
\eqnb\label{in_out_problem_again}
\begin{cases}
&\Psi'' - (K_0 +K_2 \xi^2 )  \Psi =Q\left( \widetilde{X}_{\rm in}K_{\rm err} (W + \Psi  ) - R\phi_{\rm out} X_{\rm out}''  - 2 n^{-3/4}R\phi_{\rm out}' X_{\rm out}' \right), \\
&\int \Psi W^*= 0 ,\\
&\phi_{\rm out}'' - \widetilde{k} \phi_{\rm out} = - 2\phi_{\rm in}'\chi_{\rm in}' - \phi_{\rm in}\chi_{\rm in}'' \, . \\
\end{cases}
\eqne

\begin{lemma}[Solvability of projected system]
    \label{lem:solvabilityprojected}
There exist $D_{\rm out},\hat{D} \geq 1$ and $D_{\rm in} \geq 2D_{\rm out}$ such that for all $\widehat{ \omega} \in B(\hat{D}^{-1} n^{-1/2})$ and $n$ sufficiently large, the projected system~\eqref{in_out_problem_again} is uniquely solvable in the class $(\Psi,\phi_{\rm out})  \in Y_w \times Z$ with the choices of cut-off scales
\begin{equation}
    \label{eq:choicesofscales}
\ell_{\rm out} = D_{\rm out} n^{-3/4} \, , \quad \ell_{\rm in} \in [D_{\rm in} n^{-3/4}, D_{\rm in}^{-1} n^{-5/8}] \, .
\end{equation}
The solution satisfies
\begin{equation}
    \label{eq:estimatepsi}
    \| \Psi \|_{Y_w} \lesssim n^{1/2} |\hat{\omega}| + n^{-1/2} + (\ell_{\rm in} n^{3/4})^4 n^{-1/2} + (\ell_{\rm in} n^{3/4})^{-1} \exp\left( - C^{-1} (\ell_{\rm in} n^{3/4})^2 \right),
\end{equation}
\begin{equation}
    \label{eq:estimatephiout}
    \| \phi_{\rm out} \|_{Z} \lesssim (\ell_{\rm in} n^{3/4})^{-1} \exp\left( - C^{-1} (\ell_{\rm in} n^{3/4})^2 \right)\hspace{6.2cm}
\end{equation}
 and is  holomorphic with respect to  $\hat{\omega}$ in the above norms.
\end{lemma}

In particular,  taking $\ell_{\rm in} = D_{\rm in}^{-1} n^{-3/4+\delta}$ for $\delta \in (0,1/8]$ and $n$ sufficiently large (depending on $\delta$),~\eqref{eq:estimatepsi} and~\eqref{eq:estimatephiout} become 
\begin{equation}
    \label{eq:consequentialestimate}
    \| \Psi \|_{Y_w} \lesssim n^{1/2} |\hat{\omega}| + n^{-1/2+4\delta} \, , \quad \| \phi_{\rm out} \|_{Z} \lesssim \exp\left( - C^{-1} n^{2\delta} \right) \, .
\end{equation}


Before proving Lemma~\ref{lem:solvabilityprojected}, we define the solution operators 
\eqnb\label{sol_operators}
\begin{split} 
\mathsf{A} \colon Y_w \to Y_w \, , & \quad  \mathsf{A}F =  G \Leftrightarrow   G'' - (K_0 + K_2 \xi^2 ) G = Q ( \widetilde{X}_{\rm in} K_{\rm err} F ) \, , \; \int G W^* = 0  \\
\mathsf{B} \colon Z \to Y_w \, , & \quad  \mathsf{B}f =  G \Leftrightarrow   G'' - (K_0 + K_2 \xi^2 ) G = Q ( -R f X_{\rm out}'' - 2n^{-3/4} R f' {X}_{\rm out}''  ) \, ,  \; \int G W^* = 0  \\
\mathsf{C} \colon Y_w \to Z \, , & \quad  \mathsf{C}F =  g \Leftrightarrow   g'' - \widetilde{k} g =  -2n^{3/4} R^{-1} F' \chi_{\rm in}' - R^{-1}  F \chi_{\rm in}''  \, .
\end{split}
\eqne
We rewrite~\eqref{in_out_problem_again} in the vector form
\eqnb\label{entire_clean}
\begin{pmatrix}
I  & 0  \\ 0 & I
\end{pmatrix}
\begin{pmatrix}
\Psi \\ \phi_{\rm out}
\end{pmatrix}
-
\underbrace{ \begin{pmatrix}
\mathsf{A} & \mathsf{B} \\ \mathsf{C}  &0
\end{pmatrix}
}_{=: \mathsf{T} }
\begin{pmatrix}
\Psi \\ \phi_{\rm out}
\end{pmatrix} 
= \begin{pmatrix}
\mathsf{A} W \\ \mathsf{C} W
\end{pmatrix}\qquad \text{ in } \mathsf{N},
\eqne
where $\mathsf{N}\coloneqq Y_w \times Z $ is equipped with the  norm
\begin{equation}
    \label{eq:Nnormdef}
    \| (\Psi,\phi_{\rm out} ) \|_{\mathsf{N}} \coloneqq \| \Psi \|_{Y_w} + E \| \phi_{\rm out} \|_Z ,
\end{equation}
and $E\geq 1$ is a parameter.
The system~\eqref{entire_clean} will have a unique solution if $\| \mathsf{T} \|_{\mathsf{N}\to \mathsf{N}}  \leq 1/2$, in which case,
\begin{equation}
    \label{eq:firstestimateT}
    \| (\Psi,\phi_{\rm out}) \|_{\mathsf{N}} \leq 2 \| (\mathsf{A} W, \mathsf{C} W) \|_{\mathsf{N}} \, .
\end{equation}
Moreover, we see from \eqref{entire_clean} that the estimate on $\phi_{\rm out}$ can be improved:
\begin{equation}
    \label{eq:secondestimateT}
    \| \phi_{\rm out} \|_Z \leq \| \mathsf{C} \Psi \|_{Z} + \| \mathsf{C} W \|_{Z} \lesssim \| \mathsf{C} \|_{Y_w \to Z} \| (\mathsf{A} W, \mathsf{C} W) \|_{\mathsf{N}}  + \| \mathsf{C} W \|_{Z} \, .
\end{equation}

We now briefly explain the choices of cut-off scales~\eqref{eq:choicesofscales} and the role of the Gaussian weight (appearing in the definition~\eqref{def_of_Yw} of $Y_w$): The requirement that $\ell_{\rm in} \ll n^{-5/8}$ ensures that $\mathsf{A}$ is small because $\tilde{X}_{\rm in} K_{\rm err}$ is small pointwise. The requirement that $\ell_{\rm out} \gg n^{-3/4}$ ensures coercivity of the outer equation. It is convenient to choose $\ell_{\rm out}$ at the lower end of its range, $D_{\rm out} n^{-3/4}$, so that the operator $\mathsf{B}$ does not see the Gaussian weight, and hence $\| \mathsf{B} \|_{Z \to Z} = O(1)$. Together, these conditions account for the choice of $D_{\rm out}$. Choosing $E$ will make the coefficient associated to $\mathsf{B}$ small.

The parameter $D_{\rm in}$ enforces a separation between the inner scale $n^{-3/4}$ and the inner cut-off scale $\ell_{\rm in}$. The operator $\mathsf{C}$ benefits from the Gaussian weight when there is a separation of scales, since it only sees the `tails' of the inner function. However, the inner equation in \eqref{in_out_problem_again} depends on the error $K_{\rm err}$ of the approximation of $K$ by $K_0+K_2\xi^2$, and so our control of  $K_{\rm err}$, as well as  the operator $\mathsf{A}$, become worse when  $\ell_{\rm in}$ is increased. Since $\mathsf{A} W$ and $\mathsf{C} W$ enter the right-hand side of~\eqref{entire_clean}, the different choices of $\ell_{\rm in}$ produce a range of estimates on the solutions. With uniqueness, the solutions to~\eqref{eq_phi} will satisfy all the estimates simultaneously, with the worse of the two estimates holding in the overlap region.\footnote{The Gaussian weight is not technically necessary to prove existence in this problem, but it yields more precise estimates.}

We now prove Lemma~\ref{lem:solvabilityprojected} by making the above discussion rigorous. Before we proceed, we note the following estimate on $K_{\rm err}$:
 \begin{equation}
        \label{eq:estsonthepotential}
        K_{\rm err} = in^{1/2} \ho \frac{2p_0}{{b_0}^{1/2}} + ( 1 + |\xi|^4 ) O(n^{-1/2}) \qquad \text{ as }n\to \infty \, ,
    \end{equation}
    for $|\hat{\omega}| \leq n^{-1/2}$ and $|\xi| \leq n^{1/4}$, see \eqref{remainder_11_app} for a justification.



\begin{proof}[Proof of Lemma~\ref{lem:solvabilityprojected}] We have
\begin{equation}
    \| \mathsf{T}(F,f) \|_{N} \leq (\| \mathsf{A} \|_{Y_w \to Y_w} + E^{-1} \| \mathsf{B} \|_{Z \to Y_w} + E \| \mathsf{C} \|_{Y_w \to Z} ) \| (F,f) \|_{N} \, .
\end{equation}
The parameter $D_{\rm out} \gg 1$ is fixed first to ensure solvability of the outer equation in Lemma~\ref{lem:solvabilityouter}. 

For $\| \mathsf{A} \|_{Y_w \to Y_w}$, the solvability of the projected inner equation in Lemma~\ref{lem:solvabilityinnerprojected} grants us
\eqnb\label{temp_est_A}
\begin{split}
\| G \|_{Y_w} &\lec \|  Q (\widetilde{X}_{\rm in} K_{\rm err} F) \|_{L^2_w} \lesssim \|  \widetilde{X}_{\rm in} K_{\rm err} F \|_{L^2_w} \lec (n^{1/2} |\hat{\omega}| + n^{-1/2} + \ell_{\rm in}^4 n^{5/2}) \| F \|_{L^2_w}  ,
\end{split}
\eqne
where we used~\eqref{eq:estsonthepotential} to estimate $\tilde{X}_{\rm in} K_{\rm err}$ pointwise 
on ${\rm supp} \, \tilde{X}_{\rm in} \subset B(n^{1/4})$. Based on~\eqref{temp_est_A}, we impose $\hat{D} \gg 1$ and the restriction $\ell_{\rm in} \leq D_{\rm in}^{-1} n^{-5/8}$ with $D_{\rm in} \gg 1$ to ensure that $\| \mathsf{A} \|_{Y_w \to Y_w} \leq 1/6$.



For $\| \mathsf{B} \|_{Z \to Y_w}$, Lemma~\ref{lem:solvabilityinnerprojected} grants us
\begin{equation}\label{temp_est_B}
    \begin{split}
\| G \|_{Y_w} &\lesssim \| Q  (  -Rf X_{\rm out}''- n^{-3/4} Rf' X_{\rm out}') \|_{L^2_w}\\
&\lec  n^{3/8} \left( \| f \| (\ell_{\rm out}n^{3/4})^{-2} + n^{-3/4} \| f' \| (\ell_{\rm out}n^{3/4})^{-1} \right) \\
&\lec  \| f \|_Z \, ,
\end{split}
\end{equation}
using  that $l_{\rm out} = D_{\rm out} n^{-3/4}$. We choose $E \gg 1$ to ensure that $E^{-1} \| \mathsf{B} \|_{Z \to Y_w} \leq 1/6$.
 
For $\| \mathsf{C} \|_{Y \to Z}$, Lemma~\ref{lem:solvabilityouter} grants us
\begin{equation}
\label{temp_est_C}
    \begin{aligned}
\| g \|_Z &\lec n^{-9/8} \|  2 n^{3/4} R^{-1}F'\chi_{\rm in}' + R^{-1} F \chi_{\rm in}''  \|\\
&\lec n^{-9/8} \left( n^{3/8} \ell_{\rm in}^{-1}  \| F' \|_{L^2(\supp X_{\rm in}')} +  \ell_{\rm in}^{-2} n^{-3/8} \| F \|_{L^2(\supp X_{\rm in}'')} \right)\\
&\lec  ( \ell_{\rm in}n^{3/4}) ^{-1}  \| F' \|_{L^2(\supp X_{\rm in}')} +  (\ell_{\rm in}n^{3/4})^{-2}\| F \|_{L^2(\supp X_{\rm in}'')} \\
&\lesssim (\ell_{\rm in} n^{3/4})^{-1} \exp\left( - C^{-1} (\ell_{\rm in} n^{3/4})^2 \right)\| F \|_{Y_w} \, .
    \end{aligned}
\end{equation}
The right-hand side is $\lesssim D_{\rm in}^{-1} \exp \left( - C^{-1} D_{\rm in}^2 \right) \| F \|_{Y_w}$. Therefore, we may choose $D_{\rm in} \gg 1$ to ensure that $E \| \mathsf{C} \|_{Y_w \to Z} \leq 1/6$.

Altogether, we complete the proof by observing that
\begin{equation}
    \| \mathsf{T}(F,f) \|_{\mathsf{N}}  \leq \frac{1}{2} \| (F,f) \|_{\mathsf{N}} 
\end{equation}
and applying~\eqref{eq:firstestimateT} and~\eqref{eq:secondestimateT}.  The operators $\mathsf{A}$ and $\mathsf{C}$ are  holomorphic with respect to $\hat{\omega}$ (through $K_{\rm err}$ and $\widetilde{k}$, respectively), and therefore the solutions to the projected system~\eqref{in_out_problem_again}, obtained via the Neumann series, are also holomorphic with respect to $\hat{\omega}$.
\end{proof}

It will be convenient to build two further conditions into the choice of $D_{\rm in}$ in Lemma~\ref{lem:solvabilityprojected}. Observe that $I \coloneqq \int W^2$ satisfies $0 < |I| \leq 1$ (see Corollary~\ref{cor:Inonzero}). We can thus assume that $D_{\rm in}$ is also large enough so that
\begin{equation}
    \label{eq:Imconditions}
    \left| I - \int \tilde{X}_{\rm in} W^2 \right| \leq |I|/4 \, , \quad \| \Psi \|_{L^2} \leq |I|/4 \, .
\end{equation}

\subsection{The one-dimensional reduced equation}\label{sec_solv_reduced}

In order to complete the proof of existence part of Theorem~\ref{thm_main}, it remains to solve the reduced equation~\eqref{reduced_eq}, that is, to find $\widehat{\omega} = \widehat{\omega }(n)$ 
such that
\begin{equation}\label{reduced_eq_again1}
    \int_{\R }  \left( \widetilde{X}_{\rm in}K_{\rm err} (W + \Psi ) - R\phi_{\rm out} X_{\rm out}''  - 2n^{-3/4}R\phi_{\rm out}' X_{\rm out}' \right) W = 0 \, .
\end{equation}

\begin{lemma}[Solvability of reduced equation]\label{L_reduced}
In the notation of Lemma~\ref{lem:solvabilityprojected}, for $n$ sufficiently large, there exists a unique $\widehat{\omega} \in B(\hat{D}^{-1} n^{-1/2})$ such that the reduced equation~\eqref{reduced_eq_again1} is satisfied. The eigenvalue correction $\widehat{\omega}$ satisfies the estimate
\begin{equation}
  \label{eq:omegaest}
    |\widehat{\omega}| \lesssim n^{-1} (1+ \ell_{\rm in}^4 n^3) + n^{-1/2} (\ell_{\rm in} n^{3/4})^{-1} \exp\left( - C^{-1} (\ell_{\rm in} n^{3/4})^2 \right) \, . 
\end{equation}
\end{lemma}

In particular, with the choice $\ell_{\rm in} = D_{\rm in}^{-1} n^{-3/4+\delta}$ for $\delta \in (0,1/8]$ (as specified in Lemma~\ref{lem:solvabilityprojected}) and $n$ sufficiently large depending on $\delta$, we see that
\begin{equation}
  |\widehat{\omega}| \lesssim n^{-1+4\delta} \, ,
\end{equation}
which, according to~\eqref{eq:consequentialestimate}, implies the estimate
\begin{equation}
    \| \Psi \|_{Y_w} \lesssim n^{-1/2+4\delta} \, ,
\end{equation}
giving \eqref{apriori}, as required.

\begin{proof}[Proof of Lemma~\ref{L_reduced}.] We will apply Rouch{\'e}'s theorem.\footnote{An alternative approach involving the implicit function theorem could also be suitable. However, making use of Rouch{\'e}'s theorem, we avoid the need to compute the derivative with respect to $\widehat{\omega}$.} The key point is to extract the main dependence of $K_{\rm err}$ on $\widehat{\omega}$. Recall the characterization~\eqref{eq:estsonthepotential},
\begin{equation} 
  \label{eq:kappaextract}
    K_{\rm err} (\xi ) = in^{1/2} \ho \frac{2p_0}{b_0^{1/2}} + \underbrace{(1+\xi^4) O(n^{-1/2} )}_{=: \kappa(\xi)} \, .
\end{equation}
We can thus rewrite the reduced equation \eqref{reduced_eq} as
\begin{equation}\label{fgh}
    \begin{aligned}
         &0 = \underbrace{i  \widehat \omega \frac{2p_0}{b_0^{1/2}} \int_{\R}  \widetilde{X}_{\rm in}W^2  }_{=: f(\widehat \omega )}  + \underbrace{ i  \widehat \omega \frac{2p_0}{b_0^{1/2}} \int_{\R}  \widetilde{X}_{\rm in}\Psi  W }_{=: g(\widehat \omega)} \\
    &\qquad + \underbrace{ n^{-1/2} \int_{\R } \left( \kappa \widetilde{X}_{\rm in} (W+ \Psi )- R \phi_{\rm out} X_{\rm out}''- 2n^{-3/4} R \phi_{\rm out}' X_{\rm out}' \right) W }_{=: h(\widehat \omega )} \, .
    \end{aligned}
\end{equation}
   
Each of the functions $f$, $g$, and $h$ is holomorphic in $\widehat{\omega} \in B(\hat{D}^{-1} n^{-1/2})$, recall the discussion below Lemma~\ref{lem:solvabilityouter}. Evidently, $f$ has a single zero in $\C$, and we wish to demonstrate that $f+g+h$ has a single zero in $\widehat{\omega} \in B(\hat{D}^{-1} n^{-1/2})$. The function $g$ should be regarded as a lower order term, while the function $h$ should be regarded as a forcing term (however, it also has mild $\hat{\omega}$-dependence).

To apply Rouch{\'e}'s theorem, we must estimate $|f|$ from below and $|g|+|h|$ from above on the boundary of a disc $B(R_0)$ for some $R_0>0$. By~\eqref{eq:Imconditions}, we have
\begin{equation}
    |f| \geq R_0 \frac{2p_0 }{ b_0^{1/2}} \frac{3}{4} |I| \, , \quad |g| \leq R_0 \frac{2p_0 }{ b_0^{1/2}} \frac{1}{4} |I| \, .
\end{equation}
By applying the estimates for the solutions to the projected system, we have
\begin{equation}
    \label{eq:hest}
\begin{aligned}
        |h| &= n^{-1/2} \left| \int_{\R } \left( \kappa \widetilde{X}_{\rm in} (W+ \Psi )- R \phi_{\rm out} X_{\rm out}''- 2n^{-3/4} R \phi_{\rm out}' X_{\rm out}' \right) W \right| \\
        &\leq n^{-1/2} \| \kappa \|_{L^\infty({\rm supp} \, (\tilde{X}_{\rm in}))} (1 + \| \Psi \|_{L^2}) + 2n^{-1/2} \| \phi_{\rm out} \|_Z \\
        &\leq C_0 n^{-1} (1+ \ell_{\rm in}^4 n^3) + C_0 n^{-1/2} (\ell_{\rm in} n^{3/4})^{-1} \exp\left( - C_0^{-1} (\ell_{\rm in} n^{3/4})^2 \right) \, ,
\end{aligned}
\end{equation}
Hence, any
\begin{equation}
  \label{eq:restrictiononR0}
    R_0 \in \left[ \frac{4 b_0^{1/2}}{p_0 |I|} \times (\text{RHS of Eq.~\ref{eq:hest}}) \, , \hat{D}^{-1} n^{-1/2} \right]
\end{equation}
suffices. (Notably, the interval is nonempty when $D_{\rm in}$ is sufficiently large.) To obtain the estimate~\eqref{eq:omegaest}, we choose $R_0$ to be the smallest value allowed by~\eqref{eq:restrictiononR0}. 
\end{proof}

This concludes the existence part of Theorem~\ref{thm_main}. 



\subsection{Uniqueness}\label{sec_uniqueness}

We now demonstrate the uniqueness part of Theorem~\ref{thm_main}, for which we use a gluing procedure that is somewhat dual to the gluing procedure for existence:\footnote{The cutting operator $\phi \mapsto (\phi \chi_{\rm in}, \phi \chi_{\rm out})$ in the uniqueness proof is the $L^2$ adjoint of the operator $(\phi_{\rm in},\phi_{\rm out}) \mapsto \phi_{\rm in} \chi_{\rm in} + \phi_{\rm out} \chi_{\rm out}$, which sums the components into the ansatz for the existence proof. For linear gluing problems involving self-adjoint operators, one expects that the gluing systems for existence and uniqueness will be formal adjoints. For our purposes, it is enough to note that the roles of the cut-offs $\chi_{\rm in}$, $\chi_{\rm out}$ appearing in the inner-outer system \eqref{in_out_problem} are swapped.}
 Suppose that $\phi \in H^1_0(\R^+)$ is a solution to~\eqref{eq_phi} with $\omega = n\Lambda (r_0) + ib_0^{1/2} + \mu$ for some $\mu \in B(C_M n^{-1/2})$. We decompose
\begin{equation}
\label{reverse_decomp}
\phi_{\rm in} \coloneqq \phi \chi_{\rm in} \, , \quad \phi_{\rm out} \coloneqq \phi \chi_{\rm out} \, .
\end{equation}
We compute equations for $\phi_{\rm in}$ and $\phi_{\rm out}$. For example, the inner equation is
\begin{equation}
\phi_{\rm in}'' + k\phi_{\rm in} = 2\phi' \chi_{\rm in}' + \phi \chi_{\rm in}'' = 2\phi_{\rm out}' \chi_{\rm in}' + \phi_{\rm out} \chi_{\rm in}'' \text{ on } {\rm supp} \, \chi_{\rm in} \, , 
\end{equation}
where $\phi = \phi_{\rm out}$ on ${\rm supp} \, \chi_{\rm}'$ due to the overlap in the cut-off. After the analogous computation for $\phi_{\rm out}$, we obtain a system of equations
\begin{equation}
\label{in_out_problem_uniqueness}
\begin{cases}
&\phi_{\rm in}'' - k \phi_{\rm in}  = 2\phi_{\rm out}' \chi_{\rm in}' + \phi_{\rm out} \chi_{\rm in}'' \text{ on } {\rm supp} \, \chi_{\rm in} \\
&\phi_{\rm out}'' - k \phi_{\rm out} = 2\phi_{\rm in}'\chi_{\rm out}' + \phi_{\rm in}\chi_{\rm out}''  \text{ on } {\rm supp} \, \chi_{\rm out}  \, ,
\end{cases}
\end{equation}
nearly identical to~\eqref{in_out_problem}; notably, the cutoffs $\chi_{\rm in}$ and $\chi_{\rm out}$ in the boundary terms are swapped. Uniqueness for solutions to~\eqref{in_out_problem_uniqueness}, up to multiplication by constants, will imply uniqueness for solutions to~\eqref{eq_phi}.


\begin{proof}[Proof of Theorem~\ref{thm_main}: Uniqueness] The proof is divided into two cases. When the eigenvalue is close to an approximate eigenvalue, we proceed as in the existence part, applying the Lyapunov-Schmidt reduction and uniquely solving the equations, so that $\phi$ is a constant multiple of a constructed solution. When the eigenvalue is not close to one of the approximate eigenvalues, the inner operator will be invertible without projecting the equation, and we demonstrate that $\phi = 0$.

One difference is that we do not use Gaussian weights. Moreover, it will be convenient to choose
\begin{equation}
     \ell_{\rm in} \coloneqq 2 \ell_{\rm out} \, , \quad \ell_{\rm out} \coloneqq D n^{-3/4}
\end{equation}
for sufficiently large $D>0$.\\

\noindent \textbf{Case 1}. There exists $m=1,\ldots , M$ such that $\mu = \mu_m + \hat{\omega}$ satisfies $\hat{\omega} \in B(\hat{D}^{-1} n^{-1/2})$.\\

The constant $\hat{D}$ (which depends on $M$) will be determined below~\eqref{ests_ABD_gluing}, and that various constants below may depend on $M$.

We proceed as in the proof of existence. For brevity, we will omit writing the index $m$. We let $k_0$, $k_2$ be defined as in \eqref{k0_and_k2} with $\mu$ replaced by $\mu_m$. Let $\Psi = \Phi_{\rm in} - P \Phi_{\rm in}$, where $PF = (\int FW^*) W$ denotes the $L^2$-orthogonal projection onto the Weber function $W$, so that $P\Psi = 0$. After multiplying $\phi$ by a suitable constant, we may assume that $\Phi_{\rm in} = \iota W + \Psi$ with $\iota \in \{ 0,1\}$. Suppose $\iota = 1$. (When $\iota = 0$ it is not even necessary to consider the one-dimensional reduced equation.)  
We define analogues of the operators in~\eqref{sol_operators}~by 
\eqnb\label{tildesol_operators100}
\begin{split} 
\Aa \colon Y \to Y,& \quad  \Aa F =  G \Leftrightarrow   G'' - (K_0 + K_2 \xi^2 ) G = Q( \widetilde{X}_{\rm in} K_{\rm err} F ) \, , \;  \int G W^* = 0 \\
\Bb \colon Z \to Y,& \quad  \Bb f =  G \Leftrightarrow   G'' - (K_0 + K_2 \xi^2 ) G = Q( R f X_{\rm in}'' + 2n^{-3/4} R f' {X}_{\rm in}'' ) \, , \;  \int G W^* = 0 \\
\Dd \colon Y \to Z,& \quad  \Dd F =  g \Leftrightarrow   g'' - \widetilde{k} g =  2n^{3/4} R^{-1} F' \chi_{\rm out}' + R^{-1}  F \chi_{\rm out}'' \, .
\end{split}
\eqne
Recall~\eqref{def_of_Y}, \eqref{def_of_Z} for the definitions of $Y,Z$. Note that we do not use the weighted space $Y_w$ here, as Gaussian decay is not necessary for uniqueness.
We note that the roles of $\chi_{\rm in}$ and $\chi_{\rm out}$ on the right-hand sides in~\eqref{tildesol_operators100} are switched compared to~\eqref{sol_operators}. Then $(\Psi,\phi_{\rm out})$ satisfies the projected system
\eqnb\label{entire_clean_reverse2}
\begin{pmatrix}
I  & 0  \\ 0 & I
\end{pmatrix}
\begin{pmatrix}
\Psi \\ \phi_{\rm out}
\end{pmatrix}
-
\underbrace{\begin{pmatrix}
\Aa & \Bb \\ \Dd  &0
\end{pmatrix}}_{=: \mathsf{T}}
\begin{pmatrix}
\Psi  \\ \phi_{\rm out}
\end{pmatrix}
= \begin{pmatrix}
 \Aa W \\ \Dd W
\end{pmatrix}\qquad \text{ in } Y\times Z \, .
\eqne
The one-dimensional reduced equation
\eqnb\label{reduced_reverse}
\int_{\R }  \left( \widetilde{X}_{\rm in}K_{\rm err} (W + \Psi ) + R\phi_{\rm out} X_{\rm in }''  + 2n^{-3/4}R\phi_{\rm out}' X_{\rm in }' \right) W =0 \, ,
\eqne
is also satisfied; it is obtained by integrating the inner equation against $W$. 

Similarly to~\eqref{temp_est_A}-\eqref{temp_est_C}, we obtain
\begin{equation}
\label{ests_ABD_gluing}
\| \Aa  \|_{Y \to Y} \leq C \hat{D}^{-1} + C(D) n^{-1/2} \, , \qquad \| \Bb  \|_{Z \to Y} \lec 1 \, , \qquad \| \Dd \|_{Y \to Z} \lec D^{-1} \, .
\end{equation}
We define the auxiliary space $\mathsf{N} \coloneqq Y \times Z$ and norm $\| (F,f) \|_{\mathsf{N}} \coloneqq \| F \|_{Y} + E \| f \|_{Z}$, as in~\eqref{eq:Nnormdef} with $Y$ replacing $Y_w$. The parameter $E$ is chosen so that $E^{-1} \| \Bb  \|_{Z \to Y} \leq 1/6$.\footnote{Alternatively, unlike in the existence proof, one actually gets $D^{-1}$ in this term.} Then the restriction $D \gg 1$ is applied so that $E \| \Dd \|_{Y \to Z} \leq 1/6$. The restrictions that $n \gg 1$ (depending on $D$) and $\hat{D} \gg 1$ (independently of $D$), guarantee that $\| \Aa  \|_{Y \to Y} \leq 1/6$. Hence, since the operator norm  $\| \mathsf{T} \|_{\mathsf{N} \to \mathsf{N}} \leq 1/2$, the projected system~\eqref{entire_clean_reverse2} is uniquely solvable, the solution is holomorphic with respect to  $\hat{\omega}$, and
\begin{equation}
    \| \Psi \|_{Y} + \| \phi_{\rm out} \|_Z \lesssim D^{-1} + \hat{D}^{-1} + C(D) n^{-1/2} \, .
\end{equation}

We next analyze the reduced equation~\eqref{reduced_reverse}. Recall the analogy with~\eqref{fgh}. The main point is to estimate the holomorphic function
\begin{equation}
    h(\hat{\omega}) \coloneqq n^{-1/2} \int_{\R } \left( \kappa \widetilde{X}_{\rm in} (W+ \Psi ) +  R \phi_{\rm out} X_{\rm out}'' + 2n^{-3/4} R \phi_{\rm out}' X_{\rm out}' \right) W
\end{equation}
from above. A similar calculation as in~\eqref{eq:hest} yields 
\begin{equation}\label{eee3}
    |h| \leq C(D) n^{-1} (1+ C(D) n^{-1/2}) + C_0 D^{-1} n^{-1/2} (D^{-1} + \hat{D}^{-1} + C(D) n^{-1/2}) \, ,
\end{equation}
where in the second term we exploited the cut-offs: $|X'_{\rm out}|, |X''_{\rm out}| \lesssim D^{-1}$. As before, we may choose $R_0 = \hat{D}^{-1} n^{-1/2}$ and close Rouch{\'e}'s theorem in the disc $D(R_0)$. Therefore, there is a unique choice of $\hat{\omega}$ and $(\Psi,\phi_{\rm out})$. A straightforward computation shows that the solution constructed in Sections~\ref{sec_solv_inner}--\ref{sec_solv_reduced} also solves \eqref{entire_clean_reverse1}, and so it must equal to $\phi$, as desired. \\

\noindent \textbf{Case 2}. $\mu \in \overline{B(C_M n^{-1/2})} \setminus \cup_{m=1}^M B(\mu_m,\widehat{D}^{-1} n^{-1/2})$. \\

Similarly to \eqref{k0_and_k2}, we set 
\begin{equation}
    k_0 \coloneqq -i \mu n^2 \frac{2p_0}{b_0^{1/2}}, \quad \text{ and }\quad k_2 \coloneqq in^3 \frac{p_0\Lambda'' (r_0)}{b_0^{1/2}} \, .
\end{equation}
However, in contrast to \eqref{k0_and_k2}, $k_0$ is defined using $\mu$, not an approximate eigenvalue. Since $\mu \ne \mu_m$ for any $m=1,\ldots , M$, the classification in Lemma~\ref{lem:theta0spectrum} yields that
\begin{equation}
G'' - (K_0+ K_2 \xi^2 ) G = F
\end{equation}
is uniquely solvable for every $F \in L^2$. Moreover, we have the uniform estimate
\begin{equation}
    \label{eq:UniformGest}
    \left\| G \right\|_{Y} \leq C(\hat{D}) \| F \|_{L^2}
\end{equation}
since $\mu \mapsto \d^2/\d\xi^2 - (K_0(\mu) + K_2\xi^2)$ is a continuous family of operators  defined on a compact set. (Notice the dependence on $\hat{D}$, which is considered to be fixed.) The outer equation remains solvable with the uniform estimates demonstrated in Lemma~\ref{lem:solvabilityouter}.

We have the error estimate
\eqnb
K_{\rm err} (\xi ) = K (\xi ) - (K_0+ K_2 \xi^2 ) = (1+|\xi|^4  ) O(n^{-1/2}) \, ,
\eqne
as compared to~\eqref{eq:estsonthepotential}, see Appendix~\ref{app_exp_of_k}. We define analogues of the operators \eqref{sol_operators} by 
\eqnb
\begin{split} 
\Aa \colon Y \to Y,& \quad  \Aa F =  G \Leftrightarrow   G'' - (K_0 + K_2 \xi^2 ) G =  \widetilde{X}_{\rm in} K_{\rm err} F \\
\Bb \colon Z \to Y,& \quad  \Bb f =  G \Leftrightarrow   G'' - (K_0 + K_2 \xi^2 ) G =  R f X_{\rm in}'' + 2n^{-3/4} R f' {X}_{\rm in}''  \\
\Dd \colon Y \to Z,& \quad  \Dd F =  g \Leftrightarrow   g'' - \widetilde{k} g =  2n^{3/4} R^{-1} F' \chi_{\rm out}' + R^{-1}  F \chi_{\rm out}'' \, .
\end{split}
\eqne
The equation~\eqref{eq_phi} is rewritten as
\eqnb\label{entire_clean_reverse1}
\begin{pmatrix}
I  & 0  \\ 0 & I
\end{pmatrix}
\begin{pmatrix}
{\Phi}_{\rm in} \\ \phi_{\rm out}
\end{pmatrix}
-
\underbrace{\begin{pmatrix}
\Aa & \Bb \\ \Dd  &0
\end{pmatrix}}_{=: \mathsf{T}}
\begin{pmatrix}
{\Phi}_{\rm in} \\ \phi_{\rm out}
\end{pmatrix}
= \begin{pmatrix}
0 \\ 0
\end{pmatrix}\qquad \text{ in } Y\times Z \, .
\eqne
We obtain estimates akin to~\eqref{ests_ABD_gluing},
\begin{equation}
\label{ests_ABD_gluing3}
\| \Aa  \|_{Y \to Y} \leq C(\hat{D}) C(D) n^{-1/2} \, , \qquad \| \Bb  \|_{Z \to Y} \leq C(\hat{D}) \, , \qquad \| \Dd \|_{Y \to Z} \lec D^{-1} \, .
\end{equation}
which guarantee that $I - \mathsf{T}$ is invertible on $Y \times Z$. Consequently, $\phi_{\rm out} =0$ and $\Phi_{\rm in}=0$, so there are no nontrivial solutions when $\mu$ belongs to Case~2. \end{proof}


\section{Proof of Theorem~\ref{thm:nonlinearinstab}}\label{sec_pf_thm2}

In this section, we prove Theorem~\ref{thm:nonlinearinstab}, namely, that for vortex columns~\eqref{vortex_col_repeat} satisfying the mild decay assumption~\eqref{eq:decayassumption}, modal linear instability implies non-linear instability. We assume the existence of an unstable eigenvalue $\lambda$ ($a \coloneqq \Re \lambda > 0$), with $L^Q$ eigenfunction $\eta$ in vorticity form. The strategy is to construct a nonlinear trajectory which tracks the growing linear solution $\Re (\ee^{t\lambda} \eta)$ \emph{backward in time}, that is, to construct a single trajectory on the unstable manifold~\cite{grenier,ABCD}. 

Nonlinear instability arguments often begin with linear semigroup estimates. For vortex columns, the main difficulty is that the abstract spectral theory only guarantees $L^Q$ semigroup estimates \emph{Fourier mode-by-Fourier mode, not uniformly in the modes}. We circumvent this difficulty by (i) developing high-order approximations which decay backward-in-time at any prescribed rate, and (ii) controlling the growth of the remainder by na{\"i}ve methods (energy estimates). This method is due to Grenier~\cite{grenier}. The key point is that the construction of the approximate solution only requires the linear semigroup estimates for finitely many iterations depending on the desired order of the approximation.



We proceed in the following way:
\begin{itemize}
    \item Section~\ref{sec_bs}. We analyze the Biot-Savart law on $\R^2 \times \T$ for $\omega \in L^Q_\sigma$. 
    \item Section~\ref{sec_prop_LL}. We analyze the spectrum of the linearized Euler operator $\LL$, defined in~\eqref{def_LL}, mode-by-mode and prove smoothness of unstable eigenfunctions.
    \item Section~\ref{sec_semigroup}. We prove estimates on the linear semigroup $\ee^{tL_{n,\alpha}}$ acting in the $(n,\alpha)$ mode.
    \item Section~\ref{sec_iteration}. We construct high-order approximate solutions by Picard iteration and close the nonlinear argument around them.
\end{itemize}






\subsection{The Biot-Savart law}\label{sec_bs}

The Biot-Savart law $u={\rm BS}[\omega ]$ requires special attention on the slab domain $\R^2 \times \T$ since, when $\omega$ is independent of $z$, the problem is two dimensional.\footnote{The problem is more precisely \emph{two-and-a-half dimensional}. The two-dimensional Biot-Savart law is well defined on $m$-fold rotationally symmetric vorticities in $L^2$, $m \geq 2$, as exploited in~\cite{vishik_2,ABCD}. The two-and-a-half dimensional case also contains vertical shears like $u = f(x') e_z$,  $x' \in \R^2$, with vorticity $\omega = (- \nabla^\perp_{x'} f, 0) =: (\omega',0)$ and $u = (- \Delta_{x'})^{-1} (0, \nabla^\perp_{x'} \cdot \omega')$.} The two-dimensional Biot-Savart law maps $L^2 \to \dot H^1$ and \emph{a priori} is only well defined up to constants unless additional decay is assumed; this is precisely the reason for our assumption $\omega \in L^Q$, $Q \in (1,2)$. 

\begin{lemma}[Biot-Savart law on $\R^2 \times \T$]
    \label{lem:biotsavartlaw}
Let $p \in (1,2)$. For $\omega \in L^p_\sigma(\R^2 \times \T)$, there exists a Biot-Savart operator $u = {\rm BS}[\omega] \in L^1_{\rm loc}(\R^2 \times \T)$ with the following properties:

Set $\omega = P_0 \omega + P_{\neq 0} \omega$, where $P_0 \omega(x) \coloneqq  \int_\T \omega(x',y_3) \, \d y_3$. Then
\begin{equation}
	\label{eq:Pneqzeroest}
    \| {\rm BS} [P_{\neq 0} \omega] \|_{W^{1,q}} \lesssim \| \omega \|_{L^q} \, , \quad \text{ for }q \in (1,2] \, ,
\end{equation}
\begin{equation}
    \label{eq:Lpestimatein2d}
    \| {\rm BS} [P_0 \omega] \|_{\dot W^{1,p} \cap L^{p^*_2}} \lesssim \| \omega \|_{L^p} \, ,
\end{equation}
where $p_2^* \coloneqq 2p/(2-p)$ denotes the $2$D Sobolev exponent.
\end{lemma}
Recall \eqref{def_Lpsigma} that the subscript ``$\sigma$'' in $L^p_\sigma$ denotes weakly divergence-free vector fields.
\begin{proof}
By density, it is enough to consider $\omega \in C^\infty_0(\R^2 \times \T)$. We first show \eqref{eq:Pneqzeroest}.

We consider $P_{\neq 0} \omega$, which has zero mean in the vertical variable, so we may compute $\psi = (-\Delta)^{-1} P_{\neq 0} \omega$ such that $\psi$ also has zero mean in the vertical variable  and decays in the horizontal variable. Namely we define\footnote{$\psi$ is computed via multiplication by $1/|(\xi',\alpha)|^2$ on the Fourier side, where $\xi' \in \R^2$ and $\alpha \in 2\pi\Z$. Because $P_{\neq 0} \omega$ has zero mean in the vertical variable, its Fourier transform is zero whenever $\alpha = 0$, and therefore its Fourier support is away from the origin.} $\psi$ as the solution to the Poisson problem
\begin{equation}
\label{psi_eq_bs}
\begin{aligned}
&-\Delta \psi = P_{\ne 0} \omega \quad\text{ in } \R^2\times \T \qquad \text{ with } \\
&\quad \qquad\int_{\T} \psi (x',y_3)\,\d y_3=0  \text{ and } |\psi (y',x_3)|\to 0\text{ as } |y'|\to \infty  ,  \quad \text{ for all }x\, .
\end{aligned}
\end{equation}

To prove~\eqref{eq:Pneqzeroest}, it will be enough to obtain
\begin{equation}
	\label{eq:psiinlq}
\| \psi \|_{L^q(\R^2 \times \T)} \lesssim \| P_{\neq 0} \omega \|_{L^q(\R^2 \times \T)} \, .
\end{equation}
Indeed, given \eqref{eq:psiinlq}, we can apply the local regularity theory for $-\Delta \psi = P_{\neq 0} \omega$ to demonstrate
\begin{equation}
	\label{eq:esttosumoveracoveringofballs}
\| \nabla^2 \psi \|_{L^q(B_1(x'_0) \times \T)} \lesssim \| P_{\neq 0} \omega \|_{L^q(B_2(x'_0) \times \T)} + \| \psi \|_{L^q(B_2(x'_0) \times \T)} \, , \quad \forall x_0' \in \R^2 \, .
\end{equation}
We sum~\eqref{eq:esttosumoveracoveringofballs} in $\ell^q$ over a covering of $\R^2\times \T$ by balls of radius $1$ to obtain the global estimate
\begin{equation}
\| \psi \|_{W^{2,q}(\R^2 \times \T)} \lesssim \| P_{\neq 0} \omega \|_{L^q(\R^2 \times \T)} \, .
\end{equation}
To obtain the Biot-Savart law $u={\rm BS}[\omega ]$, we set $u \coloneqq {\rm curl} \, \psi$.

In order to prove~\eqref{eq:psiinlq}, we  first  multiply \eqref{psi_eq_bs} by $|\psi|^{q-2}\psi/(q-1)$ and integrate by parts to obtain
\begin{equation}
	\label{eq:nablapsiq2}
\int |\nabla |\psi|^{q/2}|^2  = \frac{1}{q-1} \int P_{\neq 0} \omega |\psi|^{q-2}\psi  \, .
\end{equation}

We will also prove the localized inequality
\begin{equation}
	\label{eq:Poincaretypeineq}
\int_{B_1 \times \T} |\psi|^q  \lesssim \int_{B_1 \times \T} |\nabla |\psi|^{q/2}|^2  + \int_{B_1 \times \T} |P_{\neq 0} \omega|^q  ,
\end{equation}
specifically for solutions to~\eqref{psi_eq_bs}.  We prove~\eqref{eq:Poincaretypeineq} for scalar-valued solutions of~\eqref{psi_eq_bs} by contradiction and compactness. For the sake of contradiction, suppose that there exist $C_k \to +\infty$ and scalar-valued functions $\psi_k$, $\omega_k$ solving~\eqref{psi_eq_bs}, with $P_{\neq 0} \omega_k = \omega_k$, and satisfying
\begin{equation}
	\label{eq:inequalityhooray}
\int_{B_1 \times \T} |\psi_k|^q  \geq C_k \left(  \int_{B_1 \times \T} |\nabla |\psi_k|^{q/2}|^2  + \int_{B_1 \times \T} |\omega_k|^q  \right) \, .
\end{equation}
We define
\begin{equation}
\widetilde{\psi}_k := \frac{\psi_k}{\| \psi_k \|_{L^q}} \, , \quad \tilde{\omega}_k := \frac{\omega_k}{\| \psi_k \|_{L^q}} \, ,
\end{equation}
where the $L^q$ norm is on $B_1 \times \T$.
Then
\begin{equation}\label{eq:compact1}
\Delta \widetilde{\psi}_k = \widetilde{\omega}_k \, , \quad \| \widetilde{\psi}_k \|_{L^q} = 1
\end{equation}
while
\begin{equation}\label{eq:compact2}
\int_{B_1 \times \T} |\nabla |\widetilde{\psi}_k|^{q/2}|^2  + \int_{B_1 \times \T} |\tilde{\omega}_k|^q  \to 0 
\end{equation}
as $k\to \infty$. By the local regularity theory, we moreover have
\begin{equation}
	\label{eq:compactnessa}
\sup_k \| \nabla \widetilde{\psi}_k \|_{L^q(B_{r} \times \T)} \lesssim_r 1 \, , \quad \forall r \in [1/2,1) \, .
\end{equation}
Upon passing to a subsequence (which we relabel), we obtain $\widetilde{\psi }_\infty \in L^q_{\rm loc} (B_1\times \T )$ such that
\begin{equation}
\widetilde{\psi}_k \to \widetilde{\psi}_\infty \text{ in } L^q(B_r \times \T) \, , \quad \forall r \in [1/2,1)
\end{equation}
by the compactness granted in~\eqref{eq:compactnessa}, and
\begin{equation}
\Delta \widetilde{\psi}_\infty = 0 \text{ in } B_1 \times \T \, , \quad \int_{\T} \widetilde{\psi}_\infty (x',y_3)\,\d y_3=0 \quad \text{ for all }x.
\end{equation}
By the compactness granted by~\eqref{eq:compact1}--\eqref{eq:compact2}, we have
\begin{equation}
|\widetilde{\psi}_k|^{q/2} \to |\widetilde{\psi}_\infty|^{q/2} \text{ in } L^2 (B_1 \times \T) \, , \quad \nabla |\widetilde{\psi}_\infty|^{q/2} = 0 \, , \quad \| |\widetilde{\psi}_\infty|^{q/2} \|_{L^2} = 1 \, ,
\end{equation}
which implies that $|\widetilde{\psi}_\infty|^{q/2}$  (hence, also $|\widetilde{\psi}_\infty|$)  is a non-zero constant. Since $\widetilde{\psi}_\infty$ solves $\Delta \widetilde{\psi}_\infty = 0$, it is a smooth function.  Therefore, since any smooth, scalar-valued function whose absolute value is a constant function must be a constant, we deduce that $\widetilde{\psi}_\infty$ is a non-zero constant. Finally, since $\widetilde{\psi}_\infty$ has zero mean in $x_3$, it must be  zero. This contradicts that $\| \widetilde{\psi}_\infty \|_{L^q} = 1$ and completes the proof of~\eqref{eq:Poincaretypeineq}.

Finally, we combine~\eqref{eq:Poincaretypeineq}, summed over balls covering $\R^2$, together with~\eqref{eq:nablapsiq2} and the mean-zero property $\int \psi(x',y_3) \, \d y_3 = 0$ (for every $x$) to obtain
\begin{equation}
\int |\psi|^q \lesssim \int P_{\neq 0} \omega |\psi|^{q-1}  \lesssim \| P_{\neq 0} \omega \|_{L^q} \| \psi \|_{L^q}^{1-q} \, .
\end{equation}
This proves~\eqref{eq:psiinlq} and, therefore,~\eqref{eq:Pneqzeroest}.

We now verify \eqref{eq:Lpestimatein2d}, that is we consider  vorticities $\omega$ supported in Fourier modes with $\alpha = 0$, that is, independent of $x_3$. For these functions, we have
\begin{equation}
    u = (- \p_1^2 - \p_2^2)^{-1} ( \p_2 \omega_3, - \p_1 \omega_3, \p_1 \omega_2 - \p_2 \omega_1 ) \, .
\end{equation}
The horizontal velocity $(u_1,u_2)$ is determined by the 2D Biot-Savart law, and $u_3$ is obtained from $(\omega_1,\omega_2)$ by $-1$-order Fourier multipliers with homogeneous symbol which is smooth away from the origin. The 2D Calder\'on-Zygmund inequality and the Sobolev inequality therefore yield the desired estimate~\eqref{eq:Lpestimatein2d} when $\omega \in L^p$ and $1 < p < 2$.
\end{proof}

\subsection{Linear instability in $W^{1,Q}$}\label{sec_lin_instab_w1q}

Here we verify that, after rescaling, the solution $\phi$ constructed in Theorem~\ref{thm_main} satisfies $u\in W^{1,Q}(\R^2 \times \T)$ for some $Q\in (1,2)$. Eigenfunctions of the form~\eqref{perturbation_form} are periodic in $z$ with period $2\pi/\alpha$, and we may assume that $\alpha \in 2\pi \N$ after rescaling.


For the $W^{1,Q}$ regularity of $u$, we first recall that the potential $k$ in the Rayleigh equation \eqref{eq_phi} admits a singularity of order $n^2r^{-2}$ as $r\to 0$, and that $k(r)=O(n^2)$ as $r\to +\infty$. Thus, from the information on $\phi$ provided by Theorem~\ref{thm_main}, we see that $|u_r (r)| \lec r^{cn}$ as $r\to 0^+$ and $|u_r(r)| \lec  \ee^{-cnr} $ as $r\to \infty$ for some constant $c>0$. (Recall also the rescaling \eqref{phi_vs_ur} that $\phi = r^{3/2}(1+\beta^2 r^2 )^{-1/2} u_r $.) In particular, $\frac{u_r}r, r u_r', r^2 u_r'' \in L^2 (\R^+)$, which can be obtained by a simple weighted energy estimate for the equation  \eqref{eq_for_u}. This and the system \eqref{utheta_uz_from_ur} determining $u_\theta$, $u_z$ from $u_r$ show that the eigenfunction $u$ of the form \eqref{perturbation_form} belongs to $W^{1,Q} (\R^2 \times \T)$ for some $Q\in (1,2)$, which gives the following.
\begin{cor}[Linear instability in $W^{1,Q}$]\label{cor_lin_unst}
Let $M\geq 1$. For every $m\in \{ 1, \ldots , M\}$ and sufficiently large $n$ there exist  $\lambda_m \in \C$ with $\re\,\lambda_m >0$ (i.e. $\lambda_m = -i\omega_m$) and eigenfunction 
\[
u=\ee^{i(\alpha z - n \theta )}(u_r (r) e_r + u_\theta (r) e_\theta + u_z (r) e_z ) \in W^{1,Q} (\R^2\times \T ) \, ,
\]
such that $u\ee^{\lambda_m t}$ satisfies the linearized Euler equations \eqref{eq:linearizedevaleqn}.
\end{cor}

\subsection{Properties of $\LL$}\label{sec_prop_LL}
We study properties of the operator $\LL$, defined as a linearization of the $3$D Euler equations \eqref{3d_euler} around $\ou$. Recall from \eqref{def_LL} that $\LL  \coloneqq D(\LL ) \subset X \to X$ is defined by 
\[
\begin{split}
- \LL \omega   &= \ou \cdot \nabla \omega - \omega \cdot \nabla \ou - \oo \cdot \nabla u + u \cdot \nabla \oo \\
&=: \ou \cdot \nabla \omega + \wL \omega ,
\end{split}
\]
where $u={\rm BS}[\omega ]$ is the velocity field given by the Biot-Savart operator in Lemma~\ref{lem:biotsavartlaw}. 

Using the decay assumption~\eqref{eq:decayassumption}, one may show that the operator $\wL \colon L^Q_\sigma \to L^Q_\sigma$ is bounded. In particular, we may demonstrate that, for any $\lambda \in \C$ with $|\re \, \lambda| > \| \tilde{L} \|_{L^Q_\sigma \to L^Q_\sigma}$ and any $f\in L^Q_\sigma$, there exists a unique solution $\omega \in L^Q_\sigma$ to 
\eqnb\label{eq111}
(\lambda - \LL ) \omega = f 
\eqne
which satisfies the \emph{a~priori} estimate
\eqnb\label{eq112}
    \left( |\re \lambda| - \| \wL \|_{L^Q_\sigma \to L^Q_\sigma} \right) \| \omega \|_{L^Q_\sigma} \leq \| f \|_{L^Q_\sigma} \, ,
\eqne
obtained when we multiply~\eqref{eq111} by  $|\omega|^{Q-2} \omega^*$ (complex conjugate)  and integrate by parts.  We emphasize that here we made use of the cancellation
\[
\int_{\R^2\times \T} (\overline{u}\cdot \nabla \omega ) \cdot \omega^* |\omega |^{Q-2}=0,
\]
a consequence of the fact that $\mathrm{div}\,\overline{u} =0$, so that the \emph{a~priori} estimate~\eqref{eq112} involves $\wL$ only, while the resolvent problem~\eqref{eq111} involves $\LL$. 
We note that the existence of solutions to~\eqref{eq111} can be proved, for example, by considering vanishing viscosity approximations. 

Thus, \eqref{eq112} lets us estimate the spectrum of $\LL$,
\[
    \sigma(\LL ) \subset \left\lbrace  |\re\, \lambda| \leq \| \wL \|_{L^Q_\sigma \to L^Q_\sigma} \right\rbrace .
\]
In particular, the spectral bound is finite:
\begin{equation}
    \label{eq:Lambdaspectralbound}
    \Lambda := \sup_{\lambda \in \sigma(L)} \re \, \lambda < +\infty \, .
\end{equation}
In order to extract more information about $\sigma (\LL)$, we first rewrite $\LL$ in cylindrical coordinates. 

To this end, we recall that the Euler equations \eqref{3d_euler} in cylindrical coordinates take the form 
\[
\begin{cases}
    D_t u_r - \frac{u_\theta^2}{r} + \p_rp =0 \\
    D_t u_\theta + \frac{u_\theta u_r}{r} + \frac1r \p_\theta p =0\\
    D_t u_z + \p_z p =0,
\end{cases}
\]
where $D_t = \p_t + u_r \p_r + \frac{u_\theta }r \p_\theta + u_z \p_z$. In vorticity formulation, this gives
\[
\begin{cases}
    \p_t \omega_r + (u\cdot \nabla )\omega_r  - (\omega \cdot \nabla )u_r =0 \\
    \p_t \omega_\theta + (u\cdot \nabla )\omega_\theta  - (\omega \cdot \nabla )u_\theta = \frac1r (u_r \omega_\theta - u_\theta \omega_r ) \\
    \p_t \omega_z + (u\cdot \nabla )\omega_z  - (\omega \cdot \nabla )u_z =0, 
\end{cases}
\]
where $u\cdot \nabla = u_r \p_r + \frac1r u_\theta \p_\theta + u_z \p_z $, and 
\[
\omega_r = \frac1r \p_\theta u_z -\p_z u_\theta ,\qquad \omega_\theta = \p_z u_r - \p_r u_z ,\qquad \omega_z = \frac1r \p_r (ru_\theta ) - \frac1r \p_\theta u_r. 
\]
The velocity field associated to the vortex column \eqref{vortex_col_repeat} is
\[
\ou = \ou_\theta e_\theta + \ou_z e_z = V(r)  e_\theta + W(r) e_z \, ,
\]
and its vorticity is
\[
\ow = \ow_\theta e_\theta + \ow_z e_z = W'(r) e_\theta + \frac{(rV)'}{r} e_z 
\]
Thus $\omega + \ow $ satisfies the $3$D Euler equations if
\eqnb\label{001}
\begin{split}
\p_t \begin{pmatrix}
\omega_r \\ \omega_\theta \\ \omega_z 
\end{pmatrix} &+ 
\begin{pmatrix}
\frac{V}{r} \p_\theta \omega_r + W \p_z \omega_r \\
\frac{V}{r} \p_\theta \omega_\theta + W \p_z \omega_\theta +\left( - V'  +\frac{V}{r} \right) \omega_r \\
\frac{V}{r} \p_\theta \omega_z + W \p_z \omega_z -W' \omega_r 
\end{pmatrix}
+
\begin{pmatrix}
-\frac{W'}{r} \p_\theta u_r - \frac{(rV)'}{r} \p_z u_r \\
\left( W''-\frac{W'}r \right) u_r - \frac{W'}{r} \p_\theta u_\theta - \frac{(rV)'}{r} \p_z u_\theta \\
\left( \frac{(rV)'}{r} \right)' u_r -\frac{W'}{r} \p_\theta u_z - \frac{(rV)'}{r} \p_z u_z 
\end{pmatrix} \\
&+\begin{pmatrix}
u\cdot \na \omega_r - \omega \cdot \na u_r \\
u\cdot \nabla \omega_\theta - \omega \cdot \na u_\theta - \frac1r u_r \omega_r +\frac1r u_\theta \omega_r  \\
u\cdot \nabla \omega_z - \omega \cdot \na u_z 
\end{pmatrix} =0 \, .
\end{split}
\eqne

We recover $\LL$ by considering only the last two terms in the first line of \eqref{001}. 
Namely, writing\footnote{By this, we mean that $L^Q_\sigma$ is the $L^Q$-closure of the set of finite linear combinations of elements of $\bigcup_{n,\alpha} X_{n,\alpha} \ee^{i(\alpha z + n \theta)}$.} 
\[
L^Q_\sigma(\R^2 \times \T;\R^3)  = \bigoplus_{(n,\alpha) \in \Z^2} X_{n,\alpha }
\]
where $X_{n,\alpha}$ is defined in \eqref{def_Xna}, namely
\[
X_{n,\alpha } = \left\lbrace \omega  = w(r) \ee^{i(n\theta +  \alpha z)} \colon w \in L^Q (\R^+; r\d r )   \, , \; \frac1r  (r w_r )'  + \frac{in }r w_\theta + i\alpha w_z =0 \right\rbrace , 
\]
we have, for $\omega = w(r) \ee^{i(n\theta +  \alpha z)} \in X_{n,\alpha}$, that $
\LL (w(r) \ee^{i(\alpha z + n \theta)} ) = (L_{n,\alpha} w )(r) \ee^{i(\alpha z + n \theta)}$, where
\[
L_{n,\alpha } = A_{n,\alpha } + B_{n,\alpha } \, ,
\]
and, in cylindrical coordinates of $w$,
\eqnb\label{def_Ana}
- A_{n,\alpha } \begin{pmatrix}
w_r \\ w_\theta \\ w_z 
\end{pmatrix} \coloneqq
\begin{pmatrix}
in \frac{V}{r} + i\alpha  W  &0 & 0  \\
- V'  +\frac{V}{r} & in \frac{V}{r}  +i\alpha  W & 0 \\
-W' &0 & in \frac{V}{r} + i \alpha W
\end{pmatrix}
 \begin{pmatrix}
w_r \\ w_\theta \\ w_z 
\end{pmatrix}
\eqne
\eqnb\label{def_Bna}
- B_{n,\alpha } \begin{pmatrix}
w_r \\ w_\theta \\ w_z 
\end{pmatrix} \coloneqq
\begin{pmatrix}
- in \frac{W'}{r} - i\alpha  \frac{(rV)'}{r} &0&0 \\
 W''-\frac{W'}r & - in \frac{W'}{r}   - i\alpha \frac{(rV)'}{r} &0 \\
\left( \frac{(rV)'}{r} \right)' & 0& - in \frac{W'}{r}  -i\alpha \frac{(rV)'}{r}  
\end{pmatrix}
 \begin{pmatrix}
u_r \\ u_\theta \\ u_z 
\end{pmatrix} \, .
\eqne
In the above notation, $w = (w_r, w_\theta , w_z )$ is a vector function of one variable $r$, and the Biot-Savart law in $X_{n,\alpha }$ takes the form 
\eqnb\label{BS_in_Xna}
 \begin{pmatrix}
w_r \\ w_\theta \\ w_z 
\end{pmatrix}  = \begin{pmatrix}
0 & -i\alpha  & \frac{in }r \\
i\alpha  & 0 & - \p_r  \\
\frac{-in}r & \p_r + \frac1r  & 0 \end{pmatrix}
 \begin{pmatrix}
u_r \\ u_\theta \\ u_z 
\end{pmatrix}.
\eqne
We often exploit the canonical isomorphism between $\omega \in X_{n,\alpha}$ and $w \in L^Q(\R^+; r \, dr)$, where $\omega = w(r) \ee^{i(\alpha z + n\theta)}$.
Using the above decomposition we now characterize the spectra of $A_{n,\alpha}$ and $L_{n,\alpha}$
\begin{prop}[Mode-by-mode spectral properties of $\LL$]\label{prop_spec}
$A_{n,\alpha }$ and $B_{n,\alpha }$ are bounded on $X_{n,\alpha }$, and
\begin{enumerate}
\item[(i)] $\sigma (A_{n,\alpha }) = \overline{{\rm range}(-in \Omega - i\alpha W)}$
\item[(ii)] $B_{n,\alpha }$ is compact.
\end{enumerate}
\end{prop}
\begin{proof}
It is direct from~\eqref{def_Ana} that $A_{n,\alpha}$ is bounded. To see that $B_{n,\alpha}$ is bounded, we can use that it represents the terms $\bar{\omega} \cdot \nabla u - u \cdot \nabla \bar{\omega}$, which are bounded. 

For (i), we observe that $A_{n,\alpha}$ is multiplication operator by a bounded matrix function. Clearly, $\lambda - A_{n,\alpha}$ is invertible when $\lambda$ is away from the range of $-in \Omega - i\alpha W$, since the determinant of $\lambda - A_{n,\alpha}$ will be bounded away from zero. When $\lambda = -in \Omega(r_0) - i\alpha W(r_0)$ is in the range, we simply approximate $\delta_{r_0}(r) e_z$ by an $L^Q$ function. Finally, we recall that the spectrum is closed. 

As for (ii), the key observation is that $u$ gains a derivative according to Lemma~\ref{lem:biotsavartlaw} and the multiplier is localized according to~\eqref{eq:decayassumption}; cf.~\cite[Proposition~6.1]{gs_spectral} and \cite[Proposition~2.1]{gs_spectral}. 
\end{proof}

Because compact perturbations do not change the essential spectrum, we deduce from Proposition~\ref{prop_spec} that
\eqnb\label{spectral_info}
\sigma (L_{n,\alpha }) = \sigma_{\rm ess}(L_{n,\alpha}) \cup \sigma_{\rm disc}(L_{n,\alpha}),
\eqne
where the essential spectrum of $\LL$, $\sigma_{\rm ess}(L_{n,\alpha})$, is contained in the imaginary axis, and the discrete spectrum, $\sigma_{\rm disc}(L_{n,\alpha})$,  consists of isolated eigenvalues of finite algebraic multiplicity. In particular, for every open neighborhood $O$ of $\sigma_{\rm ess}(L_{n,\alpha})$, we have that $\sigma_{\rm disc}(L_{n,\alpha}) \cap O$ is finite.

We now prove that stable/unstable eigenfunctions are smooth.

\begin{lemma}[Smoothness of stable/unstable eigenfunctions]
\label{lem:smoothnessofefns}
   If $\omega \in X_{n,\alpha} \cap D(\mathcal{L})$ is an unstable eigenfunction of $\LL$, i.e.,
   \eqnb\label{omega_unst}
    (\lambda - \LL) \omega = 0 
   \eqne
   for some $\lambda \in \C$ with $\re \, \lambda >0$, then $\omega \in L^Q \cap H^k$ for every $k \geq 0$. 
\end{lemma}
We point out that $\omega$ solves \eqref{omega_unst} if and only if $\omega \ee^{\lambda t}$ solves the linearized Euler equations~\eqref{linearization}.
\begin{proof}
Writing $\omega = (w_r(r) e_r + w_\theta (r) e_\theta + w_z (r) e_z ) \ee^{in\theta + i\alpha z}$ we calculate that 
\[
\begin{split}
\overline{u}\cdot \nabla \omega &= (Ve_\theta + W e_z )\cdot \nabla ((w_r e_r + w_\theta e_\theta + w_z e_z )\ee^{in\theta +i\alpha z})\\
&= (V \frac{\p_\theta }r + W\p_z ) ((w_r e_r + w_\theta e_\theta + w_z e_z )\ee^{in\theta +i\alpha z}) \\
&= (in \frac{V}r + i\alpha W ) \omega +\frac{V}{r} (-w_\theta e_r + w_r e_\theta )\ee^{in\theta +i\alpha z}\\
&= (in \frac{V}r + i\alpha W ) \omega +\frac{V}{r^3} (-(\omega \cdot x^\perp ) x + (\omega \cdot x ) x^\perp)
\end{split}
\]
and 
\[
\begin{split}
\omega \cdot \nabla \ou & = \ee^{in\theta + i\alpha z} (w_r e_r + w_\theta e_\theta + w_z e_z ) \cdot \nabla (V e_\theta + W e_z)\\
&=  \ee^{in\theta + i\alpha z} (w_r \p_r + w_\theta \frac{\p_\theta }{r} (V e_\theta + W e_z)\\
&=  \ee^{in\theta + i\alpha z} \left( w_r (V' e_\theta + W' e_z) - w_\theta \frac{V}{r} e_r \right) \\
&= \frac{V'}{r^2} (\omega \cdot x ) x^\perp + \frac{W'}{r} (\omega \cdot x ) e_z - \frac{V}{r^3} (\omega \cdot x^\perp ) x,
\end{split}
\]
where we use the two-dimensional notation $(x,z) \in \R^2 \times \T$. 
Thus,
\[\begin{split}
  \lambda \omega + \bar{u} \cdot \nabla \omega - \omega \cdot \nabla \bar{u} &= \lambda \omega + \left( -in \frac{V}r + i\alpha W \right) \omega + \frac{V}{r^3} (\omega \cdot x )x^\perp - \left( \frac{V'}{r^2} (\omega \cdot x ) x^\perp + \frac{W'}{r} (\omega\cdot x )e_z \right)\\
  &= \lambda \omega + \left( -in \frac{V}r + i\alpha W \right) \omega - \frac{\Omega'}{r} (\omega \cdot x )x^\perp -  \frac{W'}{r} (\omega\cdot x )e_z \\
  &= (\lambda - M_{n,\alpha}) \omega \, ,
  \end{split} 
\]
where
\[
- M_{n,\alpha} \coloneqq \begin{pmatrix}
in \frac{V}r + i\alpha W + \frac{\Omega'}r xy & \frac{\Omega'}r y^2 &0 \\
-\frac{\Omega'}r x^2 & in \frac{V}r + i\alpha W - \frac{\Omega'}r xy &0\\
-\frac{W'}r x & - \frac{W'}r y &  in \frac{V}r +i\alpha W
\end{pmatrix}.
\]
Note that $M_{n,\alpha} \in W^{k,\infty } (\R^3)$ for every $k\geq 0$. Moreover,
\[
\det \, (\lambda - M_{n,\alpha}) = \left( \lambda - in \frac{V}r +i\alpha W \right)^3 \, ,
\]
as is easily seen from rewriting in cylindrical coordinates to obtain the lower diagonal matrix $\lambda - A_{n,\alpha}$. This shows that also $(\lambda - M)^{-1}\in W^{k,\infty }(\R^3 )$ for all $k\geq 0$, since both the cofactor matrix and $\det (\lambda - M_{n,\alpha})$ are smooth and $| \det (\lambda - M_{n,\alpha}) | \geq \re \lambda > 0$. Multiplying \eqref{omega_unst} on the left by $(\lambda - M_{n,\alpha})^{-1}$, we thus obtain
\[
\omega = (\lambda - M_{n,\alpha})^{-1} \left( \oo \cdot \nabla u - u \cdot \nabla \oo \right) \, .
\] 
In order to iterate regularity of $\omega$, we must be careful with the term $\oo \cdot \nabla u$, since at a first glance $\nabla u $ is of the same order as $\omega$. To this end, we observe that 
\[
\oo \cdot \nabla v = K_{n,\alpha} v
\]
for every $v\in X_{n,\alpha}$, where
\[
K_{n,\alpha} \coloneqq \begin{pmatrix}
in \frac{W'}r + i\alpha \left( V'+ \frac{V}r \right) &\frac{W'}{r} &0 \\
-\frac{W'}r & in \frac{W'}r + i\alpha \left( V'+ \frac{V}r \right)  & 0\\
0&0& in \frac{W'}r + i\alpha \left( V'+ \frac{V}r \right) 
\end{pmatrix}.
\]
We thus have 
\[
\omega = (\lambda - M_{n,\alpha})^{-1} (K_{n,\alpha} u - u\cdot \nabla \oo ),
\]
and so, since $K_{n,\alpha} \in W^{k,\infty } (\R^3)$ for every $k\geq 0$, we can pass derivatives through this identity and iterate regularity to obtain the claim. 
\end{proof}

We define the spectral bound 
\eqnb\label{sna}
s_{n,\alpha}=s(L_{n,\alpha}) \coloneqq \sup \{ \re\, \lambda : \lambda \in \sigma(L_{n,\alpha}) \}.
\eqne
Thus, if $s_{n,\alpha} > 0$, then our spectral information \eqref{spectral_info} implies that $s_{n,\alpha}$ is attained by an eigenvalue, that is, there exists $\lambda \in \sigma(L_{n,\alpha})$ such that $\re \,\lambda = s(L_{n,\alpha})$. Let
\begin{equation}
    \label{eq:Sdef}
    S \coloneqq \sup_{n,\alpha} s_{n,\alpha}  .
\end{equation}
Note that $S \leq \Lambda < +\infty$, where $\Lambda$ was defined in~\eqref{eq:Lambdaspectralbound}.

\subsection{Semigroup bounds for $L_{n,\alpha }$}\label{sec_semigroup}
Here we prove semigroup estimates on $\ee^{tL_{n,\alpha}}$ as follows. 
\begin{lemma}[Semigroup estimates on $\ee^{tL_{n,\alpha }}$]\label{lem:semigroupestimate}
If a vortex column \eqref{vortex_col_repeat} satisfies the decay assumption \eqref{eq:decayassumption}, then
    \begin{equation}
        \label{eq:Abound}
        \| \ee^{t A_{n,\alpha}} \|_{L^p \to L^p} \lesssim_{n,\alpha,p} (1+t) \, , \qquad \text{ for all } p \in [1,+\infty] \, , \; t \geq 0 \, .
    \end{equation}
    Moreover, for every $k \geq 0$ the semigroup $\ee^{tL_{n,\alpha}} \colon X_{n,\alpha} \subset L^Q \cap H^k \to L^Q \cap H^k$ admits the estimate
    \begin{equation}
        \label{eq:fullLbound}
      \| \ee^{tL_{n,\alpha}} \|_{L^Q \cap H^k \to L^Q \cap  H^k} \lesssim_{n,\alpha,\delta,C_0,\bb,k} \ee^{(s(L_{n,\alpha}) + \delta) t} \, ,\qquad \text{ for all } \delta >0 \, , \; t\geq 0 \, .
    \end{equation}
\end{lemma}
Before we prove the lemma, we note that, since $L_{n,\alpha}$ is a bounded operator (recall~\eqref{def_Ana}--\eqref{def_Bna}), it generates a \emph{uniformly continuous} semigroup $t \mapsto \ee^{ tL_{n,\alpha}}$, $t \geq 0$. Moreover, for such groups the growth bound equals the spectral bound (see~\cite[Chapter IV, Corollary 2.4, p. 252]{engel_nagel}, for example), i.e. $\omega_0(L_{n,\alpha}) = s(L_{n,\alpha})$. In other words,
\begin{equation}
    \label{eq:semigroupest}
    \| \ee^{tL_{n,\alpha}} \|_{X_{n,\alpha} \subset L^Q \to L^Q } \lesssim_{n,\alpha,\delta} \ee^{(s(L_{n,\alpha}) + \delta) t} \, , \quad \delta > 0 \, .
\end{equation}
Notice that the implicit constant in~\eqref{eq:semigroupest} may depend on $n,\alpha$. This is because, while the resolvents $R(\lambda,L_{n,\alpha})$ may exist for all $n,\alpha$, the abstract theory does not guarantee \emph{uniform resolvent estimates} with respect to $n,\alpha$. This is the reason why we cannot say $S = \Lambda$ and why we assume that $\mathcal{L}$ has an unstable eigenvalue, not merely unstable spectrum, to obtain nonlinear instability. For columnar vortices which satisfy Rayleigh's criterion and ``should" be stable, this difficulty is also remarked in~\cite{gs_spectral} and dealt with in~\cite{gs_linear}.

The point of Lemma~\ref{lem:semigroupestimate} is to obtain semigroup estimates not only in $L^Q$, which we already know from~\eqref{eq:semigroupest}, but also in $H^k$ for arbitrary $k$. 
Moreover, Lemma~\ref{lem:semigroupestimate} yields the same growth bound in $L^Q\cap H^k$ as the $L^Q$ bound \eqref{eq:semigroupest} above.
\begin{proof}[Proof of Lemma~\ref{lem:semigroupestimate}.]

We first show~\eqref{eq:Abound}. This semigroup is given explicitly by the matrix exponential of $A_{n,\alpha}(r)$. By the structure $A_{n,\alpha}(r) = \text{const.} \times I + N$, where $N^2 = 0$, we have
\eqnb\label{eq:etA}
    \ee^{tA_{n,\alpha}} = \ee^{-it(n\Omega+\alpha W)} \left( I + t\begin{bmatrix} 0 & & \\
    V'+V/r & 0 & \\
    W' & & 0
    \end{bmatrix} \right) ,
\eqne
which gives \eqref{eq:Abound}.

Second, we prove~\eqref{eq:fullLbound}. Write $\omega = \ee^{tL_{n,\alpha}} \omega_0$ with $\omega_0 \in X_{n,\alpha} \ee^{i(\alpha z + n \theta)} \cap H^k$. It is not difficult to show, using \eqref{eq:etA} and smoothness of $\ou$, that $\omega \in L^Q \cap H^k$ for every $k\geq 0$. We thus focus only on the growth bound \eqref{eq:fullLbound}. We follow a bootstrapping procedure.

To begin, we consider the equation for $\omega = \ee^{L_{n,\alpha }t} \omega_0$ in cylindrical variables, namely,
\begin{equation}
    \p_t \omega + A_{n,\alpha}\omega + B_{n,\alpha} u = 0 \, ,
\end{equation}
where $A_{n,\alpha}$, $B_{n,\alpha}$ are defined in \eqref{def_Ana}--\eqref{def_Bna}.
We estimate $B_{n,\alpha}$, which encodes the terms $- \bar{\omega} \cdot \nabla u + u \cdot \nabla \bar{\omega}$. We have
\begin{equation}
    \| \oo \cdot \nabla u \|_{L^2} \lesssim \| \oo \|_{L^p} \| \nabla u \|_{L^Q} \lesssim \| \omega \|_{L^Q}
\end{equation}
with $1/2 = 1/p + 1/Q$. 
Next, we consider the $u \cdot \nabla \oo$ term in two parts. First,
\begin{equation}
    \| {\rm BS}[P_0 \omega] \cdot \nabla \oo \|_{L^2} \lesssim \| {\rm BS}[P_0 \omega] \|_{L^{Q^*_2}} \| \nabla \oo \|_{L^p} \lesssim \| \omega \|_{L^Q} \, ,
\end{equation}
where $1/2 = 1/Q^*_2 + 1/p$. Second, suppose that $Q \in (1,6/5)$. (When $Q \in [6/5,2)$, then $W^{1,Q} \hookrightarrow L^2$, which simplifies this estimate.) Then 
\begin{equation}
	\| {\rm BS}[P_{\neq 0} \omega] \cdot \nabla \oo \|_{L^{6/5}} \lesssim \| {\rm BS}[P_{\neq 0} \omega] \|_{L^Q}^\theta \| {\rm BS}[P_{\neq 0} \omega] \|_{L^2}^{1-\theta} \| \nabla \oo \|_{L^\infty} \, ,
\end{equation}
where $\theta/Q + (1-\theta)/2 = 5/6$. We write Duhamel's formula:
\begin{equation}
    \omega = \ee^{-tA_{n,\alpha}} \omega_0 + \int_0^t \ee^{-(t-s)A_{n,\alpha}} (\oo \cdot \nabla u - u \cdot \nabla{\oo}) \, \d s \, .
\end{equation}
The term arising from the initial data evidently satisfies the desired estimates. We have
\begin{equation}
\begin{aligned}
      &\left\| \int_0^t \ee^{-(t-s)A_{n,\alpha}} (\oo \cdot \nabla u - u \cdot \nabla{\oo}) \d s\right\|_{L^2} \\
      &\quad \leq \int_0^t \ee^{(t-s)\delta} \ee^{(s_{n,\alpha} + \delta)s} \, \d s \sup_{s \in (0,t)} \ee^{-(s_{n,\alpha} + \delta)s}  \left( C_\varepsilon \| \omega \|_{L^Q} + \varepsilon C \| \omega \|_{L^2} \right) \\
      &\quad \leq \ee^{(s_{n,\alpha} + \delta)t} \left( C_\varepsilon \| \omega_0 \|_{L^Q} + \varepsilon C \sup_{s \in (0,t)}  \ee^{-(s_{n,\alpha} + \delta)s}  \| \omega \|_{L^2} \right) \, ,
\end{aligned}
\end{equation}
where we used the short-hand notation $s_{n,\alpha} \coloneqq s(L_{n,\alpha })$. The term containing $\varepsilon$ can be absorbed into the left-hand side of the desired estimate. This completes the $k=0$ case.

We now proceed by induction. Given the semigroup estimate \eqref{eq:fullLbound} with $H^k$, we wish to prove the same estimate with $H^{k+1}$. We have
\begin{equation}
    \p_t \nabla^k \omega + [\nabla^k, A_{n,\alpha}] \omega + A_{n,\alpha} \nabla^k \omega + \nabla^k (B_{n,\alpha} \omega) = 0 \, .
\end{equation}
We write Duhamel's formula:
\begin{equation}
    \nabla^{k+1} \omega = \ee^{- t A_{n,\alpha}} \nabla^{k+1} \omega_0 - \int_0^t \ee^{- (t-s) A_{n,\alpha}} [\nabla^{k+1}, A_{n,\alpha}] \omega \, ds - \int_0^t \ee^{- (t-s) A_{n,\alpha}}\nabla^{k+1} (B_{n,\alpha} \omega) \, \d s \, .
\end{equation}
The commutator satisfies $\| [\nabla^{k+1}, A_{n,\alpha}] \omega \|_{L^2} \lesssim \| \omega \|_{H^k}$. For the remaining term, we have
\begin{equation}
    \nabla^{k+1} (-\oo \cdot \nabla u + u \cdot \nabla \oo ) \, ,
\end{equation}
which can be treated similarly as before.
\end{proof}

\subsection{Conclusion of the proof of  Theorem~\ref{thm:nonlinearinstab}}\label{sec_iteration}

In this section, we complete the proof of Theorem~\ref{thm:nonlinearinstab}. Our approach is inspired by Grenier \cite{grenier}.

\subsubsection{Constructing approximate solutions}

Here we construct a sequence of approximate solutions $\omega_k (x,t)$, $k\geq 0$, such that
\begin{equation}
    \label{eq:estimatingomegak}
\| \omega_k(\cdot,t) \|_{L^Q \cap H^m} \lesssim_{k,m} \ee^{at}  \qquad \text{ for } m \geq 0, \, t\in \R  ,
\end{equation}
and 
\begin{equation}
    \label{eq:differenceofPicard}
\| (\omega_k - \omega_{l})(\cdot,t) \|_{L^Q \cap H^m} \lesssim_{k,l,m} \ee^{a(l+1)t} \qquad \text{ for } k,l\geq 0, k\geq l , m\geq 2, t\in \R .
\end{equation}

From a certain perspective, our goal of this step is to produce solutions satisfying the integral equations
\begin{equation}
    \omega(\cdot,t) = \omega_0 - \int_{-\infty}^t \ee^{(t-s)\mathcal{L}} [u,\omega] \, \d s \, ,
\end{equation}
with $\omega - \omega_0 = O(\ee^{2ta})$ in a certain topology. The integral equations lose one derivative, but this will not be an issue, as the eigenfunction is smooth, and we will consider finitely many $k$'s.

To be precise we first note that the assumptions of Theorem~\ref{thm:nonlinearinstab} imply not only that $\mathcal{L}$ has an unstable eigenvalue but also that one of the invariant subspaces must contain an unstable mode. Let $\lambda \in \sigma(L_{n,\alpha})$ be an unstable eigenvalue of $L_{n,\alpha}$ for some $(n,\alpha) \in \Z^2 \setminus \{ (\pm 1,0) \}$ satisfying
\begin{equation}\label{choice_a}
    a \coloneqq \re \,\lambda \in (S/2,S] \, .
\end{equation} Recall that $S$ was defined in~\eqref{eq:Sdef}. We can thus fix a nontrivial eigenfunction $\eta \in L^Q$ of \eqref{omega_unst} with some $\lambda \in \C$ with $\re\, \lambda >0$. We note that $\eta \in L^Q\cap H^k$ for every $k\geq 0$, by Lemma~\ref{lem:smoothnessofefns}.  We set
\begin{equation}
    \omega_0(\cdot,t) \coloneqq \ee^{t\mathcal{L}} \eta + (\ee^{t\mathcal{L}} \eta )^* = \ee^{\lambda t} \eta + \ee^{\lambda^* t} \eta^* ,
\end{equation}
and we observe that $\omega_0$ satisfies the linearized equation \eqref{linearization}, 
\begin{equation}
    \p_t \omega_0 - \mathcal{L} \omega_0 = 0  .
\end{equation}

In order to show~\eqref{eq:estimatingomegak}, we first consider the algebra $\mathbb{A}$ of functions generated, under function-function addition and multiplication and scalar-function multiplication, by $L^Q \cap (\cap_{m \geq 0} H^m)$ functions in the $(n,\alpha)$ and $(-n,-\alpha)$ modes. This algebra only contains functions with components in finitely many Fourier modes, which are necessarily of the form $(\tau n, \tau \alpha)$, $\tau \in \Z$. Crucially, for all $f \in \mathbb{A}$, we have the semigroup bound (see Lemma~\ref{lem:semigroupestimate})
\begin{equation}
    \| \ee^{t\mathcal{L}} f \|_{L^Q \cap H^m} \leq C \ee^{tS} \| f \|_{L^Q \cap H^m} \, ,
\end{equation}
where $C$ depends on the background $\oo$, the Fourier modes in which $f$ is concentrated, and integrability and regularity parameters $Q$ and $m$.

 For each $k\geq 1$, we define the \emph{$k$th Picard iterate} by
\begin{equation}
\omega_k (\cdot,t) \coloneqq \omega_0 - \int_{-\infty }^t \ee^{(t-s)\mathcal{L}}  [u_{k-1},\omega_{k-1} ] \, \d s.    
\end{equation}
A direct calculation, using induction, will show that $\omega_k$ is well defined. Crucially, we can estimate the nonlinear term as 
\begin{equation}
    \label{eq:estimateonproducts}
    \| u \cdot \nabla \omega \|_{L^Q \cap H^m} + \| \omega \cdot \nabla u \|_{L^Q \cap H^m} \lesssim_m \| \omega \|_{L^Q \cap H^{m+1}}^2 \, .
\end{equation}
Then~\eqref{eq:estimatingomegak} follows from the known estimate on $\omega_0$ and 
\begin{equation}
\begin{aligned}
    \left\| \int_{-\infty }^t \ee^{(t-s)\mathcal{L}} [u_{k-1},\omega_{k-1} ] \d s \right\|_{L^Q \cap H^m} &\leq C \int_{-\infty}^t \ee^{S(t-s)} \| [u_{k-1}, \omega_{k-1}](\cdot,s) \|_{L^Q \cap H^m} \, \d s \\
    &\leq C \int_{-\infty}^t \ee^{S(t-s)} \ee^{2sa} \, \d s \left( \sup_{s \in (-\infty,t)} \ee^{-as} \| \omega_k \|_{L^Q \cap H^{m+1}} \right)^2  \\ &\lesssim \ee^{2at} \, ,
\end{aligned}
\end{equation}
for $m \geq 2$, where we used the fact \eqref{choice_a} that $2a > S$. In conclusion,
\begin{equation}
    \label{eq:estimatingomegak_repeat}
\| \omega_k(\cdot,t) \|_{L^Q \cap H^m} \lesssim_{k,m} \ee^{at} \, , \quad \forall m \geq 0, \, t\in \R  ,
\end{equation}
and the Biot-Savart estimates of Lemma~\ref{lem:biotsavartlaw} give 
\begin{equation}
    \| u_k(\cdot,t) \|_{L^{Q^*_2}} + \| \nabla u_k(\cdot,t) \|_{H^m} \lesssim_{k,m} \ee^{at} \, \qquad \text{ for } m \geq 0,\, t\in \R ,
\end{equation}
where $Q^*_2 \coloneqq 2Q/(2-Q)$ denotes the $2$D Sobolev exponent of $Q$. This in particular shows \eqref{eq:estimatingomegak}.

Similarly, one can use the observation that
\[
\begin{aligned}
(\omega_k - \omega_{k-1} )(t) &= -\int_{-\infty }^t \left( [u_{k-1}, \omega_{k-1}] - [u_{k-2},\omega_{k-2}] \right) \d s\\
&= -\int_{-\infty }^t \left( [u_{k-1}-u_{k-2}, \omega_{k-1}] + [u_{k-2},\omega_{k-1}-\omega_{k-2}] \right) \d s \, .
\end{aligned}
\]
to show inductively that
\[
\| (\omega_k - \omega_{k-1})(\cdot,t) \|_{L^Q \cap H^m} \lesssim_{k,m} \ee^{a(k+1)t} \qquad \text{ for }m\geq 2, t \in \R,
\]
from which \eqref{eq:differenceofPicard} follows. 

\subsubsection{Nonlinear estimate}

Here we show that, for every $\ell \in \N$ and any sequence $T_{\ell } \to -\infty$, as $\ell \to \infty$, there exists a unique solution $\wo^{(\ell )}$ to
\begin{equation}
    \label{eq:woequation}
\p_t \wo- \mathcal{L} \wo + [ \wu , \wo] + [ \wu , {\omega_k }] + [ {u_k} , \wo] + \underbrace{[u_k - u_{k-1},\omega_k ] +[u_{k-1} , \omega_k - \omega_{k-1} ]}_{=: - E_k } = 0 
\end{equation}
with $\wo^{(\ell )} (T_{\ell } )=0$, 
and $T<0$, such that
\eqnb\label{toshow_wo}
  \| \wo^{(\ell)} \|_{L^Q \cap H^N}^2 \leq C(N,k) \int_{T_\ell}^t \ee^{(M-1/2)(t-s)} \ee^{Ms} \, \d s \leq C_1(N,k) \ee^{Mt} 
\eqne
for all $t\in [T_{\ell} , T ]$, upon an appropriate choice of $k$, where $M\coloneqq 2a (k+1)$. \\

Note that then 
\[
    \omega \coloneqq \oo + \omega_k + \wo ^{(\ell )}
\]
satisfies the $3$D Euler equations \eqref{3d_euler}.  We note that, by classical theory for the solution $\omega = \oo + \omega_k + \wo^{(\ell)}$, a solution exists on a short time interval which \emph{a priori} may depend on $\ell$. Our goal in this step is to estimate $\wo^{(\ell)}$ and thereby demonstrate that it exists up to time $T = O(1)$ and decays like $O(\ee^{a(k+1)t})$ backward in time.\\

In order to estimate $\omega^{(\ell )}$ in $L^Q \cap H^N$, we first note that
\begin{equation}
\| E_k \|_{L^Q \cap H^m} \lesssim_{k,m} 
\ee^{a(k+1)t} 
\end{equation}
for all $k, m \geq 0$. For brevity we will write $\wo = \wo^{(\ell )}$ in the remainder of this section. We rewrite \eqref{eq:woequation} as
\begin{equation}
    \label{eq:woequationwithoutcommutator}
\begin{aligned}
       &\p_t \wo + (\ou + u_k + \wu ) \cdot \nabla \wo \\
       &\quad + \wu  \cdot \nabla (\oo + \omega_k) - \wo \cdot \nabla (\ou  + u_k + \wu) - (\oo  + \omega_k) \cdot \nabla \wu  = E_k ,
       \end{aligned}
\end{equation}
We apply $\p^\alpha$, where $\alpha \in (\N_0)^3$ is a multi-index with $|\alpha| \leq N$, to the equation for $\wo^{(\ell)}$:
\begin{equation}
    \label{eq:equationwithcommutator}
\begin{aligned}
       &\p_t \p^\alpha \wo + (\ou + u_k + \wu) \cdot \nabla \p^\alpha \wo + [\p^\alpha, (\ou  + u_k + \wu) \cdot] \nabla \wo \\
       &\quad + \p^\alpha \left( \wu  \cdot \nabla (\ow  + \omega_k) - \wo \cdot \nabla (\bar{u} + u_k + \tilde{u}) - (\oo  + \omega_k) \cdot \nabla \wu \right) = \p^\alpha E_k \, . 
\end{aligned}
\end{equation}
We now estimate the various terms. 
To begin, we have the estimate on the commutator term:
\begin{equation}
\begin{aligned}
    \| [\p^\alpha, (\ou + u_k + \wu) \cdot] \nabla \wo \|_{L^2} &\leq C(N) \left( \| \ou + u_k + \wu \|_{L^\infty} + \| \nabla( \ou + u_k + \wu ) \|_{H^{N-2}}\right) \| \wo \|_{H^N} \\
    &\leq C(N) ( 1 + C_k \ee^{at}) \| \wo \|_{H^N} + C(N) \| \wo \|_{L^Q \cap H^N} \| \wo \|_{H^N} \, .
    \end{aligned}
\end{equation} 
The remaining terms can be estimated in $L^Q \cap H^N$ using the product estimate~\eqref{eq:estimateonproducts}. A crucial point is that estimates on the terms containing the background arise with a constant depending on the regularity $N$ but not the order of approximation $k$.
We multiply~\eqref{eq:equationwithcommutator} by $\p^\alpha \wo$ and integrate by parts to obtain
\begin{equation}
    \label{eq:HNdiffeq}
    \frac{1}{2} \frac{\d}{\d t} \| \p^\alpha \wo \|_{L^2}^2 \leq C(N)\left( 1+C_k \ee^{at} + \| \wo \|_{L^Q \cap H^N} \right) \| \wo \|_{L^Q \cap H^N}^2 + C(N,k) \ee^{2a(k+1)t} \, .
\end{equation}
Next, we multiply~\eqref{eq:woequationwithoutcommutator} by $|\wo|^{Q-2}\wo$ and integrate by parts:
\begin{equation}
    \label{eq:Qdiffeq}
    \frac{1}{Q} \frac{\d}{\d t} \| \wo \|_{L^Q}^Q \leq C \left( 1+ C_k \ee^{at} + \| \wo \|_{L^Q \cap H^N} \right) \| \wo \|_{L^Q \cap H^N}^Q + C_k \ee^{Qa(k+1)t}
\end{equation}
where we have used the inequality
\begin{equation}
    \int f \cdot |\wo|^{Q-2}\wo \, dx \leq \| f \|_{L^Q} \| \wo^{Q-1} \|_{L^{Q/(Q-1)}} = \| f \|_{L^Q} \| \wo \|_{L^Q}^{Q-1} \leq C_\varepsilon \| f \|_{L^Q}^Q + \varepsilon \| \wo \|_{L^Q}^Q \, ,
\end{equation}
for any $\varepsilon > 0$, to split and absorb various terms. Let $A = \| \wo \|_{L^Q}^Q$. Then~\eqref{eq:Qdiffeq} gives us a differential inequality for $\| \wo \|_{L^Q}^2 = A^{2/Q}$ according to $\frac{\d}{\d t} A^{2/Q} = \frac{2}{Q} A^{2/Q - 1} \frac{\d}{\d t} A$:
\begin{equation}
    \label{eq:Qdiffeqtwo}
    \frac{\d}{\d t} \| \wo \|_{L^Q}^2 \leq C \left( 1+ C_k \ee^{at} + \| \wo \|_{L^Q \cap H^N} \right) \| \wo \|_{L^Q}^{2-Q} \| \wo \|_{L^Q \cap H^N}^Q + C_k \ee^{2a(k+1)t}
\end{equation}
where we have used the inequality
\begin{equation}
    \| \wo \|_{L^Q}^{2-Q} \ee^{Qa(k+1)t} \leq C \| \wo \|_{L^Q}^2 + C \ee^{2a(k+1)t} \, ,
\end{equation}
that is, Young's inequality $a b \lec a^{\frac{Q}{2-Q}} + b^\frac{2}{Q}$.
 Next, we sum~\eqref{eq:HNdiffeq} over multi-indices $|\alpha| \leq N$ and~\eqref{eq:Qdiffeqtwo} to obtain
\begin{equation}\label{temp21}
\begin{aligned}
    &\frac{\d}{\d t} \| \wo \|_{L^Q \cap H^N}^2 \\
    &\quad \leq C(N) \left( 1+ C_k \ee^{at} + \| \wo \|_{L^Q \cap H^N} \right)  \| \wo \|_{L^Q \cap H^N}^2 + C(N,k) \ee^{2a(k+1)t} \, .
    \end{aligned}
\end{equation}
We now consider only $t \leq T$ for some (not yet fixed) $T \leq 0$. At $t = T_\ell$, we have $\wo^{(\ell)} = 0$, and we only concerned with $t\geq T_{\ell}$ for which $\| \wo \|_{L^Q \cap H^N} \leq  1$. With these restriction, \eqref{temp21} can be simplified to 
\begin{equation}
    \label{eq:maindifferentialinequality}
\begin{aligned}
    \frac{\d}{\d t} \| \wo \|_{L^Q \cap H^N}^2  \leq C_0(N) (1+ C_k \ee^{at})  \| \wo \|_{L^Q \cap H^N}^2 + C(N,k) \ee^{2a(k+1)t} ,
    \end{aligned}
\end{equation}
for $t\geq T_{\ell}$ such that $\| \wo \|_{L^Q \cap H^N } \leq 1$.  We now fix the order $k$  large enough to ensure
\begin{equation}
    2a(k+1) > C_0(N) + 1 \, ,
\end{equation}
where $C_0(N)$ is as in~\eqref{eq:maindifferentialinequality}., and we restrict our attention only to $T \ll -1$ such that 
\begin{equation}
    C_0(N) C_k \ee^{at} \leq 1/2 \,\quad \text{ for }t\leq T .
\end{equation}
With these restrictions, we have
\begin{equation}
    \frac{\d}{\d t} \| \wo \|_{L^Q \cap H^N}^2  \leq (M-1/2) \| \wo \|_{L^Q \cap H^N}^2 + C(N,k) \ee^{Mt} 
\end{equation}
with $M = 2a(k+1)$. Hence, provided $T_\ell \leq t \leq T$ and $\| \wo \|_{L^Q \cap H^N} \leq 1$, we have
\begin{equation}
    \| \wo^{(\ell)} \|_{L^Q \cap H^N}^2 \leq C(N,k) \int_{T_\ell}^t \ee^{(M-1/2)(t-s)} \ee^{Ms} \, \d s \leq C_1(N,k) \ee^{Mt} \, .
\end{equation}
Finally, we fix $T \ll -1$ to ensure
\begin{equation}
    C_1(N,k) \ee^{Mt} \leq 1/2 \, ,
\end{equation}
which guarantees that the estimates can be propagated until time $T$ unconditionally.

\subsubsection{Conclusion}\label{sec_conclude_it}

Here we take a limit $\ell \to \infty$ to conclude the proof of Theorem~\ref{thm:nonlinearinstab}. \\

Namely, passing to a subsequence, which we do not relabel, the solutions $\wo^{(\ell)}$ converge to $\wo$ in the sense that
\[
    \wo^{(\ell)} \overset{*}{\rightharpoonup} \wo \text{ in } L^\infty ((\tau ,T ); H^N \cap L^Q (\R^2 \times \T )) 
\]
for every $\tau \leq T$. By lower semicontinuity, $\wo$ inherits the above estimates.
It can also be shown using the estimates on the subsequence and the Aubin-Lions lemma that the solutions $u^{(\ell)}\coloneqq BS[\oo + \omega_k + \wo ^{(\ell )}]$ to the Euler equations converge strongly in $L^2_{\rm loc}(\R^2 \times \T \times (-\infty,T))$ to a vector field $u$, which is therefore a solution of the Euler equations.

We have thus obtained a solution $u$ of the Euler equations on $\R^2 \times \T \times (-\infty,T)$ which satisfies the decomposition
\begin{equation}
    \omega = \oo + \omega_0 + (\omega_k - \omega_0) + \wo \, ,
\end{equation}
where $\omega_0 = \ee^{\lambda t} \eta + \ee^{\lambda^* t}\eta^*$ is explicit and satisfies
\[
    \| u_0 \|_{L^\infty} \geq \widetilde{C}^{-1} \ee^{at}
\]
for all $t\in \R$, while the remainders satisfy the estimates
\[
    \| u_k - u_0 \|_{L^\infty} + \| \wu\|_{L^\infty} \leq \widetilde{C} \ee^{2at} 
\]
for $t\leq T$, by \eqref{eq:differenceofPicard} and \eqref{toshow_wo}. We now pick $\overline{T} \leq T$ such that $\widetilde{C}^2\ee^{a\overline{T}} \leq 1/4$. Then
\begin{equation}
    \label{eq:deltafar}
    \begin{split}
    \| u(\overline{T}) - \ou  \|_{L^\infty} &\geq \| u_0 (\overline{T})   \|_{L^\infty}-\| (u_k - u_0) (\overline{T})   \|_{L^\infty} - \| \wu (\overline{T})   \|_{L^\infty} \geq \widetilde{C}^{-1} \ee^{a\overline{T}} - \widetilde{C} \ee^{2a\overline{T}} -\widetilde{C} \ee^{2a\overline{T}} \\
    &\geq  \frac{1}{2 \widetilde{C}} \ee^{a\overline{T}} =: \delta \, ,
    \end{split}
\end{equation}
while
\begin{equation}
    \label{eq:canbemadeepsilonsmall}
    \| (\omega - \oo)(t) \|_{L^Q \cap H^N} \leq \| \omega_k (t)  \|_{L^Q \cap H^N}+ \| \wo (t)  \|_{L^Q \cap H^N} \leq  C \ee^{at} 
\end{equation}
for all $t\leq T$, by \eqref{eq:estimatingomegak} and \eqref{toshow_wo}.
Together,~\eqref{eq:canbemadeepsilonsmall} and ~\eqref{eq:deltafar} guarantee that, for any $\varepsilon > 0$, one can translate the solution $u$ in time to ensure that (i) at time zero, the vorticity $\omega$ is $\varepsilon$-close to the background $\oo$ in the $L^Q \cap H^N$ topology, whereas (ii) at time $T_\varepsilon$, the $u$ velocity is $\delta$-far from the background $\ou$ in the $L^\infty$ topology, as required.

\appendix
\section{Expansion of $k$}\label{app_exp_of_k}
We first compute the Taylor coefficients of the potential $k(r)$, defined in~\eqref{def_of_k},  at $r=r_0$,
\eqnb\label{exp_01}
\begin{split}
 k(r_0) &= p_0 n^2 \left( 1+ \frac{b_0}{(-ib_0^{1/2} - \mu_m )^2} + \frac{a_0}{n(-ib_0^{1/2} - \mu_m )} + \frac{d_0}{n^2} \right) \\
&= p_0 n^2 \left(- i  \mu_m \frac{2 }{b_0^{1/2}} + \mu_m^2 \frac{3}{b_0} +  O(\mu_m^3 ) + i \frac{a_0}{n b_0^{1/2}}+ O(\mu_m n^{-1})+ O(n^{-2})\right),\\
 k'(r_0) &=  p'(r_0)  n^2 \left( -i  \mu_m \frac{2 }{b_0^{1/2}} + \mu_m^2 \frac{3}{b_0^2} +  O(\mu_m^3 ) + O(n^{-1})\right)
+p_0 n^2 \left( \frac{a'(r_0)}{n(-i b_0^{1/2} - \mu_m )} + \frac{d'(r_0)}{n^2} \right)  \\
&=  n^2 \left( -i  \mu_m \frac{2p'(r_0) }{b_0^{1/2}}  +  O(\mu_m^2 ) + O(n^{-1})\right)\\
 k'' (r_0) &= p'' (r_0) n^2 \left(- i  \mu_m \frac{2 }{b_0^{1/2}} + \mu_m^2 \frac{3}{b_0^2} +  O(\mu_m^3 ) + O(n^{-1})\right) + 2 p'(r_0) n^2  \left( \frac{a'(r_0)}{n(-i b_0^{1/2} - \mu_m)} + \frac{d'(r_0)}{n^2} \right)\\
&\quad + p_0 n^2 \left(  \frac{b''(r_0)}{(-ib_0^{1/2}-\mu_m )^2}-\frac{2b_0 n\Lambda''(r_0) }{(-ib_0^{1/2}-\mu_m )^3} +\frac{a''(r_0)}{n(-ib_0^{1/2}-\mu_m )}-\frac{a(r_0)n\Lambda''(r_0)}{n(-ib_0^{1/2}-\mu_m )^2} +\frac{d''(r_0)}{n^2}\right) \\
&=  n^2 \left(- i\mu_m \frac{2p''(r_0) }{b_0^{1/2}} + O(\mu_m^2 ) + O(n^{-1}) 
- \frac{p_0 b''(r_0) }{b_0} -i \mu_m \frac{2p_0 b''(r_0) }{b_0^{3/2}}  + i n \frac{2 p_0 \Lambda''(r_0 )}{b_0^{1/2}} \right.
\\&\hspace{4cm}\left. + n \mu_m \frac{6p_0  \Lambda''(r_0)}{b_0}+ \frac{a(r_0) p_0 \Lambda''(r_0) }{b_0}+ i \mu_m \frac{2p_0 a(r_0)\Lambda''(r_0)}{b_0^{3/2}} \right)
\end{split}
\eqne

Note that the leading order terms of $k'(r_0)$ (arising from taking $\d/\d r$ of $b/\gamma^2$) vanish due to our choice \eqref{choice_r0_beta} of $r_0,\beta$.

 Let $B \geq 1$. We now show that, if
\begin{equation}
	\label{eq:assumptionontildeomega}
|\mu_m| + |r-r_0| \leq Bn^{-1/2}
\end{equation}
(recall~\eqref{choice_of_tilde_om_m_copy} concerning $\mu_m$) for $n$ sufficiently large, then
\eqnb\label{remainder_11_app}
k_{\rm err} (r)  = in^2 \ho \frac{2p_0}{b_0^{1/2}}+ n (1+ n^3 (r-r_0)^4 ) O(1 )   \, ,
\eqne
where the implicit constant in the $O(1)$ may depend on $B$. 
(Recall \eqref{kerr_def} that $k_{\rm err} (r) = k(r)- (k_0 + k_2 (r-r_0)^2 ) $.)
 To this end, we take into account the assumption $\mu_m = O(n^{-1/2})$ from~\eqref{eq:assumptionontildeomega} 
 to further reduce \eqref{exp_01} to 
\eqnb\label{expansion_k_more_precis}
\begin{split}
k(r_0)  &= k_0 -i n^2 \ho \frac{2p_0}{b_0^{1/2}}+ n^2 \mu_m^2 \frac{3p_0}{b_0} + \hat \omega O(n^{3/2}) + in \frac{a_0}{b_0^{1/2}} + O(n^{1/2}) = k_0 + O(n^{3/2}) ,\\
 k' (r_0) &= (-1+i) O(n^{3/2}), \\
 k''(r_0)/2 &= k_2 + n^3 \mu_m \frac{3p_0 \Lambda''(r_0)}{b_0} +n^3 \ho \frac{3p_0 \Lambda''(r_0)}{b_0} + O(n^2) ,
 \end{split}
\eqne 
where we also used the fact that $\ho = O(n^{-1/2})$. In order to estimate the remainder \eqref{remainder_11_app} of the Taylor expansion of $k$, we also need to estimate third derivatives. For this we only keep track of powers of $n$ as well as $r-r_0$ to obtain
\[
\begin{split}
    b&= O(1)  + O((r-r_0)^2)\\
    b'&=O(r-r_0)\\
    b'',b'''&= O(1)\\
    \gamma &= n O((r-r_0)^2 ) + O(1)\\
    \gamma'&= n \Lambda' = n O(r-r_0)\\
    \gamma'' &= n\Lambda'' = nO(1)\\
    \gamma''' &= n\Lambda''' = n O(1) \\
    a, a', a'', a''',p, p', p'', p''' &= O(1),    
\end{split}
\]
which implies that
\[
\begin{split}
\frac{a}{\gamma }  &= O(1), \\
\left(\frac{a}{\gamma } \right)' &= \frac{a'}{\gamma} - \frac{a}{\gamma^2 } \gamma' = n O(r-r_0),\\
\left(\frac{a}{\gamma } \right)'' &= \frac{a''}{\gamma} - 2\frac{a'}{\gamma^2 } \gamma' + \frac{a(2(\gamma' )^2 -\gamma'' )}{ \gamma^3} = O(1) + nO(r-r_0) + nO(1) = nO(1),\\
\left(\frac{a}{\gamma } \right)''' &= \frac{a'''}{\gamma} - 3\frac{a''}{\gamma^2 } \gamma' +3 \frac{a'(2(\gamma' )^2 -\gamma'' )}{ \gamma^3} +\frac{a(4\gamma \gamma' \gamma'' - \gamma \gamma''' -6 \gamma^2 (\gamma')^3 + 3\gamma^2 \gamma''\gamma' )}{\gamma^4}\\
&= O(1) + n O(r-r_0) +nO(1) + (n^2 O(r-r_0) + n O(1) + n^3 O((r-r_0)^3) + n^2 O(r-r_0) )\\
&=n^2 O(r-r_0) + nO(1), \\
 \frac{b}{\gamma^2 } &=  O(1), \\
\left( \frac{b}{\gamma^2 } \right)' &= \frac{b'}{\gamma } - 2 \frac{b }{\gamma^3 } \gamma' = O(r-r_0) + n O((r-r_0)^2)= n O((r-r_0)^2),  \\
\left( \frac{b}{\gamma^2 } \right)'' &= \frac{b''}{\gamma } - 4 \frac{b' }{\gamma^3 } \gamma' +\frac{b(6(\gamma')^2 -2\gamma \gamma'' )}{\gamma^4 } =  n O((r-r_0)^2) + (n^2 O((r-r_0)^2) + nO(1))=nO(1), \\
\left( \frac{b}{\gamma^2 } \right)''' &= \frac{b'''}{\gamma } - 6 \frac{b'' }{\gamma^3 } \gamma' +3\frac{b'(6(\gamma')^2 -2\gamma \gamma'' )}{\gamma^4 }+\frac{b(18\gamma \gamma' \gamma'' - 2\gamma^2 \gamma'''-24 (\gamma' )^3  )}{\gamma^5} \\
&= O(1) + nO(r-r_0) + (n^2 O((r-r_0)^3) + nO(r-r_0) ) \\
&\quad + (n^2 O(r-r_0) + n O(1) + n^3 O((r-r_0)^3) )\\
&= n^2 O(r-r_0) + n O(1) ,
\end{split}
\]
as long as $r-r_0 = O(n^{-1/2})$  as in~\eqref{eq:assumptionontildeomega}. 
Thus
\eqnb\label{remainder_more_precise}
\begin{split}
k''' &= pn^2 \left( \frac{a}{n\gamma } + \frac{b}{\gamma^2} + \frac{d}{n^2} \right)''' + 3p' n^2 \left( \frac{a}{n\gamma } + \frac{b}{\gamma^2} + \frac{d}{n^2} \right)'' + 3p'' n^2 \left( \frac{a}{n\gamma } + \frac{b}{\gamma^2} + \frac{d}{n^2} \right)' \\
&\quad + p''' n^2 \left( \frac{a}{n\gamma } + \frac{b}{\gamma^2} + \frac{d}{n^2} \right)\\
&= (n^4 O(r-r_0)+n^3 O(1) ) + n^3 O(1) + (n^2 O(r-r_0) + n^3 O((r-r_0)^2))+ n^2O(1) \\
&= n^4 O(r-r_0) + n^3 O(1) .
\end{split}
\eqne
This and \eqref{expansion_k_more_precis}  let us obtain the approximation error
\eqnb\label{Kerr_est_app}
\begin{split}
k_{\rm err} (r)  &= in^2 \ho \frac{2p_0}{b_0^{1/2}}+ n O(1 )   +  n^{3/2} O(r-r_0) +  n^{5/2} O((r-r_0)^2 ) + k'''(\tilde r ) (r-r_0)^3/6\\
& = in^2 \ho \frac{2p_0}{b_0^{1/2}}+ n (1+ n^3 (r-r_0)^4 ) O(1 )  ,
\end{split}
\eqne
 which proves \eqref{remainder_11_app}, as required. Here we have also denoted by $\tilde r$ some point between $r_0$ and $r$.

\subsection{Choice of $r_0,\beta >0$ satisfying \eqref{choice_r0_beta}}\label{sec_choice_r0_beta}
Here we verify that, in the case of the trailing vortex \eqref{trailing_vortex}, if $q$ satisfies~\eqref{q_restr}, then 
\eqnb\label{choice_r0_beta_claim}
\text{there exists a unique }r_0 ,\beta >0\text{ such that \eqref{choice_r0_beta} holds.}
\eqne
 Indeed, for the trailing vortex \eqref{trailing_vortex} $\Lambda'(r)=0$ if and only if
\eqnb\label{lambda_der}
\ee^{r^2} = 1+r^2 + \frac{\beta r^4 }q.
\eqne
Note that this equation has no solution $r>0$ if $\beta\leq q/2$ (in that case \eqref{lambda_der}, with ``$=$'' replaced by ``$\geq$'', holds for all $r\geq 0$, with equality only for $r=0$). On the other hand for each $\beta >q/2$ \eqref{lambda_der} has a unique solution $r_0>0$. Clearly the solution 
\eqnb\label{bijec}
r_0 = r_0(\beta)\quad \text{ is an increasing bijection }r_0\colon (q/2 ,\infty ) \to (0,\infty ).
\eqne 

On the other hand,  we have $b(r) =4  \beta (1-\beta q) q (1-\ee^{-r^2})\ee^{-r^2}/(1+\beta^2 r^2)$, and $b'(r)=0$ if and only if
\eqnb\label{find_zero}
g(r) - \beta^2 =0,\quad \text{ where } g(r)\coloneqq \frac{1-2\ee^{-r^2}}{r^2(2\ee^{-r^2}-1 )+ \ee^{-r^2}-1}.
\eqne
In particular, $b$ has a sign, depending on the sign of $1-\beta q$, and so the third condition in \eqref{choice_r0_beta} implies that we must have $\beta <1/q$. This together with the lower bound on $\beta$ above implies the constraint \eqref{beta_range}. 

This in particular gives the first smallness requirement on $q$, namely, $q<2^{1/2}$. 

We now observe that $g(r)>0$ if and only if $r <(\log 2)^{1/2}$, and that $g$ is decreasing for such $r$, with $g(r) \to +\infty$ as $ r\to 0^+$. In order to restrict ourselves to $r\in (0,(\log 2)^{1/2})$, so that we can find such $r$ solving \eqref{find_zero}, we consider $q \in (0,\sqrt{2})$ satisfying \eqref{q_restr}, i.e.
\[
q< \frac{\log 2}{\sqrt{1-\log 2}} \approx 1.251.
\]
The last inequality is equivalent to $\beta_0 <1/q$ where $\beta_0 >q/2$ is such that $r_0 (\beta_0)= \sqrt{\log 2}$, i.e.
\[
1= \log 2 + \frac{\beta_0}{q}(\log 2)^2
\]
(note that $\beta_0 >q/2$ by \eqref{bijec}). Thus $g\circ r_0 \colon (q/2 , \beta_0) \to (0,\infty )$ is  a decreasing bijection, and so there exists a unique $\beta \in (q/2,\beta_0)$ such that
\[
g(r_0(\beta )) = \beta^2.
\]
Moreover, $\beta$ is unique also on $(q/2,1/q)$, as $g(r_0 (\beta )) \leq 0$ for $\beta \geq \beta_0$.  It is also clear that, for such $r_0,\beta$, we have that $b(r_0)$ is the global maximum of $b$ and $\Lambda (r_0)$ is the global minimum of $\Lambda$. Thus we obtain \eqref{choice_r0_beta_claim}, as required.\\

We note that the restriction \eqref{q_restr} is almost sharp in the sense that a unique choice of $r_0,\beta$ still exists for some $q>\log 2/\sqrt{1-\log 2}$, but not for $q\to 2^{1/2}$.  Indeed, considering $q>1.25$ we also have $\beta^2 >(q/2)^2 > 0.390625$ and so we could restrict ourselves to the `$r$'s for which $g(r)>0.390625$ (since we look for solutions to $g(r)=\beta^2$), i.e. to $r\in (0,c) \coloneqq g^{-1}((0.390625,\infty))\subset g^{-1} ((0,\infty)) = (0,(\log 2)^{1/2})$. Thus, redefining $\beta_0$ to be such that $r_0 (\beta_0 )=c$, we will have that $\beta <1/q$ if and only if $q$ is less than a number bigger than $\log 2/\sqrt{1-\log 2}$. 
However, we point out that it is impossible to extend the range of $q$ to the entire interval $(0,\sqrt{2})$, since
\[
 \sup r_0 ((q/2,1/q)) \to 0^+   \quad \text{ as } q\to \sqrt{2}^-
\]
and consequently
\[
\inf (g\circ r_0 ) ((q/2,1/q)) \to \infty \quad \text{ as } q\to \sqrt{2}^-.
\]
On the other hand, $\beta^2 \leq 1/q^2 \leq 1/2$, so, for sufficiently small $|2^{1/2} -q|$, there is no $r_0(\beta )$ for which $g\circ r_0 (\beta ) =\beta^2$. In other words, there is no solution $r,\beta >0$ to both \eqref{lambda_der} (i.e. $r=r_0(\beta )$) and \eqref{find_zero}.



\subsection*{Acknowledgments}

DA was partially supported by NSF Postdoctoral Fellowship Grant No.\ 2002023, NSF Grant No. 2406947, and the Office of the Vice Chancellor for Research and Graduate Education at the University of Wisconsin–Madison with funding from the Wisconsin Alumni Research Foundation. WO was partially supported by the NSF Grant no.~DMS-2511556 and by the Simons grant SFI-MPS-TSM-00014233. The authors are grateful to Chongchun Zeng, Zhiwu Lin and Susan Friedlander for helpful discussions, as well as to Mustafa Aydin for many constructive comments.

\bibliographystyle{plain}
\bibliography{literature}

\end{document}